%% file: Relative_h-principles_for_closed_stable_forms.tex
\documentclass[10pt,reqno]{amsart}
\pdfoutput=1
\input{Festino_1pt0}

\title[Relative \(\lowercase{h}\)-principles for closed stable forms]{\boldmath Relative \(\lowercase{h}\)-principles for closed stable forms}
\author{Laurence H. Mayther}

\begin{document}\fontsize{10pt}{12pt}\selectfont
\begin{abstract}
\footnotesize{This paper uses convex integration to develop a new, general method for proving relative \(h\)-principles for closed, stable, exterior forms on manifolds.  This method is applied to prove the relative \(h\)-principle for 4 classes of closed stable forms which were previously not known to satisfy the \(h\)-principle, {\it viz.}\ stable \((2k-2)\)-forms in \(2k\) dimensions, stable \((2k-1)\)-forms in \(2k+1\) dimensions, \sg\ 3-forms and \sg\ 4-forms.  The method is also used to produce new, unified proofs of all three previously established \(h\)-principles for closed, stable forms, {\it viz.}\ the \(h\)-principles for closed stable 2-forms in \(2k+1\) dimensions, closed \g\ 4-forms and closed \slc\ 3-forms.  In addition, it is shown that if a class of closed stable forms satisfies the relative \(h\)-principle, then the corresponding Hitchin functional (whenever defined) is necessarily unbounded above.

Due to the general nature of the \(h\)-principles considered in this paper, the application of convex integration requires an analogue of Hodge decomposition on arbitrary \(n\)-manifolds (possibly non-compact, or with boundary) which cannot, to the author's knowledge, be found elsewhere in the literature.  Such a decomposition is proven in \aref{NCH-Sec}.}
\end{abstract}
\maketitle

\section{Introduction}

Let \(\si_0 \in \ww{p}\lt(\bb{R}^n\rt)^*\) and let \(\M\) be an oriented \(n\)-manifold.  A \(\si_0\)-form on \(\M\) shall mean a \(p\)-form \(\si \in \Om^p(\M)\) such that, for each \(x \in \M\), there exists an orientation-preserving isomorphism \(\al: \T_x\M \to \bb{R}^n\) satisfying \(\al^*\si_0 = \si\).  Write \(\ww[\si_0]{p}\T^*\M\) for the bundle of \(\si_0\)-forms over \(\M\) and \(\Om^p_{\si_0}\) for the corresponding sheaf of sections.  Following Hitchin \cite{SF&SM}, say that \(\si_0\) is stable if its \(\GL_+(n;\bb{R})\)-orbit in \(\ww{p}\lt(\bb{R}^n\rt)^*\) is open.

Now let \(A \pc \M\) be a possibly empty submanifold of \(\M\) (or, more generally, a polyhedron; see \sref{diff-rels}).  Following \cite{PDR}, write \(\Op(A)\) for an arbitrarily small but unspecified open neighbourhood of \(A\) in \(\M\), which may be shrunk whenever necessary.  Let \(D^q\) denote the \(q\)-dimensional disc (\(q \ge 0\)), let \(\al: D^q \to \dR{p}(\M)\) be a continuous map and let \(\fr{F}_0: D^q \to \Om^p_{\si_0}(\M)\) be a continuous map such that:
\begin{enumerate}
\item For all \(s \in \del D^q\): \(\dd\fr{F}_0(s) = 0\) and \([\fr{F}_0(s)] = \al(s) \in \dR{p}(\M)\);
\item For all \(s \in D^q\): \(\dd\lt(\fr{F}_0(s)|_{\Op{A}}\rt) = 0\) and \(\lt[\fr{F}_0(s)|_{\Op(A)}\rt] = \al(s)|_{\Op(A)} \in \dR{p}(\Op(A))\).
\end{enumerate}

By combining the \(h\)-principles defined in \cite[\S6.2.C]{ItthP} and \cite[Thm.\ 5.3]{NIoG2S}, \(\si_0\)-forms shall be said to satisfy the relative \(h\)-principle if for every \(\M\), \(A\), \(q\), \(\al\) and \(\fr{F}_0\) as above, there exists a homotopy \(\fr{F}_\pt: [0,1] \x D^q \to \Om^p_{\si_0}(\M)\), constant over \(\del D^q\), satisfying:
\begin{enumerate}
\setcounter{enumi}{2}
\item For all \(s \in D^q\) and \(t \in [0,1]\): \(\fr{F}_t(s)|_{\Op(A)} = \fr{F}_0(s)|_{\Op(A)}\);
\item For all \(s \in D^q\): \(\dd\fr{F}_1(s) = 0\) and \([\fr{F}_1(s)] = \al(s) \in \dR{p}(\M)\).
\end{enumerate}

The following four choices of \(\si_0\) are of primary interest in this paper:
\ew
\si_0 &= \svph_0 = \th^{123} - \th^{145} - \th^{167} + \th^{246} - \th^{257} - \th^{347} - \th^{356} \in\ww{3}\lt(\bb{R}^7\rt)^*;\\
\si_0 &= \svps_0 = \th^{4567} - \th^{2367} - \th^{2345} + \th^{1357} - \th^{1346} - \th^{1256} - \th^{1247} \in\ww{4}\lt(\bb{R}^7\rt)^*;\\
\si_0 &= \vpi_+(k) = \sum_{i=1}^k \th^{12...\h{2i-1,2i}...2k-1,2k} \in \ww{2k-2}\lt(\bb{R}^{2k}\rt)^*, \hs{3mm} (k \ge 3);\\
\si_0 &= \xi_0(k) = \th^1 \w \sum_{i=1}^k \th^{23...\h{2i,2i+1}...2k,2k+1} \in \ww{2k-1}\lt(\bb{R}^{2k+1}\rt)^*, \hs{3mm} (k \ge 2),
\eew
where \(\lt(\th^1,...,\th^n\rt)\) denotes the canonical basis of \(\lt(\bb{R}^n\rt)^*\), \(\h{~}\) denotes that the corresponding indices should be omitted and multi-index notation \(\th^{ij...k} = \th^i \w \th^j \w ... \w \th^k\) is used throughout this paper.  \(\svph_0\)-forms and \(\svps_0\)-forms are termed \sg\ 3- and 4-forms respectively and are equivalent to \sg-structures on \(\M\), i.e.\ principal \sg-subbundles of the frame bundle of \(\M\).  \(\vpi_+(k)\)-forms are here termed ossymplectic forms (or, more precisely, osemproplectic forms; see \sref{OS-P-OP} for the motivation for these names) and are equivalent to \(\Sp(2k;\bb{R})\)-structures (i.e.\ almost symplectic structures) on \(\M\).  \(\xi_0(k)\)-forms are here termed ospseudoplectic forms and define co-oriented almost contact structures on \(\M\).  The main result of this paper is:

\begin{Thm}[\trefs{os}, \ref{op}, \ref{sg-4} and \ref{sg-3}]\label{4hP}
\sg\ 3-forms, \sg\ 4-forms, ossymplectic forms and ospseudoplectic forms satisfy the relative \(h\)-principle.
\end{Thm}

\tref{4hP} has the following significant corollary.  Given \(\si_0 \in \ww{p}\lt(\bb{R}^n\rt)^*\), an oriented \(n\)-manifold \(\M\) and a fixed cohomology class \(\al \in \dR{p}(\M)\), write \(\CL^p_{\si_0}(\M)\) for the set of closed \(\si_0\)-forms on \(\M\) and \(\CL^p_{\si_0}(\al)\) for the set of closed \(\si_0\)-forms representing the cohomology class \(\al\).  More generally, given a submanifold (or polyhedron) \(A \pc \M\), let \(\si_r\) be a closed \(\si_0\)-form on \(\Op(A)\) such that \([\si_r] = \al|_{\Op(A)} \in \dR{p}(\Op(A))\) and write:
\caw
\Om^p_{\si_0}(\M;\si_r) = \lt\{ \si \in \Om^p_{\si_0}(\M) ~\m|~ \si|_{\Op(A)} = \si_r\rt\}\\
\CL^p_{\si_0}(\M;\si_r) = \lt\{ \si \in \Om^p_{\si_0}(\M;\si_r) ~\m|~ \dd\si = 0 \rt\}\\
\CL^p_{\si_0}(\al;\si_r) = \lt\{ \si \in \CL^p_{\si_0}(\M;\si_r) ~\m|~ [\si] = \al \in \dR{p}(\M)\rt\}.
\caaw

\noindent By combining \tref{4hP} with standard homotopy-theoretic arguments (see \cite[\S6.2.A]{ItthP}), one obtains:

\begin{Thm}
Let \(\si_0\) be one of \(\svph_0\), \(\svps_0\), \(\vpi_+(k)\) and \(\xi_0(k)\), as above.  Then, for every \(\M\), \(A\), \(\al\) and \(\si_r\), the inclusions:
\ew
\CL^p_{\si_0}(\al;\si_r) \emb \CL^p_{\si_0}(\M;\si_r) \emb \Om^p_{\si_0}(\M; \si_r)
\eew
are homotopy equivalences (where \(p = 3,4,2k-2, 2k-1\), as appropriate).  In particular, taking \(A = \es\), the inclusions:
\ew
\CL^p_{\si_0}(\al) \emb \CL^p_{\si_0}(\M) \emb \Om^p_{\si_0}(\M)
\eew
are, also, homotopy equivalences.  Thus, if \(\M\) admits any \sg\ 3-form, then every degree \(p\) cohomology class on \(\M\) can be represented by a \sg\ 3-form and likewise for \sg\ 4-forms, ossymplectic forms and ospseudoplectic forms.
\end{Thm}

\tref{4hP} also has notable implications for Hitchin functionals on stable forms.  Recall that, given a stable form \(\si_0 \in \ww{p}\lt(\bb{R}^n\rt)^*\), if \(\Stab_{\GL_+(n;\bb{R})}(\si_0) \cc \SL(n;\bb{R})\), then any \(\si_0\)-form \(\si\) on an oriented manifold \(\M\) induces a volume form \(vol_\si\) on \(\M\).  In the case where \(\M\) is closed, for each degree \(p\) cohomology class \(\al\) one may define a Hitchin functional by integrating this volume form over all of \(\M\):
\ew
\bcd[row sep = 0pt]
\mc{H}: \CL^p_{\si_0}(\al)  \ar[r] & (0,\infty)\\
\si \ar[r, maps to] & \bigint_\M vol_\si.
\ecd
\eew

\begin{Thm}\label{HF-thm}
Let \(\si_0 \in \ww{p}\lt(\bb{R}^n\rt)^*\) be a stable form such that \(\Stab_{\GL_+(n;\bb{R})}(\si_0) \cc \SL(n;\bb{R})\) and suppose that \(\si_0\)-forms satisfy the relative \(h\)-principle.  Then, for any closed, oriented \(n\)-manifold admitting \(\si_0\)-forms and each \(\al \in \dR{p}(\M)\), the functional:
\ew
\mc{H}: \CL^p_{\si_0}(\al) \to (0,\infty)
\eew
is unbounded above.  In particular, the Hitchin functionals on ossymplecitc forms, \sg\ 3-forms and \sg\ 4-forms are always unbounded above.  (Note that \(\Stab_{\GL_+(2k+1;\bb{R})}(\xi_0(k)) \nsubseteq \SL(2k+1;\bb{R})\) and thus one cannot define Hitchin functionals on ospseudoplectic forms; see \sref{Pseudo-P}.)
\end{Thm}
Likewise, the Hitchin functionals on \g\ 4-forms and \slc\ 3-forms are also unbounded above; see \sref{conseq} for details.  I remark that alternative proofs of the unboundedness above of Hitchin functionals on \sg\ 3- and 4-forms, and \g\ 4-forms were previously obtained by the author in \cite{UA&BotDHFoG2&SG2F}, using different methods.

The proofs of all four \(h\)-principles in \tref{4hP} make use of the following result, which provides a general method for establishing relative \(h\)-principles for closed, stable forms:

\begin{Thm}\label{hPThm-1}
Let \(\si_0\in\ww{p}\lt(\bb{R}^n\rt)^*\) be stable (\(p \ge 1\)).  Given an arbitrary \(p\) form \(\ta\) on \(\bb{R}^{n-1}\), define:
\ew
\mc{N}_{\si_0}(\ta) = \lt\{\nu \in \ww{p-1}\lt(\bb{R}^{n-1}\rt)^* ~\m|~ \th \w \nu + \ta \in \ww[\si_0]{p}\lt(\bb{R} \ds \bb{R}^{n-1}\rt)^* \rt\} \pc \ww{p-1}\lt(\bb{R}^{n-1}\rt)^*,
\eew
where \(\th\) is the standard annihilator of \(\bb{R}^{n-1} \pc \bb{R} \ds \bb{R}^{n-1}\).  Suppose that, for every \(\ta\), the set \(\mc{N}_{\si_0}(\ta)\) is ample in the sense of affine geometry, i.e.\ \(\mc{N}_{\si_0}(\ta)\) is either empty, or the convex hull of every path component of \(\mc{N}_{\si_0}(\ta)\) equals \(\ww{p-1}\lt(\bb{R}^{n-1}\rt)^*\) (in such cases, say that \(\si_0\) itself is ample).  Then, \(\si_0\)-forms satisfy the relative \(h\)-principle.
\end{Thm}

Significantly, the condition that \(\si_0\) be ample is entirely algebraic and may be verified by direct calculation.  Thus, \tref{hPThm-1} provides a way of reducing the more complicated question of whether a given class of closed stable forms satisfies the \(h\)-principle to a tractable, algebraic problem.  I remark also that, not only can \tref{hPThm-1} be used to prove the four new \(h\)-principles established in this paper, but it also subsumes all three previously known \(h\)-principles for stable forms, {\it viz.}\ the relative \(h\)-principles for stable 2-forms in \((2k+1)\) dimensions (\(k \ge 2\)), \g\ 4-forms and \(\SL(3;\bb{C})\)-forms established in \cite{AoCItS&CG, NIoG2S, RoG2MwB} respectively; see \sref{1st-Appl}.

Essential to the proof of \tref{hPThm-1} is the technique of convex integration, introduced by Gromov in \cite{CIoPDR} and later developed in \cite{PDR, CIT, ItthP}.  The proof also makes use of the following analogue of Hodge decomposition for non-compact manifolds, proved in \aref{NCH-Sec}, which the author believes to be of independent interest:
\begin{Thm}\label{NCH}
Let \(\M\) be an \(n\)-manifold (not necessarily oriented and possibly non-compact or with boundary).  Then, there exists an injective, continuous linear operator \(\io:\Ds_p\dR{p}(\M) \to \Ds_p\Om^p(\M)\) of degree 0 and a continuous linear operator \(\de:\Ds_p\Om^p(\M) \to \Ds_p\Om^p(\M)\) of degree \(-1\) satisfying \(\de^2 = 0\) such that for each \(0 \le p \le n\):
\e\label{NCHD-eq}
\Om^p(\M) = \io\dR{p}(\M) \ds \dd\Om^{p-1}(\M) \ds \de\Om^{p+1}(\M)
\ee
in the category of \F\ spaces, where both \(\dd\) and \(\de\) act as \(0\) on \(\io\dR{p}(\M)\) and the maps \(\dd: \de\Om^{p+1}(\M) \to \dd\Om^p(\M)\) and \(\de: \dd\Om^p(\M) \to \de\Om^{p+1}(\M)\) are mutually inverse.  In particular, the projections \(\Om^p(\M) \to \dd\Om^{p-1}(\M)\) and \(\Om^p(\M) \to \de\Om^{p+1}(\M)\) are given by \(\dd \circ \de\) and \(\de \circ \dd\), respectively.
\end{Thm}

The results of the present paper were obtained during the author's doctoral studies, which where supported by EPSRC Studentship 2261110.\\

\section{Preliminaries}

\subsection{Stable forms in 6 and 7 dimensions}\label{6&7}

Let \(\rh\in\ww{3}\lt(\bb{R}^6\rt)^*\) and consider the linear map \(K_\rh:\bb{R}^6\to \bb{R}^6\ts\ww{6}\lt(\bb{R}^6\rt)^*\) defined by composing \(u \in \bb{R}^6 \mt (u\hk\rh)\w\rh \in \ww{5}\lt(\bb{R}^6\rt)^*\) with the canonical isomorphism \(\ww{5}\lt(\bb{R}^6\rt)^* \cong \bb{R}^6\ts\ww{6}\lt(\bb{R}^6\rt)^*\).  \(K_\rh\) induces a \(\GL_+(6;\bb{R})\)-equivariant map \(\La: \rh \in \ww{3}\lt(\bb{R}^6\rt)^* \mt \frac{1}{6}\Tr\lt(K_\rh^2\rt) \in \lt(\ww{6}\lt(\bb{R}^6\rt)^*\rt)^{\ts2}\).  The following result is mainly proved in \cite[\S2]{TGo3Fi6&7D}, with the general formula for \(I_\rh\) and the explicit expressions for \(vol_\rh\), \(I_\rh\) and \(J_\rh\) when \(\rh = \rh_\pm\) following from \cite[Prop.\ 1.5]{HFS&SH} (although note that, for \(\La < 0\), I have used a different normalisation for \(vol_\rh\) and a different sign convention for \(J_\rh\)):

\begin{Prop}\label{SL-Prop}
The action of \(\GL_+(6;\bb{R})\) on \(\ww{3}\lt(\bb{R}^6\rt)^*\) has exactly two open orbits, namely:
\ew
\ww[+]{3}\lt(\bb{R}^6\rt)^* = \lt\{ \rh \in \ww{3}\lt(\bb{R}^6\rt)^* ~\m|~ \La(\rh) > 0 \rt\} \et \ww[-]{3}\lt(\bb{R}^6\rt)^* = \lt\{ \rh \in \ww{3}\lt(\bb{R}^6\rt)^* ~\m|~ \La(\rh) < 0 \rt\},
\eew
both of which are invariant under \(\GL(6;\bb{R})\).  Representatives of \(\ww[+]{3}\lt(\bb{R}^6\rt)^*\) and \(\ww[-]{3}\lt(\bb{R}^6\rt)^*\) may be taken to be:
\e\label{rh+}
\rh_+ = \th^{123} + \th^{456} \et \rh_- =& \th^{135} - \th^{146} - \th^{236} - \th^{245} = \fr{Re}\Big(\lt(\th^1 + i\th^2\rt)\w\lt(\th^3 + i\th^4\rt)\w\lt(\th^5 + i\th^6\rt)\Big),
\ee
respectively.  Each \(\rh \in \ww[+]{3}\lt(\bb{R}^6\rt)^*\) induces a volume form \(vol_\rh = \lt(\La(\rh)\rt)^\frac{1}{2}\) and para-complex structure \(I_\rh = vol_\rh^{-1}K_\rh\) on \(\bb{R}^6\) (i.e.\ \(I_\rh\) is an automorphism of \(\bb{R}^6\) such that \(I_\rh^2 = \Id\), with \(+1\) and \(-1\) eigenspaces \(E_{\pm,\rh}\) each having dimension 3), and \(\Stab_{\GL_+(6;\bb{R})}(\rh) \cong \SL(3;\bb{R})^2\) acting diagonally on \(\bb{R}^6 \cong E_{+,\rh} \ds E_{-,\rh}\).  Explicitly for \(\rh = \rh_+\):
\caw
vol_{\rh_+} = \th^{123456}; \hs{5mm} I_{\rh_+} = (e_1,e_2,e_3,e_4,e_5,e_6) \mt (e_1,e_2,e_3,-e_4,-e_5,-e_6);\\
E_{+,\rh_+} = \<e_1,e_2,e_3\? \et E_{-,\rh_+} = \<e_4,e_5,e_6\?.
\caaw
By contrast, each \(\rh \in \ww[-]{3}\lt(\bb{R}^6\rt)^*\) induces a volume form \(vol_\rh = \frac{1}{2}\lt(-\La(\rh)\rt)^\frac{1}{2}\) and a complex structure \(J_\rh = -\frac{1}{2}vol_\rh^{-1}K_\rh\) on \(\bb{R}^6\) , and \(\Stab_{\GL_+(6;\bb{R})}(\rh) \cong \SL(3;\bb{C})\).  Explicitly for \(\rh = \rh_-\):
\ew
vol_{\rh_-} = \th^{123456} \et J_{\rh_-} = (e_1,e_2,e_3,e_4,e_5,e_6) \mt (e_2,-e_1,e_4,-e_3,e_6,-e_5).
\eew
\end{Prop}

Now, turn to 7 dimensions.  Given \(\ph\in\ww{3}\lt(\bb{R}^7\rt)^*\), define a quadratic form \(Q_\ph\) on \(\bb{R}^7\) valued in \(\ww{7}\lt(\bb{R}^7\rt)^*\) by \(Q_\ph(v) = \frac{1}{6}\lt(v\hk\ph\rt)^2\w\ph\in\ww{7}\lt(\bb{R}^7\rt)^*\).  The determinant of \(Q_\ph\) is a polynomial in \(\ph\) and thus \(\lt\{ \ph ~\m|~ Q_\ph \text{ is degenerate}\rt\}\) is an affine subvariety of \(\ww{3}\lt(\bb{R}^7\rt)^*\) of positive codimension; hence \(Q_\ph\) must be non-degenerate whenever \(\ph\) is stable.  The following proposition is readily deduced from the results of \cite{A3F&ESLGoTG2}:

\begin{Prop}\label{Stable-in-7}
The action of \(\GL_+(7;\bb{R})\) on \(\ww{3}\lt(\bb{R}^7\rt)^*\) has precisely 4 open orbits, corresponding to \(Q\) having signature \((7,0)\), \((3,4)\), \((4,3)\) and \((0,7)\).  Write \(\ww[+]{3}\lt(\bb{R}^7\rt)^* = \lt\{ \ph ~\m|~ Q_\ph \text{ has signature } (7,0) \rt\}\).  For \(\ph \in \ww[+]{3}\lt(\bb{R}^7\rt)^*\), \(\Stab_{\GL_+(7;\bb{R})}(\ph) \cong \Gg_2\); accordingly, \(\ph\) is termed a \g\ 3-form.  It induces a unique positive definite inner product \(g_\ph\) and (positive) volume form \(vol_\ph\) such that \(Q_\ph = g_\ph \ts vol_\ph\) and \(\|\ph\|_{g_{\ph}}^2 = 7\).  A representative of \(\ww[+]{3}\lt(\bb{R}^7\rt)^*\) may be taken to be:
\ew
\vph_0 = \th^{123} + \th^{145} + \th^{167} + \th^{246} - \th^{257} - \th^{347} - \th^{356} \in\ww{3}(\bb{R}^7)^*.
\eew
Then, \(g_{\vph_0} = \sum_{i=1}^7 \lt(\th^i\rt)^{\ts2}\) and \(vol_{\vph_0} = \th^{12...7}\).  Likewise, write \(\ww[\tl]{3}\lt(\bb{R}^7\rt)^* = \lt\{ \sph ~\m|~ Q_{\sph} \text{ has signature } (3,4) \rt\}\).  For \(\sph \in \ww[\tl]{3}\lt(\bb{R}^7\rt)^*\), \(\Stab_{\GL_+(7;\bb{R})}(\sph) \cong \tld{\Gg}_2\); accordingly, \(\sph\) is termed a \sg\ 3-form.  It induces a unique inner product \(g_{\sph}\) of signature \((3,4)\) and (positive) volume form \(vol_{\sph}\) such that \(Q_{\sph} = g_{\sph} \ts vol_{\sph}\) and \(\lt\|\sph\rt\|_{g_{\sph}}^2 = 7\).  A representative of \(\ww[\tl]{3}\lt(\bb{R}^7\rt)^*\) may be taken to be:
\e\label{svph0}
\svph_0 = \th^{123} - \th^{145} - \th^{167} + \th^{246} - \th^{257} - \th^{347} - \th^{356} \in\ww{3}(\bb{R}^7)^*.
\ee
Then, \(g_{\svph_0} = \sum_{i=1}^3 \lt(\th^i\rt)^{\ts2} - \sum_{i=4}^7 \lt(\th^i\rt)^{\ts2}\) and \(vol_{\svph_0} = \th^{12...7}\).
\end{Prop}

Clearly \(\ww[+]{3}\lt(\bb{R}^7\rt)\) and \(\lt\{ \ph ~\m|~ Q_\ph \text{ has signature } (0,7) \rt\}\) are exchanged by any orientation-reversing automorphism of \(\bb{R}^7\); in particular (since \(-\Id\) is orientation-reversing on an odd-dimensional vector space) \(\lt\{ \ph ~\m|~ Q_\ph \text{ has signature } (0,7) \rt\} = -\ww[+]{3}\lt(\bb{R}^7\rt)^*\).  Similar conclusions hold for the split-signature case.

Using \tref{Stable-in-7}, the following result is easily deduced (cf.\ \cite[p.\ 541]{MwEH}):

\begin{Prop}
Define \(\ww[+]{4}\lt(\bb{R}^7\rt)^* = \lt\{ \Hs_\ph \ph ~\m|~ \ph \in \ww[+]{3}\lt(\bb{R}^7\rt)^* \rt\}\) and define \(\ww[\tl]{4}\lt(\bb{R}^7\rt)^*\) analogously.  Then the action of \(\GL_+(7;\bb{R})\) on \(\ww[\tl]{4}\lt(\bb{R}^7\rt)^*\) has precisely 4 open orbits, given by \(\ww[+]{4}\lt(\bb{R}^7\rt)^*\), \(\ww[\tl]{4}\lt(\bb{R}^7\rt)^*\), \(-\ww[\tl]{4}\lt(\bb{R}^7\rt)^*\) and \(-\ww[+]{4}\lt(\bb{R}^7\rt)^*\), each of which are also orbits of \(\GL(7;\bb{R})\).  Forms in the first two orbits are termed \g\ 4-forms and \sg\ 4-forms, respectively.  Representatives of these orbits may be taken to be:
\cag\label{vps0}
\vps_0 = \Hs_0\vph_0 = \th^{4567} + \th^{2367} + \th^{2345} + \th^{1357} - \th^{1346} - \th^{1256} - \th^{1247};\\
\svps_0 = \tld{\Hs}_0\svph_0 = \th^{4567} - \th^{2367} - \th^{2345} + \th^{1357} - \th^{1346} - \th^{1256} - \th^{1247}.\\
\caag
\end{Prop}

I end this subsection with a general definition.  The term `Hitchin form' shall refer to a stable form \(\si_0 \in \ww{p}\lt(\bb{R}^n\rt)^*\) such that \(\Stab_{\GL_+(n;\bb{R})}(\si_0) \cc \SL(n;\bb{R})\).  Given an orbit \(\mc{O} \pc \ww{p}\lt(\bb{R}^n\rt)^*\) of Hitchin forms, there is a \(\GL_+(n;\bb{R})\)-equivariant volume map \(vol:\mc{O} \to \ww{n}\lt(\bb{R}^n\rt)^*\), unique up to an overall positive constant multiple.  Since \(\mc{O}\cc \ww{p}\lt(\bb{R}^n\rt)^*\) is open, the derivative of \(vol\) at \(\si \in \mc{O}\) is a linear map \(\ww{p}\lt(\bb{R}^n\rt)^* \to \ww{n}\lt(\bb{R}^n\rt)^*\) and hence there exists \(\Xi(\si) \in \ww{n-p}\lt(\bb{R}^n\rt)^*\) such that \(\mc{D}vol|_\si(\al) = \al \w \Xi(\si)\) for all \(\al \in \ww{p}\lt(\bb{R}^n\rt)^*\).  \(\Xi\) defines a \(\GL_+(n;\bb{R})\)-equivariant map from \(\mc{O}\) to an open orbit in \(\ww{n-p}\lt(\bb{R}^n\rt)^*\), unique up to a constant positive multiple.

Hitchin forms are of particular interest when considering stable forms on manifolds.  Let \(\M\) be a closed, oriented \(n\)-manifold and let \(\si \in \Om^p(\M)\) be a \(\si_0\)-form for some Hitchin form \(\si_0 \in \ww{p}\lt(\bb{R}^n\rt)^*\).  If \(\dd\si = 0\), one can define a Hitchin functional \cite{SF&SM} \(\mc{H}: \CL^p_{\si_0}([\si]) \to (0,\infty)\) by:
\ew
\mc{H}(\si') = \bigintsss_\M vol_{\si'}.
\eew
The functional derivative of \(\mc{H}\) is simply:
\ew
\bcd[row sep = 0pt]
\mc{DH}|_{\si'}:\dd\Om^{p-1}(\M) \ar[r] & \bb{R}\\
\dd\ga \ar[r, maps to] & \bigint_\M \dd\ga \w \Xi(\si')
\ecd
\eew
and thus \(\si'\) is a critical point of the functional \(\mc{H}\) \iff\ \(\dd\Xi(\si') = 0\).  Such forms shall be termed biclosed (since both \(\si\) and \(\Xi(\si)\) are closed) and are often of geometric significance: e.g.\ biclosed \g/\sg\ 3-/4-forms correspond to metrics with holonomy contained in \g\ (resp.\ \sg) \cite{TGo3Fi6&7D, MwEH} and biclosed \slc\ 3-forms define (integrable) complex structures with trivial canonical bundle \cite[Thm.\ 12]{TGo3Fi6&7D}.\\

\subsection{Differential relations}\label{diff-rels}

For a more expansive introduction to material in this subsection, see \cite[Chs.\ 1, 5, 6, 17 \&\ 18]{ItthP}.

Let \(\pi: E \to \M\) be a vector bundle.  Write \(E^{(1)}\) for the first jet bundle of \(E\); explicitly, given a connection \(\nabla\) on \(E\), by \cite[\S9, Cor.\ to Thm.\ 7]{DOoVB} one can identify \(E^{(1)} \cong E \ds \lt(\T^*\M \ts E\rt)\) such that the following diagram commutes:
\ew
\bcd[row sep = 3mm]
\Ga\lt(\M,E^{(1)}\rt) \ar[rr, "\mLa{\cong}"] & & \Ga\lt(\M,E \ds \lt(\T^*\M \ts E\rt)\rt)\\
& \Ga(\M,E) \ar[ru, "\mLa{s \mt s \ds \nabla s}"'] \ar[lu, "\mLa{j_1}"] &
\ecd
\eew
where \(\Ga(\M,-)\) denotes the space of sections over \(\M\) and \(j_1\) is the map assigning to a section of \(E\) its corresponding \(1\)-jet; write \(p_1: E^{(1)} \to E\) for the natural projection.  In particular, note that \(E^{(1)}\) naturally has the structure of a vector bundle over \(\M\).  More generally, given \(q \ge 0\), write \(E_{D^q}\) for the pullback of the vector bundle \(E\) along the projection \(D^q \x \M \to \M\); explicitly, \(E_{D^q}\) is the vector bundle \(D^q \x E \oto{\Id \x \pi} D^q \x \M\).  In this paper, a section of \(E_{D^q}\) shall refer to a continuous map \(s: D^q \x \M \to D^q \x E\) satisfying \(\pi_{E_{D^q}} \circ s = \Id_{D^q \x \M}\), and depending smoothly on \(x \in \M\); in particular, sections of \(E_{D^q}\) over \(D^q \x \M\) correspond to continuous maps \(D^q \to \Ga(E,\M)\).  Write \(E^{(1)}_{D^q}\) for the vector bundle \(\lt(E^{(1)}\rt)_{D^q}\) and note that \(E^{(1)}_{D^q} \ne \lt(E_{D^q}\rt)^{(1)}\), since only derivatives in the `\(\M\)-direction' are considered in the bundle \(E^{(1)}_{D^q}\).  A section of \(E^{(1)}_{D^q}\) is termed holonomic if it is the \(1\)-jet of a section of \(E_{D^q}\), i.e.\ if it can be written as \(s \ds \nabla s\) for some section \(s\) of \(E_{D^q}\).  A fibred differential relation (of order 1) on \(D^q\)-indexed families of sections of \(E\) is simply a subset \(\cal{R} \pc E^{(1)}_{D^q}\).  \(\cal{R}\) is termed an open relation if it is open as a subset of \(E^{(1)}_{D^q}\).

A (possibly disconnected) subset \(A\cc\M\) is termed a polyhedron if there exists a smooth triangulation \(\cal{K}\) of \(\M\) identifying \(A\) with a subcomplex of \(\cal{K}\); in particular, \(A\) is a closed subset of \(\M\) (see \sref{App-Prelim} for the definition of smooth triangulation).  Say that a fibred relation \(\cal{R}\) satisfies the relative \(h\)-principle if for every polyhedron \(A\) and every section \(F_0\) of \(\cal{R}\) over \(D^q \x \M\) which is holonomic over \([\del D^q \x \M] \cup [D^q \x \Op(A)]\), there exists a homotopy \((F_t)_{t \in [0,1]}\) of sections of \(\cal{R}\), constant over \([\del D^q \x \M] \cup [D^q \x \Op(A)]\), such that \(F_1\) is a holonomic section of \(\cal{R}\).  Say that \(\cal{R}\) satisfies the \(C^0\)-dense relative \(h\)-principle if, in addition, the induced homotopy \(p_1(F_t)\) of sections of \(E\) can be taken to be arbitrarily small in the \(C^0\)-topology.  (The reader will note the similarity between this definition and the notion of \(h\)-principles for stable forms introduced in the introduction.)

Now fix a point \(p\in\M\).  Identifying \(E^{(1)} \cong E \ds \T^*\M\ts E\), the fibre of the map \(p_1:E^{(1)}\to E\) over \(e \in E_p\) is isomorphic to the space \(p_1^{-1}(e) = \{e\} \x \T^*_p\M \ts E_p = \{e\} \x \Hom(\T_p\M,E_p)\).  Each codimension-1 hyperplane \(\bb{B}\pc\T_p\M\) and linear map \(\la:\bb{B}\to E_p\) thus define a so-called principal subspace of \(p_1^{-1}(e)\) by:
\e\label{PiBl-defn}
\Pi_e(\bb{B},\la) &= \{e\} \x \lt\{L\in\Hom(\T_p\M,E_p)~\middle|~L|_\bb{B} = \la\rt\}\\
&= \{e\} \x \Pi(\bb{B},\la).
\ee
\(\Pi_e(\bb{B},\la)\) is an affine subspace of \(p_1^{-1}(e)\) modelled on \(E_p\), though not, in general, a linear subspace.  (Note, also, that changing the choice of connection \(\nabla\) on \(E\) changes the identification \(p_1^{-1}(e) = \{e\} \x \T^*_p\M \ts E_p\) by an affine linear map and so the collection of principal subspaces of \(p_1^{-1}(e)\) is independent of the choice of connection.)

\begin{Thm}[{Convex Integration Theorem, \cite[Chs.\ 17--18]{ItthP}}]\label{EGM-Thm}
Let \(E\to\M\) be a vector bundle, let \(q \ge 0\) and let \(\cal{R}\cc E^{(1)}_{D^q}\) be an open fibred differential relation.  For each \(s \in D^q\), define \(\cal{R}_s \pc E^{(1)}\) by the formula:
\ew
\{s\} \x \cal{R}_s = \cal{R} \cap \lt( \{s\} \x E^{(1)}\rt).
\eew
Suppose that, for every \(s \in D^q\), \(e\in E\) and principal subspace \(\Pi_e \pc p_1^{-1}(e)\), the subset \(\cal{R}_s\cap \Pi_e \cc \Pi_e\) is ample in the sense of affine geometry, i.e.\ \(\cal{R}_s\cap \Pi_e\) is either empty, or the convex hull of every path component of \(\cal{R}_s\cap \Pi_e\) equals \(\Pi_e\) (in this case, say that \(\cal{R}\) is ample).  Then, \(\cal{R}\) satisfies the \(C^0\)-dense, relative \(h\)-principle.
\end{Thm}

\begin{Rk}\label{calc-with-amp}
The relations considered in this paper will all be of the form:
\ew
\cal{R} = E_{D^q} \x_{(D^q \x \M)} \cal{R}' \cc E_{D^q} \ds \lt(\T^*\M \ts E\rt)_{D^q}
\eew
for some subbundle \(\cal{R}' \cc \lt(\T^*\M \ts E\rt)_{D^q}\), where \(\x_{(D^q \x \M)}\) denotes the fibrewise Cartesian product of bundles over \(D^q \x \M\).  In this case, define \(\cal{R}'_s\) by the relation:
\ew
\{s\} \x \cal{R}'_s = \cal{R}' \cap \lt[\{s\} \x \lt(\T^*\M \ts E\rt)\rt] \pc \lt(\T^*\M \ts E\rt)_{D^q}.
\eew
Then, for all \(s \in D^q\), \(e \in E\) and principal subspaces \(\Pi_e(\bb{B},\la) \pc \pi^{-1}(e)\):
\ew
\cal{R}_s \cap \Pi_e(\bb{B},\la) = \{s\} \x \{e\} \x \lt(\cal{R}'_s \cap \Pi(\bb{B},\la)\rt)
\eew
for \(\Pi(\bb{B},\la)\) as defined in \eref{PiBl-defn}, and thus \(\cal{R}\) is ample \iff\ \(\cal{R}'_s \cap \Pi(\bb{B},\la) \pc \Pi(\bb{B},\la)\) is ample for all \(\bb{B}\) and \(\la\).  Thus, the underlying point \(e\) is irrelevant for relations of this form.\\
\end{Rk}

\section{Ossymplectic, pseudoplectic and ospseudoplectic forms}\label{OS-P-OP}

The aim of this section is to establish the fundamental properties of stable \(2\)-forms and \((n-2)\)-forms which will be needed in this paper.  I begin with some general results on stable forms.  Fix a volume form \(vol\) on \(\bb{R}^n\); the following result is easily verified by direct calculation:

\begin{Lem}\label{no-orb-lem}
Consider the anti-isomorphism:
\e\label{Ph-defn}
\bcd[row sep = 0pt]
\Ph: \GL_+(n;\bb{R}) \ar[r]& \GL_+(n;\bb{R})\\
F \ar[r, maps to]& \det(F)^\frac{1}{n-p} \cdot F^{-1},
\ecd
\ee
with inverse given by \(\Ps: F \mt \det(F)^\frac{1}{p} \cdot F^{-1}\).  Then, the following \(\GL_+(n;\bb{R})\)-equivariant diagram commutes:
\e\label{almost-equi}
\bcd[column sep = 1.4cm]
\hs{4.5mm} \ww{p}\bb{R}^n \hs{4.5mm} \ar[d, " \si \mt \si \hk vol "] \ar[loop right] & \GL_+(n;\bb{R}) \ar[d, " \Ph "]\\
\ww{n-p}\lt(\bb{R}^n\rt)^* \ar[loop right]& \GL_+(n;\bb{R})\\
\ecd
\ee
where the left-hand vertical arrow is an isomorphism, the top action is a left action and the bottom action is a right action.
\end{Lem}
~\qed

Next, given a \(p\)-form \(\si_0\) on \(\bb{R}^n\), define linear maps \(\io_{\si_0}\) and \(\ep_{\si_0}\) by:
\e\label{io-ep}
\bcd[row sep = 0pt]
\io_{\si_0}:\bb{R}^n \ar[r] & \ww{p-1}\lt(\bb{R}^n\rt)^* & \ep_{\si_0}:\lt(\bb{R}^n\rt)^* \ar[r] & \ww{p+1}\lt(\bb{R}^n\rt)^*\\
u \ar[r, maps to] & u \hk \si_0 & \be \ar[r, maps to]& \be \w \si_0\hs{1.5pt}.
\ecd
\ee
Recall from \cite[\S2]{MASF} that a \(p\)-form \(\si_0\) on \(\bb{R}^n\) is termed multi-symplectic if \(\io_{\si_0}\) is injective.\footnote{{\it Caveat}: not every stable form is multi-symplectic, despite the claim to the contrary in \cite[Cor.\ 2.3]{MASF}.  A counterexample is provided by pseudoplectic forms; see \pref{P}.}  Analogously, I term \(\si_0\) multi-ossymplectic if \(\ep_{\si_0}\) is injective.  Write \(\GL_-(n;\bb{R})\) for the set of orientation-reversing automorphisms of \(\bb{R}^n\) and \(\Stab_{\GL_-(n;\bb{R})}(\si_0)\) for the set of orientation-reversing automorphisms fixing a given \(p\)-form \(\si_0\).

\begin{Lem}\label{stabiliser-comp-lem}~\\
(1) Let \(\al\in\ww{p}\lt(\bb{R}^n\rt)^*\) be a multi-symplectic form, embed \(\bb{R}^n \emb \bb{R}^k \ds \bb{R}^n \cong \bb{R}^{n+k}\) for some \(k>0\) and embed \(\ww{p}\lt(\bb{R}^n\rt)^*\emb \ww{p}\lt(\bb{R}^{n+k}\rt)^*\) in the canonical way.  Then:
\ew
\Stab_{\GL_+(n+k;\bb{R})}(\al) &= \bigg\{ \bpm A_{k \x k} & B_{k \x n} \\ 0_{n \x k} & C_{n \x n} \epm ~\bigg|~ \raisebox{4pt}{\parbox{52mm}{\center{\(A \in \GL_+(k;\bb{R})\), \(B \in \End(\bb{R}^n,\bb{R}^k)\), \(C \in \Stab_{\GL_+(n;\bb{R})}(\al)\)}}} \bigg\}\\
& \hs{2cm} \coprod  \bigg\{ \bpm A_{k \x k} & B_{k \x n} \\ 0_{n \x k} & C_{n \x n} \epm ~\bigg|~ \raisebox{4pt}{\parbox{52mm}{\center{\(A \in \GL_-(k;\bb{R})\), \(B \in \End(\bb{R}^n,\bb{R}^k)\), \(C \in \Stab_{\GL_-(n;\bb{R})}(\al)\)}}} \bigg\}.
\eew

Thus, if \(\Stab_{\GL(n;\bb{R})}(\al)\) is connected, then \(\Stab_{\GL_+(n+k;\bb{R})}(\al)\) is, also, connected.

\noindent (2) Now, let \(\al\in\ww{p}\lt(\bb{R}^n\rt)^*\) be a multi-ossymplectic form. Then:
\e\label{mcs-stab}
\Stab_{\GL_+(n+k;\bb{R})}(\th^{12..k} \w \al) &= \bigg\{ \bpm A_{k \x k} & 0_{k \x n}\\ B_{n \x k} & \det(A)^{-\frac{1}{p}}C_{n \x n} \epm ~\bigg|~ \raisebox{4pt}{\parbox{52mm}{\center{\(A \in \GL_+(k;\bb{R})\), \(B \in \End(\bb{R}^k,\bb{R}^n)\), \(C \in \Stab_{\GL_+(n;\bb{R})}(\al)\)}}} \bigg\}\\
& \hs{2cm} \coprod \bigg\{ \bpm A_{k \x k} & 0_{k \x n}\\ B_{n \x k} & |\det(A)|^{-\frac{1}{p}}C_{n \x n} \epm ~\bigg|~ \raisebox{4pt}{\parbox{52mm}{\center{\(A \in \GL_-(k;\bb{R})\), \(B \in \End(\bb{R}^k,\bb{R}^n)\), \(C \in J\cdot \Stab_{\GL_{-\sgn{J}}(n;\bb{R})}(\al)\)}}} \bigg\},
\ee
where \(J \in \GL(n;\bb{R})\) is any map such that \(J^*\al = -\al\).  Thus, if either:
\begin{itemize}
\item \(\al\) and \(-\al\) lie in separate \(\GL(n;\bb{R})\)-orbits, or
\item \(\al\) and \(-\al\) lie in the same \(\GL_+(n;\bb{R})\)-orbit and \(\Stab_{\GL(n;\bb{R})}(\al) = \Stab_{\GL_+(n;\bb{R})}(\al)\),
\end{itemize}
then, the second set on the right-hand side of \eref{mcs-stab} is empty.  In particular, if either of these conditions holds and additionally \(\Stab_{\GL_+(n;\bb{R})}(\al)\) is connected, then \(\Stab_{\GL_+(n+k;\bb{R})}(\th^{12...k} \w \al)\) is, also, connected.
\end{Lem}

\begin{proof}
(1) Since \(\al\in\ww{p}\lt(\bb{R}^n\rt)^*\) is multi-symplectic, the kernel of the linear map:
\ew
\bcd[row sep = 0pt]
\bb{R}^{n+k} \ar[r]& \ww{p-1}\lt(\bb{R}^{n+k}\rt)^*\\
u \ar[r, maps to]& u \hk \al
\ecd
\eew
is precisely \(\bb{R}^k \ds 0\) and hence this subspace is invariant under \(\Stab_{\GL_+(n+k;\bb{R})}(\al)\).  Thus, any element of \(\Stab_{\GL_+(n+k;\bb{R})}(\al)\) has the form \(F = \bpm A_{k \x k} & B_{k \x n} \\ 0_{n \x k} & C_{n \x n} \epm\) for some \(A \in \GL(k;\bb{R})\), \(B \in \End(\bb{R}^n,\bb{R}^k)\) and \(C \in \GL(n;\bb{R})\).  Therefore, \(F^*\al = C^*\al = \al\) and the result follows.

(2) Since \(\al\in\ww{p}\lt(\bb{R}^n\rt)^*\) is multi-ossymplectic, the kernel of the linear map:
\ew
\bcd[row sep = 0pt]
\lt(\bb{R}^{n+k}\rt)^* \ar[r]& \ww{p+1}\lt(\bb{R}^{n+k}\rt)^*\\
\be \ar[r, maps to]& \be \w \lt(\th^{12...k} \w \al\rt)
\ecd
\eew
is precisely \(\lt(\bb{R}^k\rt)^* \ds 0\) and hence this subspace is invariant under \(\Stab_{\GL_+(n+k;\bb{R})}(\th^{12...k} \w \al)\).  Thus, any element of \(\Stab_{\GL_+(n+k;\bb{R})}(\th^{12...k} \w \al)\) has the form \(F = \bpm A_{k \x k} & 0_{k \x n}\\ B_{n \x k} & D_{n \x n} \epm\) for some \(A \in \GL(k;\bb{R})\), \(B \in \End(\bb{R}^k,\bb{R}^n)\) and \(D \in \GL(n;\bb{R})\), hence \(F^*\lt(\th^{12...k} \w \al\rt) = \det(A) \th^{12...k} \w D^*\al = \th^{12...k} \w \al\) and whence:
\e\label{A.D-eqn}
\det(A) \cdot D^*\al = \al.
\ee

If \(\det(A) > 0\), one can rewrite \eref{A.D-eqn} as \(\lt(\det(A)^\frac{1}{p} \cdot D\rt)^*\al = \al\).  Since \(\det(F) = \det(A)\det(D) > 0\), it follows that \(\det\lt(\det(A)^\frac{1}{p} \cdot D\rt) > 0\) and thus \(C = \det(A)^\frac{1}{p} \cdot D \in \Stab_{\GL_+(n;\bb{R})}(\al)\) as claimed.

Now suppose \(\det(A)<0\) and rewrite \eref{A.D-eqn} as \(\lt(|\det(A)|^\frac{1}{p} \cdot D\rt)^*\al = -\al\), where now \(C = |\det(A)|^\frac{1}{p} \cdot D\) has negative determinant.  Then clearly \(\al\) and \(-\al\) lie in the same \(\GL(n;\bb{R})\)-orbit; let \(J\) be some fixed element of \(\GL(n;\bb{R})\) (not necessarily equal to \(C\)) such that \(J^*\al = -\al\).  If \(\al\) and \(-\al\) lie in different \(\GL_+(n;\bb{R})\)-orbits, then \(J\) is automatically orientation-reversing and so \(C \in J \cdot \Stab_{\GL_{+}(n;\bb{R})}(\al)\) as required.  Else, \(J\) may be chosen to be orientation-preserving and hence \(C \cdot J\) is an orientation-reversing automorphism of \(\al\), implying that \(\Stab_{\GL(n;\bb{R})}(\al) \ne \Stab_{\GL_+(n;\bb{R})}(\al)\) and \(C \in J \cdot \Stab_{\GL_-(n;\bb{R})}(\al)\), again as required.  This completes the proof.

\end{proof}

\subsection{\(\mb{(2k-2)}\)-forms in \(\mb{2k}\) dimensions, \(\mb{k \ge 3}\)}

I begin by recalling the following well-known result:

\begin{Prop}\label{symp}
Let \(k \ge 2\).  The action of \(\GL_+(2k;\bb{R})\) on \(\ww{2}\lt(\bb{R}^{2k}\rt)^*\) has exactly two open orbits, given by:
\ew
\ww[+]{2}\lt(\bb{R}^{2k}\rt)^* = \lt\{\om \in \ww{2}\lt(\bb{R}^{2k}\rt)^*~\m|~ \om^k > 0\rt\} \et \ww[-]{2}\lt(\bb{R}^{2k}\rt)^* = \lt\{\om \in \ww{2}\lt(\bb{R}^{2k}\rt)^*~\m|~ \om^k < 0\rt\},
\eew
which form a single orbit under \(\GL(2k;\bb{R})\).  The stabiliser in \(\GL_+(2k;\bb{R})\) (equivalently \(\GL(2k;\bb{R})\)) of forms in either orbit is isomorphic to the real symplectic group \(\Sp(2k;\bb{R})\).  `Standard' representatives of the two orbits are given by:
\e\label{om/pm}
\om_+(k) = \th^{12} + \th^{34} + ... + \th^{2k-1,2k} \et \om_-(k) = -\th^{12} + \th^{34} + ... + \th^{2k-1,2k},
\ee
respectively.
\end{Prop}

Forms in \(\ww[+]{2}\lt(\bb{R}^{2k}\rt)^*\) are here termed emproplectic and forms in \(\ww[-]{2}\lt(\bb{R}^{2k}\rt)^*\) termed pisoplectic.\footnote{I employ these non-standard terms to avoid overuse of the phrases `correctly-oriented symplectic 2-form' (emproplectic form) and `oppositely oriented symplectic 2-form' (pisoplectic form).  The words `emproplectic' and `pisoplectic' are based on the word `symplectic', which combines the Ancient Greek `\(\si\up\mu + \pi\la\ep\ka\ta\io\ka\omi\si\)' (`sym + plektikos') to mean `braided together'.  Likewise, the words `emproplectic' and `pisoplectic' combine `\(\ep\mu + \pi\rh\omi\si +\pi\la\ep\ka\ta\io\ka\omi\si\)' (`em + pros + plektikos') and `\((\omi)\pi\io\si\om + \pi\la\ep\ka\ta\io\ka\omi\si\)' (`(o)piso + plektikos') to mean `braided forwards' and `braided backwards', respectively.}  The \(\GL(2k;\bb{R})\)-orbit comprising of both emproplectic and pisoplectic forms is termed the orbit of symplectic forms \(\om\) and can be characterised by the property that the linear map \(\io_\om\) is an isomorphism (see \eref{io-ep}).

Using \pref{symp}, I prove the following result:
\begin{Prop}\label{dual-symp}
Let \(k \ge 3\).  The action of \(\GL_+(2k;\bb{R})\) on \(\ww{2k-2}\lt(\bb{R}^{2k}\rt)^*\) has exactly two orbits, given by:
\ew
\ww[+]{2k-2}\lt(\bb{R}^{2k}\rt)^* = \lt\{\om^{k-1}~\m|~\om \text{ is emproplectic}\rt\} \et \ww[-]{2k-2}\lt(\bb{R}^{2k}\rt)^* = \lt\{-\om^{k-1}~\m|~\om \text{ is pisoplectic}\rt\}.
\eew
Term forms in \(\ww[+]{2k-2}\lt(\bb{R}^{2k}\rt)^*\) osemproplectic and forms in \(\ww[-]{2k-2}\lt(\bb{R}^{2k}\rt)^*\) ospisoplectic.\footnote{The prefix `os' (`\(\om\si\)') means `as if' and is pronounced similarly to the word `horse'.}  Standard examples of osemproplectic and ospisoplectic forms are given by:
\e\label{vpi+}
\vpi_+(k) = \lt(\sum_{i=2}^k \th^{12...\h{2i-1,2i}...2k-1,2k}\rt) + \th^{34...2k-1,2k} = \frac{\om_+(k)^{k-1}}{(k-1)!}
\ee
and:
\ew
\vpi_-(k) = \lt(\sum_{i=2}^k \th^{12...\h{2i-1,2i}...2k-1,2k}\rt) - \th^{34...2k-1,2k} = -\frac{\om_-(k)^{k-1}}{(k-1)!}.
\eew
The stabiliser in \(\GL_+(2k;\bb{R})\) of an osemproplectic or ospisoplectic form is isomorphic to \(\Sp(2k;\bb{R})\).  If \(k\) is odd, then \(\ww[+]{2k-2}\lt(\bb{R}^{2k}\rt)^*\) and \(\ww[-]{2k-2}\lt(\bb{R}^{2k}\rt)^*\) are both individually invariant under \(\GL(n;\bb{R})\), while if \(k\) is even, the two \(\GL_+(n;\bb{R})\)-orbits form a single \(\GL(n;\bb{R})\)-orbit.  I shall say that \(\vpi\) is ossymplectic if it lies in \(\ww[+]{2k-2}\lt(\bb{R}^{2k}\rt)^* \cup \ww[-]{2k-2}\lt(\bb{R}^{2k}\rt)^*\).  This is equivalent to the condition that \(\ep_\vpi\) is an isomorphism.
\end{Prop}

\begin{proof}
Recall diagram \eqref{almost-equi} and let \(n = 2k\) and \(p = 2\).  Let \(w\) denote an emproplectic (respectively, pisoplectic) element of \(\ww{2}\bb{R}^{2k}\) and let \(\vpi = w \hk vol\).  Then, one obtains the following diagram:
\ew
\bcd[column sep = 1.4cm]
\GL_+(2k;\bb{R}) \cdot w \arrow[loop right] \ar[d, " \si \mt \si \hk vol "] & \GL_+(2k;\bb{R}) \ar[d, " \Ph "]\\
\GL_+(2k;\bb{R}) \cdot \vpi \ar[loop right]& \GL_+(2k;\bb{R})\\
\ecd
\eew
Since the left-hand vertical arrow is an isomorphism, one has \(\Stab_{\GL_+(2k;\bb{R})}(\vpi) = \Ph\lt(\Stab_{\GL_+(2k;\bb{R})}(w)\rt)\).  As \(\Stab_{\GL_+(2k;\bb{R})}(w) \cong \Sp(2k;\bb{R})\) lies in \(\SL(2k;\bb{R})\), it follows that \(\Ph(F) = F^{-1}\) on \(\Stab_{\GL_+(2k;\bb{R})}(w)\) (see \eref{Ph-defn}) and hence \(\Stab_{\GL_+(2k;\bb{R})}(\vpi) = \Stab_{\GL_+(2k;\bb{R})}(w) \cong \Sp(2k;\bb{R})\), as required.

Now take \(w = e_{12} + e_{34} + ... + e_{2k-1,2k}\) and recall \(\om_+(k) = \th^{12} + \th^{34} + ... + \th^{2k-1,2k}\) defined in \eref{om/pm}.  Taking \(vol = \frac{\om_+(k)^k}{k!}\), one finds that:
\ew
w \hk vol = \frac{\om_+(k)^{k-1}}{(k-1)!} = \vpi_+(k) \in \ww[+]{2k-2}\lt(\bb{R}^{2k}\rt)^*.
\eew
Likewise, if one takes \(w = -e_{12} + e_{34} + ... + e_{2k-1,2k}\) and \(vol = \frac{\om_+(k)^k}{k!} = -\frac{\om_-(k)^k}{k!}\), one finds:
\ew
w \hk vol = -\frac{\om_-(k)^{k-1}}{(k-1)!} = \vpi_-(k) \in \ww[-]{2k-2}\lt(\bb{R}^{2k}\rt)^*,
\eew
as required.  To prove the statement regarding \(\GL(2k;\bb{R})\)-orbits, let \(\om\) be an emproplectic form, let \(F \in \GL_-(2k;\bb{R})\) and consider \(-F^*\om\).  If \(k\) is odd, \(\lt(-F^*\om\rt)^k = (-1)^kF^*\lt(\om^k\rt) = -F^*\lt(\om^k\rt) > 0\) and so \(-F^*\om\) is emproplectic.  Thus, \(F^*\lt(\om^{k-1}\rt) = \lt(-F^*\om\rt)^{k-1}\) is osemproplectic as claimed.  Alternatively, if \(k\) is even, then \(\lt(-F^*\om\rt)^k < 0\), hence so \(-F^*\om\) is pisoplectic and whence \(F^*\lt(\om^{k-1}\rt) = -\lt(-F^*\om\rt)^{k-1}\) is ospisoplectic, as claimed.  The final statement regarding the characterisation of ossymplectic forms is clear.

\end{proof}

I remark that the notions of osemproplectic and ospisoplectic forms are still valid in dimension 4, however in this case they coincide with emproplectic and pisoplectic forms, respectively.  Accordingly, I reserve the terms osemproplectic and ospisoplectic for dimension \(2k\), \(k \ge 3\).  Also, note that given an osemproplectic form \(\vpi\) on an oriented \(2k\)-manifold, the form \(-\vpi\) is an ospisoplectic form on \(\ol{\M}\), where the overline denotes orientation-reversal.  Thus, to prove the relative \(h\)-principle for osemproplectic and ospisoplectic forms, it suffices to prove it only for osemproplectic forms.\\

\subsection{\(\mb{2}\)- and \(\mb{(2k-1)}\)-forms in \(\mb{2k+1}\) dimensions, \(\mb{k \ge 2}\)}\label{Pseudo-P}  Although some aspects of the following result are known, to the author's knowledge the complete statement does not appear in the literature:

\begin{Prop}\label{P}
Let \(k \ge 2\).  The action of \(\GL_+(2k+1;\bb{R})\) on \(\ww{2}\lt(\bb{R}^{2k+1}\rt)^*\) has a unique open orbit:
\e\label{2-odd-orb}
\ww[P]{2}\lt(\bb{R}^{2k+1}\rt)^* = \lt\{\mu \in \ww{2}\lt(\bb{R}^{2k+1}\rt)^* ~\m|~ \mu^k \ne 0\rt\},
\ee
which is also an orbit of \(\GL(2k+1;\bb{R})\).  Equivalently, a 2-form \(\mu\) lies in \(\ww[P]{2}\lt(\bb{R}^{2k+1}\rt)^*\) \iff\ \(\io_\mu\) has rank \(2k\) (see \eref{io-ep}).  Call forms in this orbit pseudoplectic; a standard representative of this orbit may be taken to be:
\e\label{mu0}
\mu_0(k) = \th^{23} + \th^{45} + ... + \th^{2k,2k+1}.
\ee
The stabiliser of any pseudoplectic form is connected and is isomorphic to the group:
\ew
\lt\{ \bpm A_{1 \x 1} & B_{1 \x 2k} \\ 0_{2k \x 1} & C_{2k \x 2k} \epm ~\m|~ A \in \bb{R}_{>0}, B \in \lt(\bb{R}^{2k}\rt)^* \text{ and } C \in \Sp(2k;\bb{R}) \rt\}.
\eew

Likewise, the action of \(\GL_+(2k+1;\bb{R})\) on \(\ww{2k-1}\lt(\bb{R}^{2k+1}\rt)^*\) has a unique open orbit:
\ew
\ww[OP]{2k-1}&\lt(\bb{R}^{2k+1}\rt)^* =  \lt\{ \xi \in \ww{2k-1}\lt(\bb{R}^{2k+1}\rt)^* ~\m|~ \text{the linear map } \ep_\xi \text{ has rank } 2k \rt\},
\eew
which is also an orbit of \(\GL(2k+1;\bb{R})\).  Term forms in this orbit ospseudoplectic.  A standard representative of this orbit may be taken to be:
\ew
\xi_0(k) = \sum_{i = 1}^k \th^{123...\h{2i,2i+1}...2k,2k+1} = \th^1 \w \vpi_+(k),
\eew
where \(\vpi_+(k)\) is viewed as a form on \(\bb{R}^{2k+1}\) via \(\bb{R}^{2k} \cong \<e_2,...,e_{2k+1}\? \pc \bb{R}^{2k+1}\) (and formally \(\vpi_+(k) = \om_+(k)\) when \(k = 2\)).  The stabiliser of any ospseudoplectic form is isomorphic to the group:
\e\label{OP-stab}
\lt\{ \bpm A_{1 \x 1} & 0_{1 \x 2k} \\ B_{2k \x 1} & A^{-\frac{1}{2k-2}} \cdot C_{2k \x 2k} \epm ~\m|~ A \in \bb{R}_{>0}, B \in \bb{R}^{2k} \text{ and } C \in \Sp(2k;\bb{R}) \rt\}.
\ee
In particular, the stabiliser is connected.
\end{Prop}

\begin{proof}
The set \(\lt\{\mu \in \ww{2}\lt(\bb{R}^{2k+1}\rt)^* ~\m|~ \mu^k = 0\rt\}\) is an affine subvariety of \(\ww{2}\lt(\bb{R}^{2k+1}\rt)^*\), so if it is a proper subset of \(\ww{2}\lt(\bb{R}^{2k+1}\rt)^*\) it must have positive codimension, and hence can contain no open orbits of \(\GL_+(2k+1;\bb{R})\).  Thus, to prove \eref{2-odd-orb} it suffices to prove that \(\ww[P]{2}\lt(\bb{R}^{2k+1}\rt)^* = \lt\{\mu \in \ww{2}\lt(\bb{R}^{2k+1}\rt)^* ~\m|~ \mu^k \ne 0\rt\}\) is non-empty and a single orbit.

Firstly, let me show that \(\ww[P]{2}\lt(\bb{R}^{2k+1}\rt)^*\) is a single orbit of \(\GL_+(2k+1;\bb{R})\).  Given \(\mu \in \ww[P]{2}\lt(\bb{R}^{2k+1}\rt)^*\), since the rank of an anti-symmetric bilinear form is always even (and the dimension of \(\bb{R}^{2k+1}\) is odd) it follows that \(\io_\mu\) has a non-trivial kernel.  Pick a 1-dimensional subspace \(\bb{L}\) of the kernel and let \(\bb{B} \pc \bb{R}^{2k+1}\) be a \(2k\)-dimensional complement to \(\bb{L}\) in \(\bb{R}^{2k+1}\).  Since \(\mu^k \ne 0\), one may regard \(\mu\) as an emproplectic 2-form on \(\bb{B}\) for some suitable choice of orientation on \(\bb{B}\).  Thus, one can pick a correctly-oriented basis \((f_2,...,f_{2k+1})\) of \(\bb{B}\) with dual basis \((f^2,...,f^{2k+1})\) such that \(\mu = f^{23} + ... + f^{2k,2k+1}\).  Now, define \(f_1\) to be a non-zero vector in \(\bb{L}\) such that \((f_1,...,f_{2k+1})\) is a correctly oriented basis of \(\bb{R}^{2k+1}\).  Then, \wrt\ this basis:
\ew
\mu = f^{23} + ... + f^{2k,2k+1}.
\eew
Thus, \(\ww[P]{2}\lt(\bb{R}^{2k+1}\rt)^*\) is a single orbit under \(\GL_+(2k+1;\bb{R})\).  This argument also shows that \(\ww[P]{2}\lt(\bb{R}^{2k+1}\rt)^* \ne \es\) (since \(\lt(f^{23} + ... + f^{2k,2k+1}\rt)^k \ne 0\) in the above basis) and:
\ew
\ww[P]{2}\lt(\bb{R}^{2k+1}\rt)^* = \lt\{\mu \in \ww{2}\lt(\bb{R}^{2k+1}\rt)^* ~\m|~ \io_\mu \text{ has rank } 2k\rt\}.
\eew
Moreover, since this is the only open orbit of \(\GL_+(2k+1;\bb{R})\), it follows that \(\ww[P]{2}\lt(\bb{R}^{2k+1}\rt)^*\) must also form a single orbit under the action of \(\GL(2k+1;\bb{R})\).

Now fix a (positive) volume element \(vol\) on \(\bb{R}^{2k+1}\), let \(\mu \in \ww{2}\bb{R}^{2k+1}\) and write \(\xi = \mu \hk vol\).  Recall that \(\ep_\xi\) and \(\io_\mu\) denote the linear maps:
\ew
\bcd[row sep = 0pt]
\ep_\xi: \lt(\bb{R}^{2k+1}\rt)^* \ar[r]& \ww{2k}\lt(\bb{R}^{2k+1}\rt)^* & \io_\mu: \lt(\bb{R}^{2k+1}\rt)^* \ar[r]& \bb{R}^{2k+1}\\
\al \ar[r, maps to]& \al \w \xi & \al \ar[r, maps to]& \al \hk \mu\hs{1.5pt},
\ecd
\eew
respectively.  The first step is to understand how the maps \(\ep_\xi\) and \(\io_\mu\) are related:
\begin{Cl}
The maps \(\ep_\xi\) and \(\io_\mu\) satisfy the relation:
\ew
\ep_\xi = -\io_\mu \hk vol.
\eew
\end{Cl}
\begin{proof}[Proof of Claim]
This follows directly from the (readily verified) identity:
\ew
\al \w (\be \hk \ga) = -(\al \hk \be) \hk \ga,
\eew
where \(\al \in \lt(\bb{R}^n\rt)^*\), \(\be \in \ww{2}\bb{R}^n\) and \(\ga \in \ww{n}\lt(\bb{R}^n\rt)^*\).  Setting \(\be = \mu\) and \(\ga = vol\) yields the required result, completing the proof of the claim.

\let\qed\relax
\end{proof}

Thus, \(\io_\mu\) has the same rank as \(\ep_\xi\) and so \(\io_\mu\) has rank \(2k\) \iff\ \(\ep_\xi\) has rank \(2k\).  It follows at once from \lref{no-orb-lem} that the action of \(\GL_+(2k+1;\bb{R})\) on \(\ww{2k-1}\lt(\bb{R}^{2k+1}\rt)^*\) has a single open orbit (which is also an orbit of \(\GL(2k+1;\bb{R})\)) given by:
\ew
\ww[OP]{2k-1}&\lt(\bb{R}^{2k+1}\rt)^* = \lt\{ \xi \in \ww{2k-1}\lt(\bb{R}^{2k+1}\rt)^* ~\m|~ \ep_\xi \text{ has rank } 2k \rt\}.
\eew
Likewise, the explicit formula for \(\xi_0\) also follows at once.

The formula for the stabiliser of pseudoplectic forms is a simple application of \lref{stabiliser-comp-lem}(1).  For the stabiliser of ospseudoplectic forms, one uses \lref{stabiliser-comp-lem}(2), together with the observations that:
\begin{itemize}
\item If \(k\) is odd, then \(-\vpi_+(k) = -\frac{(-\om_+(k))^{k-1}}{(k-1)!}\) is ospisoplectic (since \(-\om_+(k)\) is pisoplectic when \(k\) is odd) and since (when \(k\) is odd) the orbit of osemproplectic and ospisoplectic forms are each closed under the action of \(\GL(2k;\bb{R})\) (and not just \(\GL_+(2k;\bb{R})\)) it follows that \(\vpi_+(k)\) and \(-\vpi_+(k)\) lie in separate \(\GL_+(2k;\bb{R})\)-orbits.  This forces the second set on the right-hand side of \eref{mcs-stab} to be empty, as claimed.
\item If \(k\) is even, then \(-\vpi_+(k) = \frac{(-\om_+(k))^{k-1}}{(k-1)!}\) is osemproplectic (since \(-\om_+(k)\) is emproplectic when \(k\) is even) and thus \(\vpi_+(k)\) and \(-\vpi_+(k)\) lie in the same \(\GL_+(2k;\bb{R})\)-orbit.  Moreover, since \(k\) is even:
\ew
\Stab_{\GL(2k;\bb{R})}\lt(\vpi_+(k)\rt) = \Stab_{\GL_+(2k;\bb{R})}\lt(\vpi_+(k)\rt)
\eew
and thus, once again, the second set on the right-hand side of \eref{mcs-stab} is empty.  This completes the proof.
\end{itemize}

\end{proof}

\begin{Rk}\label{OP-Hyp}
Let \(\mu \in \ww[P]{2}\lt(\bb{R}^{2k+1}\rt)^*\).  The kernel of the linear map \(\io_\mu\) defines a 1-dimensional subspace of \(\bb{R}^{2k+1}\) which I denote \(\ell_\mu\).  Moreover, the orientation on \(\bb{R}^{2k+1}\) induces a natural orientation on \(\ell_\mu\) defined as follows: given a 1-form \(\th\) on \(\bb{R}^{2k+1}\) which does not vanish on \(\ell_\mu\), say that \(\th\) is positive on \(\ell_\mu\) if \(\th \w \mu^k > 0\).  Likewise, let \(\xi \in \ww[OP]{2k-1}\lt(\bb{R}^{2k+1}\rt)^*\) be ospseudoplectic.  Then, the annihilator of the 1-dimensional subspace \(\Ker(\ep_\xi) \pc \lt(\bb{R}^{2k+1}\rt)^*\) defines a hyperplane in \(\bb{R}^{2k+1}\) associated to \(\xi\); denote this hyperplane by \(\Pi_\xi\).  By \eref{OP-stab}, every element \(F\) of \(\Stab_{\GL_+(2k+1;\bb{R})}(\xi)\) restricts to an orientation-preserving automorphism of \(\Pi_\xi\); thus, it is possible to orient the planes \(\Pi_\xi\) consistently for all \(\xi\).  Specifically, there is a unique orientation on \(\Pi_{\xi}\) such that \(\xi = \th \w \vpi = \th \w \om^{k-1}\), where \(\th\) is a compatibly oriented generator of \(\Ann(\Pi|_\xi) = \Ker(\ep_\xi)\) and \(\vpi\) is an osemproplectic form on \(\Pi_\xi\) \wrt\ the given orientation (equivalently, \(\om\) is an emproplectic form on \(\Pi_\xi\)).  Moreover, as \(\th\) varies, \(\om\) defines a conformal class of emproplectic forms on \(\Pi_\xi\); thus \(\xi\) determines a co-oriented almost contact structure on \(\bb{R}^{2k+1}\) (in an algebraic sense); see \cite[\S10.1.B]{ItthP} for further details.  E.g.\ In the case of the standard ospseudoplectic form \(\xi_0(k) = \th^1 \w \vpi_+(k)\), \(\Pi_{\xi_0(k)}\) is simply \(\<e_2,...,e_{2k+1}\?\) and the corresponding conformal class of emproplectic forms on \(\Pi_\xi\) is just \(\la \1 \lt(\th^{23} + ... + \th^{2k,2k+1}\rt)\) for \(\la > 0\).\\
\end{Rk}

\subsection{Classification of stable forms}

For the sake of completeness, I briefly recount the classification of stable forms; see \cite{MASF} for further details of the 8-dimensional case:

\begin{Thm}\label{stab-clas}
The action of  \(\GL_+(n;\bb{R})\) on \(\ww{p}\lt(\bb{R}^n\rt)^*\) has precisely the following open orbits for \(2 \le p \le n-2\):

\(\mb{n = 2k,}\) \(\mb{k \ge 2:}\) \(\ww[\pm]{2}\lt(\bb{R}^{2k}\rt)^*\) and \(\ww[\pm]{2k-2}\lt(\bb{R}^{2k}\rt)^*\);

\(\mb{n = 2k+1,}\) \(\mb{k \ge 2:}\) \(\ww[P]{2}\lt(\bb{R}^{2k+1}\rt)^*\) and \(\ww[OP]{2k-1}\lt(\bb{R}^{2k+1}\rt)^*\);

\(\mb{n = 6:}\) \(\ww[\pm]{3}\lt(\bb{R}^6\rt)^*\);

\(\mb{n = 7:}\) \(\pm \ww[+]{3}\lt(\bb{R}^7\rt)^*\), \(\pm \ww[\tl]{3}\lt(\bb{R}^7\rt)^*\), \(\pm \ww[+]{4}\lt(\bb{R}^7\rt)^*\) and \(\pm \ww[\tl]{4}\lt(\bb{R}^7\rt)^*\);

\(\mb{n = 8:}\) The action of \(\GL_+(8;\bb{R})\) on \(\ww{3}\lt(\bb{R}^8\rt)^*\) has precisely three open orbits, represented by the 3-forms:
\caw
\ze_c = \th^{123} + \frac{1}{2}\th^{147} - \frac{1}{2}\th^{156} + \frac{1}{2}\th^{246} + \frac{1}{2}\th^{257} + \frac{1}{2}\th^{345} - \frac{1}{2}\th^{367} + \frac{\sqrt{3}}{2}\th^{458} + \frac{\sqrt{3}}{2}\th^{678};\\
\ze_s = \frac{\sqrt{3}}{2}\th^{147} - \frac{\sqrt{3}}{2}\th^{156} + \th^{238} + \frac{1}{2}\th^{246} - \frac{1}{2}\th^{257} + \frac{1}{2}\th^{347} + \frac{1}{2}\th^{356} + \frac{1}{2}\th^{458} - \frac{1}{2}\th^{678};\\
\ze_n = -\th^{123} - \frac{1}{2}\th^{156} - \frac{1}{2}\th^{178} + \frac{1}{2}\th^{257} - \frac{1}{2}\th^{268} - \frac{1}{2}\th^{358} - \frac{1}{2}\th^{367} - \frac{\sqrt{3}}{2}\th^{458} + \frac{\sqrt{3}}{2}\th^{467},
\caaw
with stabilisers \(\bb{P}\SU(3)\), \(\SL(3;\bb{R})\) and \(\bb{P}\SU(1,2)\) respectively.  Likewise, the action of \(\GL_+(8;\bb{R})\) on \(\ww{5}\lt(\bb{R}^8\rt)^*\) has precisely three open orbits, represented by the 5-forms:
\caw
\eta_c = -\frac{\sqrt{3}}{2}\th^{12345} - \frac{\sqrt{3}}{2}\th^{12367} - \frac{1}{2}\th^{12458} + \frac{1}{2}\th^{12678} + \frac{1}{2}\th^{13468} + \frac{1}{2}\th^{13578} - \frac{1}{2}\th^{23478} + \frac{1}{2}\th^{23568} + \th^{45678};\\
\eta_s = \frac{1}{2}\th^{12345} - \frac{1}{2}\th^{12367} + \frac{1}{2}\th^{12478} + \frac{1}{2}\th^{12568} - \frac{1}{2}\th^{13468} + \frac{1}{2}\th^{13578} - \th^{14567} - \frac{\sqrt{3}}{2} \th^{23478} + \frac{\sqrt{3}}{2}\th^{23568};\\
\eta_n = -\frac{\sqrt{3}}{2} \th^{12358} + \frac{\sqrt{3}}{2}\th^{12367} - \frac{1}{2}\th^{12458} - \frac{1}{2}\th^{12467} -\frac{1}{2}\th^{13457} + \frac{1}{2}\th^{13468} - \frac{1}{2}\th^{23456} - \frac{1}{2}\th^{23478} - \th^{45678},
\caaw
with stabilisers \(\bb{P}\SU(3)\), \(\SL(3;\bb{R})\) and \(\bb{P}\SU(1,2)\) respectively.

In particular, for any stable form \(\si_0 \in \ww{p}\lt(\bb{R}^n\rt)^*\) with \(2 \le p \le n-2\) (and hence with \(0 \le p \le n\)), \(\Stab_{\GL_+(n;\bb{R})}\) is connected.\\
\end{Thm}

\section{Unboundedness of Hitchin functionals via \(h\)-principles}\label{conseq}

The aim of this section is to describe how \tref{HF-thm} follows from \tref{4hP}.  Let \(\si_0 \in \ww{p}\lt(\bb{R}^n\rt)^*\) be a Hitchin form and suppose that \(\si_0\)-forms satisfy the \(h\)-principle.
\begin{Thm}
For any closed, oriented \(n\)-manifold \(\M\) admitting \(\si_0\)-forms and any \(\al \in \dR{p}(\M)\), the Hitchin functional:
\ew
\mc{H}: \CL^p_{\si_0}(\al) \to (0,\infty)
\eew
is unbounded above.  More generally, if \(\M\) is a closed, oriented \(n\)-orbifold and \(\CL^p_{\si_0}(\al) \ne \es\), then the same conclusions apply.
\end{Thm}

\begin{proof}
Begin with the case where \(\M\) is a manifold.  Since \(\Om_{\si_0}^p(\M) \ne \es\) and \(\CL^p_{\si_0}(\al) \emb \Om_{\si_0}^p(\M)\) is a homotopy equivalence, \(\CL^p_{\si_0}(\al) \ne \es\).  Thus, pick \(\si \in \CL^p_{\si_0}(\al)\).  Let \(f: \B^n_1(0) \emb \M\) be an embedding, write \(W = f\lt(\B^n_1(0)\rt)\) and \(U = f\Big(\B^n_\frac{1}{2}(0)\Big)\) and consider the polyhedron \(A \pc \M\) given by \(A = \ol{U} \cup \M\osr W\).  Let \(\ch:\M \to [0,1]\) be a smooth function on \(\M\) such that \(\ch|_{\Op\lt(\ol{U}\rt)} \equiv 1\) and \(\ch|_{\Op(\M\osr W)} \equiv 0\), and for each \(\la \in (0,\infty)\) define \(\si_\la \in \Om_{\si_0}^p(\M)\) by:
\ew
\si_\la = \lt(1 + \la \cdot \ch\rt)\si.
\eew
Then \(\dd\si_\la = 0\) on \(\Op(A)\), since \(\dd\si = 0\) and \(\ch\) is locally constant on \(\Op(A)\), and hence \(\si_\la \in \Om_{\si_0}(\M;\si_\la|_{\Op(A)})\) for all \(\la > 0\).  Moreover, the restrictions \(\si|_{\Op(A)}\) and \(\si_\la|_{\Op(A)}\) both lie in \(\al|_{\Op(A)} \in \dR{p}(\Op(A))\) (for this point, it is useful to recall that \(\ol{U}\) is contractible), so by the relative \(h\)-principle, \(\CL^p_{\si_0}(\al;\si_\la|_{\Op(A)}) \emb \Om_{\si_0}(\M;\si_\la|_{\Op(A)})\) is a homotopy equivalence and thus one can continuously deform \(\si_\la\) relative to \(\Op(A)\) into \(\si'_\la \in \CL^p_{\si_0}(\al)\) such that:
\ew
\si'_\la = (1+\la)\si \text{ on } U \et \si'_\la = \si \text{ on } \M\osr W.
\eew
One now computes that:
\ew
\mc{H}(\si'_\la) \ge \bigintsss_U vol_{\si'_\la} = \bigintsss_U vol_{(1+\la)\si} = (1+\la)^\frac{n}{p} \bigintsss_U vol_{\si} \to \infty \as \la \to \infty,
\eew
as required.  In the case where \(\M\) is an orbifold, provided \(\CL^p_{\si_0}(\al) \ne \es\), one can  apply the above argument to the smooth locus of \(\M\), leading to the same conclusion.

\end{proof}

Thus, \tref{HF-thm} follows from \tref{4hP}, as claimed (noting that neither pseudoplectic forms nor ospseudoplectic forms are Hitchin forms by \pref{P}).  Moreover, since \g\ 4-forms and \slc\ 3-forms are both Hitchin forms and satisfy the \(h\)-principle (see \sref{1st-Appl}), Hitchin functionals on \g\ 4-forms and \slc\ 3-forms are also always unbounded above.\\

\section{The relative \(\lowercase{h}\)-principle for stable, ample forms}\label{ample-forms}

The aim of this section is to prove \tref{hPThm-1}.  Recall the statement of the theorem:\vs{3mm}

\noindent{\bf Theorem \ref{hPThm-1}.}
\em Let \(\si_0\in\ww{p}\lt(\bb{R}^n\rt)^*\) be stable (\(p \ge 1\)).  If \(\si_0\) is ample, then \(\si_0\)-forms satisfy the relative \(h\)-principle.\vs{2mm}\em

Let \(\M\) be an oriented \(n\)-manifold.  The symbol of the exterior derivative on \((p-1)\)-forms is the unique vector bundle homomorphism \(\mc{D}:\ww{p-1}\T^*\M^{(1)}\to\ww{p}\T^*\M\) such that the following diagram commutes:
\ew
\begin{tikzcd}
\Ga\lt(\M,\ww{p-1}\T^*\M^{(1)}\rt) \ar[rr, "\mLa{\mc{D}}"] & & \Om^{p}(\M)\\
& \Om^{p-1}(\M) \ar[ru, "\mLa{\dd}"'] \ar[lu, "\mLa{j_1}"] &
\end{tikzcd}
\eew
Explicitly, identifying \(\ww{p-1}\T^*\M^{(1)} \cong \ww{p-1}\T^*\M \ds \lt(\T^*\M \ts \ww{p-1}\T^*\M\rt)\) as usual, \(\mc{D}\) is simply the composite map:
\ew
\ww{p-1}\T^*\M \ds \lt(\T^*\M \ts \ww{p-1}\T^*\M\rt) \oto{\mns{~proj_2~}} \T^*\M \ts \ww{p-1}\T^*\M \oto{\mns{\w}} \ww{p}\T^*\M.
\eew
It follows that \(\mc{D}:\ww{p-1}\T^*\M^{(1)}\to\ww{p}\T^*\M\) is a fibrewise surjective linear map.

\begin{Defn}
Let \(a: D^q \to \Om^p(\M)\) be a continuous map.  Define a fibred differential relation \(\cal{R}_{\si_0}(a) \pc \ww{p-1}\T^*\M^{(1)}_{D^q}\) via:
\ew
\cal{R}_{\si_0}(a) = \lt\{ (s,T) \in \ww{p-1}\T^*\M^{(1)}_{D^q} ~\m|~ \mc{D}(T) + a(s) \in \ww[\si_0]{p}\T^*\M \rt\}.
\eew
Equivalently, \(\cal{R}_{\si_0}(a)\) is the preimage of \(\ww[\si_0]{p}\T^*\M_{D^q}\) under the map:
\ew
\ww{p-1}\T^*\M^{(1)}_{D^q} \oto{\mc{D} + a} \ww{p}\T^*\M_{D^q}.
\eew
\end{Defn}

\begin{Lem}\label{hP->hP}
Suppose that, for every \(q \ge 0\) and every continuous \(a: D^q \to \Om^p(\M)\), the relation \(\cal{R}_{\si_0}(a)\) satisfies the relative \(h\)-principle.  Then, \(\si_0\)-forms satisfy the relative \(h\)-principle.
\end{Lem}

The proof of \lref{hP->hP} relies on \tref{NCH}, proved in \aref{NCH-Sec}.  Recall the statement of the theorem:\vs{3mm}

\noindent{\bf Theorem \ref{NCH}}{\bf .}
\em Let \(\M\) be an \(n\)-manifold (not necessarily oriented and possibly non-compact or with boundary).  Then, there exists an injective continuous linear operator \(\io:\Ds_p\dR{p}(\M) \to \Ds_p\Om^p(\M)\) of degree 0 and a continuous linear operator \(\de:\Ds_p\Om^p(\M) \to \Ds_p\Om^p(\M)\) of degree \(-1\) satisfying \(\de^2 = 0\) such that for each \(0 \le p \le n\):
\begin{equation*}
\Om^p(\M) = \io\dR{p}(\M) \ds \dd\Om^{p-1}(\M) \ds \de\Om^{p+1}(\M) \tag{\ref*{NCHD-eq}}
\end{equation*}
in the category of \F\ spaces, where both \(\dd\) and \(\de\) act as \(0\) on \(\io\dR{p}(\M)\) and the maps \(\dd: \de\Om^{p+1}(\M) \to \dd\Om^p(\M)\) and \(\de: \dd\Om^p(\M) \to \de\Om^{p+1}(\M)\) are mutually inverse.  In particular, the projections \(\Om^p(\M) \to \dd\Om^{p-1}(\M)\) and \(\Om^p(\M) \to \de\Om^{p+1}(\M)\) are given by \(\dd \circ \de\) and \(\de \circ \dd\), respectively.\vs{2mm}\em

\begin{proof}[Proof of \lref{hP->hP}]
Let \(A \pc \M\) be a (possibly empty) polyhedron, let \(\al: D^q \to \dR{p}(\M)\) be a continuous map and let \(\fr{F}_0: D^q \to \Om^p_{\si_0}(\M)\) be a continuous map such that:
\begin{enumerate}
\item For all \(s \in \del D^q\): \(\dd\fr{F}_0(s) = 0\) and \([\fr{F}_0(s)] = \al(s) \in \dR{p}(\M)\);
\item For all \(s \in D^q\): \(\dd\lt(\fr{F}_0(s)|_{\Op{A}}\rt) = 0\) and \(\lt[\fr{F}_0(s)|_{\Op(A)}\rt] = \al(s)|_{\Op(A)} \in \dR{p}(\Op(A))\).
\end{enumerate}

Then, \(\fr{F}_0\) is a \(D^q\)-indexed family of \(p\)-forms on \(\M\), and thus one may regard \(\fr{F}_0\) as a section of the bundle \(\lt(\ww{p}\T^*\M\rt)_{D^q}\).  Define \(a: D^q \to \Om^p(\M)\) via \(a(s) = \io\al(s)\), for \(\io\) as in \tref{NCH}, so that \(a(s)\) represents the cohomology class \(\al(s)\) for each \(s \in D^q\), and consider the diagram:
\ew
\bcd[row sep = 15mm]
\lt(\ww{p-1}\T^*\M\rt)^{(1)}_{D^q} \ar[rr, " \mc{D} + a "] \ar[dr] & & \lt(\ww{p}\T^*\M\rt)_{D^q} \ar[dl, shift right = 3pt] \\
& D^q \x \M \ar[ur, shift right = 3pt, "\fr{F}_0" '] &
\ecd
\eew
The strategy of the proof will be to lift the section \(\fr{F}_0\) along the map \(\mc{D} + a\) to a section \(F_0\) of \(\lt(\ww{p-1}\T^*\M\rt)^{(1)}_{D^q}\) which is holonomic over the region \([\del D^q \x \M] \cup [D^q \x \Op(A)]\) and then apply the relative \(h\)-principle for the fibred relation \(\cal{R}_{\si_0}(a)\).

Firstly, consider the map:
\e\label{F0-a}
\bcd[row sep = 0pt]
\del D^q \ar[r]& \dd\Om^{p-1}(\M)\\
s \ar[r, maps to]& \fr{F}_0(s) - a(s),
\ecd
\ee
which is well-defined since \([\fr{F}_0(s)] = [a(s)] = \al(s)\) for \(s \in \del D^q\).  For each \(s \in \del D^q\) define \(\eta(s) = \de\lt(\fr{F}_0(s) - a(s)\rt)\) for \(\de\) as in \tref{NCH}, so that \(\dd\eta(s) = \fr{F}_0(s) - a(s)\) for all \(s \in \del D^q\) and consider the `extension problem':
\ew
\bcd[row sep = 0pt]
\del D^q \ar[r, "\eta(-)|_{\Op(A)}"] \ar[ddddddd, hook]& \Om^{p-1}(\Op(A)) \ar[ddddddd, "\dd"]\\
~&\\
~&\\
~&\\
~&\\
~&\\
~&\\
D^q \ar[ruuuuuuu, dashed, "\ze"] \ar[r] & \dd\Om^{p-1}(\Op(A))\\
s \ar[r, maps to] & \fr{F}_0(s)|_{\Op(A)} - a(s)|_{\Op(A)}
\ecd
\eew
Since, by \tref{NCH}, the right-hand vertical map is (up to isomorphism) simply a projection onto a direct summand, the extension problem for this map is trivial and thus one can extend the lift \(\eta(-)|_{\Op(A)}\) from \(\del D^q\) to all of \(D^q\), yielding a map \(\ze\) as shown in the diagram (see \cite[Thm.\ 9]{HToIDM} for a closely related result).  Define a map \(f_0: [\del D^q \x \M] \cup [D^q \x \Op(A)] \to \lt(\ww{p-1}\T^*\M\rt)^{(1)}_{D^q}\) via:
\ew
f_0(s,p) =
\begin{dcases*}
(j_1\eta(s))|_p &if \((s,p) \in \del D^q \x \M\);\\
(j_1\ze(s))|_p &if \((s,p) \in D^q \x \Op(A)\),
\end{dcases*}
\eew
noting that this definition is valid since \(\ze|_{\del D^q}(-) = \eta(-)|_{\Op(A)}\), and consider the new extension problem:
\ew
\bcd[row sep = 15mm]
~[\del D^q \x \M] \cup [D^q \x \Op(A)] \ar[r, "f_0"] \ar[d, hook]& \lt(\ww{p-1}\T^*\M\rt)^{(1)}_{D^q} \ar[d, "\mc{D} + a"]\\
D^q \x \M \ar[ru, dashed, "F_0"] \ar[r, "\fr{F}_0"]& \lt(\ww{p}\T^*\M\rt)_{D^q}
\ecd
\eew
The right-hand vertical map is a fibration with contractible fibres and thus one can extend the lift \(f_0\) from \([\del D^q \x \M] \cup [D^q \x \Op(A)]\) to all of \(D^q \x \M\), yielding a map \(F_0\), as shown in the diagram.

The \(F_0\) defines a section of \(\cal{R}_{\si_0}(a)\) which is holonomic over \([\del D^q \x \M] \cup [D^q \x \Op(A)]\).  Since \(\cal{R}_{\si_0}(a)\) satisfies the relative \(h\)-principle, one can choose a homotopy of sections \(F_t\) of \(\cal{R}_{\si_0}(a)\), constant over \([\del D^q \x \M] \cup [D^q \x \Op(A)]\), such that \(F_1\) is holonomic over all of \(D^q \x \M\).  Then, \(\fr{F}_t = \mc{D} F_t + a\) defines the required homotopy of \(\fr{F}_0\) and thus \(\si_0\)-forms satisfy the \(h\)-principle, as claimed.

\end{proof}

Note that the homotopy of sections \(\fr{F}_t: [0,1] \x D^q \x \M \to \lt(\ww{p}\T^*\M\rt)_{D^q}\) cannot be taken to be arbitrarily \(C^0\)-small, due to the well-known consequence of Stokes' Theorem that \(\Om_\cl^p(\M)\pc\Om^p(\M)\) is closed in the \(C^0\)-topology and not just the \(C^1\)-topology.  Nevertheless, for the choices of \(\si_0\) considered in this paper, the relation \(\cal{R}_{\si_0}(a)\) satisfies the \(C^0\)-dense relative \(h\)-principle and thus the homotopy \(p_1 F_t\) of sections of \(\ww{p-1}\T^*\M_{D^q}\) arising in the above proof can be taken to be arbitrarily \(C^0\)-small.  This is not a contradiction, however, since \(\fr{F}_t\) depends on the full 1-jet \(F_t\), and not just on the underlying section \(p_1 F_t\).

In view of \lref{hP->hP}, to prove \tref{hPThm-1}, it suffices to prove that if \(\si_0\) is stable and ample, then \(\cal{R}_{\si_0}(a)\) satisfies the relative \(h\)-principle for any \(a: D^q \to \Om^p(\M)\).  This follows by combining the Convex Integration Theorem (\tref{EGM-Thm}) with the following result:
\begin{Prop}
Fix \(q \ge 0\) and \(a: D^q \to \Om^p(\M)\).

(1) \(\cal{R}_{\si_0}(a)\) is an open subset of \(\ww{p-1}\T^*\M^{(1)}\) \iff\ \(\si_0\) is stable.

(2) \(\cal{R}_{\si_0}(a)\) is an ample fibred differential relation \iff\ \(\si_0\) is ample.
\end{Prop}

\begin{proof}
(1) is clear, since \(\ww{p}\T^*\M_{D^q} \oto{\mns{+a}} \ww{p}\T^*\M_{D^q}\) is a homeomorphism and \(\ww{p-1}\T^*\M^{(1)}_{D^q} \oto{\mns{\mc{D}}} \ww{p}\T^*\M_{D^q}\) is continuous and open (being a fibrewise linear surjection).  For (2), note that:
\ew
\cal{R}_{\si_0}(a) = \ww{p-1}\T^*\M_{D^q} \x_{(D^q \x \M)} \cal{R}'_{\si_0}(a) \pc \ww{p-1}\T^*\M_{D^q} \ds \lt(\T^*\M \ts \ww{p-1}\T^*\M\rt)_{D^q}
\eew
as in \rref{calc-with-amp}, where:
\ew
\cal{R}'_{\si_0}(a) = \lt\{ (s,T) \in \lt(\T^*\M \ts \ww{p-1}\T^*\M\rt)_{D^q} ~\m|~ \w(T) + a(s) \in \ww[\si_0]{p}\T^*\M \rt\}.
\eew
Then, in the notation of \rref{calc-with-amp}, for each \(s \in D^q\):
\ew
\cal{R}'_{\si_0}(a)_s = \lt\{ T \in \T^*\M \ts \ww{p-1}\T^*\M ~\m|~ \w(T) + a(s) \in \ww[\si_0]{p}\T^*\M \rt\}.
\eew
As explained in \rref{calc-with-amp}, \(\cal{R}_{\si_0}(a)\) is ample \iff\ for all points \(x\in\M\), hyperplanes \(\bb{B}\pc\T_x\M\) and linear maps \(\la \in \bb{B}^* \ts \ww{p-1}\T^*_x\M\), \(\cal{R}'_{\si_0}(a)_s \cap \Pi(\bb{B},\la) \cc \Pi(\bb{B},\la)\) is ample.

Choose a splitting \(\T_x\M = \bb{L} \ds \bb{B}\) and pick an orientation on \(\bb{B}\). This choice canonically orients \(\bb{L}\); choose a correctly oriented generator \(\th\) of \(\bb{L}^*\).  Using this data, one may write:
\ew
\T^*_x\M = \bb{R}\1\th \ds \bb{B}^* \et \ww{p-1}\T^*_x\M = \th \w \ww{p-2}\bb{B}^* \ds \ww{p-1}\bb{B}^*.
\eew
Hence, there is an isomorphism:
\ew
\bcd[row sep = 0pt]
\ww{p-2}\bb{B}^* \ds \ww{p-1}\bb{B}^* \ds \lt(\bb{B}^* \ts \ww{p-1}\T^*_x\M\rt) \ar[r, "\mLa{\cong}"] & \T^*_x\M \ts \ww{p-1}\T^*_x\M\\
\al \ds \nu \ds \la \ar[r, maps to] & \th \ts (\th \w \al + \nu) + \la.
\ecd
\eew
(For completeness, in the case \(p=1\) one simply treats the space \(\ww{p-2}\bb{B}^*\) as 0, although note this paper will only be concerned with the case \(p \ge 2\).)  Using this isomorphism, one obtains the explicit description:
\ew
\Pi(\bb{B},\la) \cong \ww{p-2}\bb{B}^* \x \ww{p-1}\bb{B}^* \x \{\la\}.
\eew

Now, define \(\nu_0\in\ww{p-1}\bb{B}^*\) and \(\ta_0\in\ww{p}\bb{B}^*\) by the equation:
\e\label{la-nu&ta}
\w(\la) + a(s) = \th \w \nu_0 + \ta_0.
\ee
Then given \((\al,\nu) \in \ww{p-2}\bb{B}^* \ds \ww{p-1}\bb{B}^*\), one can compute that:
\ew
\w\lt[\th \ts \lt(\th \w \al + \nu\rt) + \la\rt] + a(s) = \th \w (\nu + \nu_0) + \ta_0,
\eew
which is a \(\si_0\)-form \iff\ \(\nu + \nu_0 \in \mc{N}_{\si_0}(\ta_0)\). Thus:
\ew
\cal{R}'_{\si_0}(a)_s \cap \Pi(\bb{B},\la) \cong \ww{p-2}\bb{B}^* \x (\mc{N}_{\si_0}(\ta_0) - \nu_0) \x \{\la\} \cc \ww{p-2}\bb{B}^* \x \ww{p-1}\bb{B}^* \x \{\la\} \cong \Pi(\bb{B},\la).
\eew
Thus, \(\cal{R}'_{\si_0}(a)_s \cap \Pi(\bb{B},\la) \cc \Pi(\bb{B},\la)\) is ample \iff\ \(\mc{N}_{\si_0}(\ta_0) - \nu_0 \cc \ww{p-1}\bb{B}^*\) is ample, which in turn is equivalent to \(\mc{N}_{\si_0}(\ta_0) \cc \ww{p-1}\bb{B}^*\) being ample.

Finally, note that, for fixed \(a(s)\), the assignment \(\la \mt (\nu_0,\ta_0)\) described in \eref{la-nu&ta} is surjective and thus, as \(\la\) varies, \(\ta_0\) realises all possible values in \(\ww{p}\bb{B}^*\).  Hence, \(\cal{R}_{\si_0}(a)\) is ample \iff\ \(\mc{N}_{\si_0}(\ta_0) \cc \ww{p-1}\bb{B}^*\) is ample for all \(\ta_0 \in \ww{p}\bb{B}^*\), i.e.\ \iff\ \(\si_0\) is ample, as claimed.

\end{proof}

\section{Faithful, connected and abundant \(p\)-forms}\label{FCA}

The aim of this section is to develop theoretical tools for effectively verifying whether a given stable form is ample.  Let \(Emb\lt(\bb{R}^{n-1},\bb{R}^n\rt)\) denote the space of linear embeddings \(\io:\bb{R}^{n-1} \to \bb{R}^n\).  Given \(\si_0 \in \ww{p}\lt(\bb{R}^n\rt)^*\) (\(p \ge 1\)), there is a natural map:
\ew
\bcd[row sep = 0pt]
\cal{T}_{\si_0}:Emb\lt(\bb{R}^{n-1},\bb{R}^n\rt) \ar[r] & \ww{p}\lt(\bb{R}^{n-1}\rt)^*\\
\io \ar[r, maps to] & \io^*\si_0.
\ecd
\eew
\(\GL_+(n-1;\bb{R})\) acts on \(Emb\lt(\bb{R}^{n-1},\bb{R}^n\rt)\) via precomposition, and the quotient \(\rqt{Emb\lt(\bb{R}^{n-1},\bb{R}^n\rt)}{\GL_+(n-1;\bb{R})}\) can naturally be identified with the oriented Grassmannian \(\oGr_{n-1}\lt(\bb{R}^n\rt)\).  Given \(f\in\GL_+(n-1;\bb{R})\), a direct computation shows:
\ew
\cal{T}_{\si_0}(\io \circ f) = f^*\io^*\lt(\si_0\rt) = f^*\cal{T}_{\si_0}(\io).
\eew
Thus, \(\cal{T}_{\si_0}\) descends to a map:
\ew
\bcd
\oGr_{n-1}\lt(\bb{R}^n\rt) \ar[r] & \rqt{\ww{p}\lt(\bb{R}^{n-1}\rt)^*}{\GL_+(n-1;\bb{R})}.
\ecd
\eew
Write \(\mc{S}(\si_0)\) for the stabiliser of \(\si_0\) in \(\GL_+(n;\bb{R})\) and note that \(\mc{S}(\si_0)\) acts on \(Emb\lt(\bb{R}^{n-1},\bb{R}^n\rt)\) (and hence on \(\oGr_{n-1}\lt(\bb{R}^n\rt)\)) on the left, via post-composition.  Clearly, \(\cal{T}_{\si_0}\) is invariant under this action and thus \(\cal{T}_{\si_0}\) descends further to a map:
\ew
\bcd
\mc{T}_{\si_0}:\lqt{\oGr_{n-1}\lt(\bb{R}^n\rt)}{\mc{S}(\si_0)} \ar[r] & \rqt{\ww{p}\lt(\bb{R}^{n-1}\rt)^*}{\GL_+(n-1,\bb{R})}.
\ecd
\eew
Using this notation, I make the following definition:
\begin{Defn}
Let \(\si_0\in\ww{p}(\bb{R}^n)^*\).
\begin{enumerate}
\item \(\si_0\) is termed faithful if \(\mc{T}_{\si_0}\) is injective.

\item \(\si_0\) is termed connected if for each orbit \(\mc{O}\in\lqt{\oGr_{n-1}\lt(\bb{R}^n\rt)}{\mc{S}(\si_0)}\) the stabiliser of some (equivalently any) \(\ta \in \mc{T}_{\si_0}(\mc{O})\) is path-connected.

\item \(\si_0\) is termed abundant if for all \(\mc{O} \in \lqt{\oGr_{n-1}\lt(\bb{R}^n\rt)}{\mc{S}(\si_0)}\) and some (equivalently any) \(\ta\in\mc{T}_{\si_0}(\mc{O})\):
\ew
0 \in \Conv~\mc{N}_{\si_0}(\ta) \cc \ww{p-1}\lt(\bb{R}^{n-1}\rt)^*.
\eew
\end{enumerate}
\end{Defn}

In words, \(\si_0\) is faithful if for all oriented hyperplanes \(\bb{B}\), \(\bb{B}' \pc \bb{R}^n\) the restrictions \(\si_0|_\bb{B}\) and \(\si_0|_{\bb{B}'}\) are isomorphic \iff\ \(\bb{B}\) and \(\bb{B}'\) lie in the same orbit of \(\mc{S}(\si_0)\), \(\si_0\) is connected if for every oriented hyperplane \(\bb{B} \pc \bb{R}^n\) the stabiliser of \(\si_0|_\bb{B}\) in \(\GL_+(\bb{B})\) is connected, and \(\si_0\) is abundant if for every \(\ta \in \ww{p}\lt(\bb{R}^{n-1}\rt)^*\) either \(\mc{N}_{\si_0}(\ta)\) is empty or the convex hull of \(\mc{N}_{\si_0}(\ta)\) contains 0.

Verifying the above three properties, in practice, is greatly helped by the following three results:
\begin{Prop}\label{open-to-open}
Let \(\si_0 \in \ww{p}\lt(\bb{R}^n\rt)^*\) be stable and equip the spaces \(\lqt{\oGr_{n-1}\lt(\bb{R}^n\rt)}{\mc{S}(\si_0)}\) and \(\rqt{\ww{p}\lt(\bb{R}^{n-1}\rt)^*}{\GL_+(n-1,\bb{R})}\) with their natural quotient topologies.  Then, \(\mc{T}_{\si_0}\) is an open map.  In particular, if \(\mc{O} \in \lqt{\oGr_{n-1}\lt(\bb{R}^n\rt)}{\mc{S}(\si_0)}\) is an open orbit, then \(\mc{T}_{\si_0}(\mc{O})\) is, also, an open orbit, i.e.\ the orbit of a stable p-form on \(\bb{R}^{n-1}\).
\end{Prop}

\begin{proof}
Consider the commutative diagram:
\ew
\bcd
Emb\lt(\bb{R}^{n-1},\bb{R}^n\rt) \ar[r, "\cal{T}_{\si_0}"] \ar[d, "quot"] & \ww{p}\lt(\bb{R}^{n-1}\rt)^* \ar[d, "quot"]\\
\lqt{\oGr_{n-1}\lt(\bb{R}^n\rt)}{\mc{S}(\si_0)} \ar[r, "\mc{T}_{\si_0}"] & \rqt{\ww{p}\lt(\bb{R}^{n-1}\rt)^*}{\GL_+(n-1,\bb{R})}
\ecd
\eew
Since the left-hand map is a continuous surjection and the right-hand map is open, to prove \(\mc{T}_{\si_0}\) is open it suffices to prove that \(\cal{T}_{\si_0}\) is open.  To this end, let \(\io:\bb{R}^{n-1} \to \bb{R}^n\) be an embedding and fix a splitting \(\be\) of the exact sequence:
\ew
\bcd
\ww{p}\lt(\bb{R}^n\rt)^* \ar[r, "\io^*", shift left = 3pt] & \ww{p}\lt(\bb{R}^{n-1}\rt)^* \ar[l, "\be", shift left = 3pt] \ar[r] & 0.
\ecd
\eew
Since \(\si_0\) is stable, for all \(\ta \in \ww{p}\lt(\bb{R}^{n-1}\rt)^*\) sufficiently small there is \(F \in \GL_+(n;\bb{R})\) (close to \(\Id\)) such that:
\ew
F^*\si_0 = \si_0 + \be(\ta).
\eew
Set \(\io' = F \circ \io \in Emb\lt(\bb{R}^{n-1},\bb{R}^n\rt)\).  For \(\ta\) small enough, \(\io'\) can be taken arbitrarily close to \(\io\).  Moreover:
\ew
\io'^*\si_0 = \io^*F^*\si_0 = \io^*(\si_0 + \be(\ta)) = \io^*\si_0 + \ta,
\eew
since \(\io^* \circ \be = \Id\).  Thus, \(\cal{T}_{\si_0}\) is an open map and the result follows.

\end{proof}

\begin{Prop}\label{tr->conn}
Let \(\si_0\) be a stable \(p\)-form on \(\bb{R}^n\) and suppose that \(\mc{S}(\si_0)\) acts transitively on \(\oGr_{n-1}\lt(\bb{R}^n\rt)\).  Then, \(\si_0\) is faithful and connected.  In particular, if \(\mc{S}(\si_0)\) contains a subgroup which (i) preserves a positive definite inner product and (ii) acts transitively on the corresponding unit sphere in \(\bb{R}^n\) (equivalently, in \(\lt(\bb{R}^n\rt)^*\)), then \(\si_0\) is faithful and connected.
\end{Prop}

\begin{proof}
Faithfulness is clear.  Since \(\mc{S}(\si_0)\) acts transitively on \(\oGr_{n-1}\lt(\bb{R}^n\rt)\), the unique orbit must be open and hence, by \pref{open-to-open}, it must map under \(\mc{T}_{\si_0}\) to the orbit of a stable form.  However, using the final statement from \tref{stab-clas}, the stabiliser in \(\GL_+(n-1;\bb{R})\) of every stable form on \(\bb{R}^{n-1}\) is connected.  The final statement follows since \(\oGr_{n-1}\lt(\bb{R}^n\rt)\) is isomorphic to the unit sphere in \(\lt(\bb{R}^n\rt)^*\).

\end{proof}

\begin{Prop}\label{Abt-Lem}
If there exists an orientation-reversing automorphism of \(\bb{R}^n\) which preserves \(\si_0\), then \(\si_0\) is abundant. In particular, if \(n = 2k+1\) is odd and \(2 \le p \le 2k\) is even, then any \(\si_0\in\ww{p}\lt(\bb{R}^n\rt)^*\) is abundant.
\end{Prop}

\begin{proof}
Fix \(\ta\in\ww{p}\lt(\bb{R}^{n-1}\rt)^*\) and suppose \(\mc{N}_{\si_0}(\ta) \ne \es\).  Choose some \(\nu\in\mc{N}_{\si_0}(\ta)\).  I claim that \(-\nu\) also lies in \(\mc{N}_{\si_0}(\ta)\).  Indeed, since \(\th \w \nu + \ta \in \ww{p}\lt(\bb{R} \ds \bb{R}^{n-1}\rt)\) is a \(\si_0\)-form, by assumption there exists an orientation-reversing map \(F \in \GL(\bb{R} \ds \bb{R}^{n-1})\) preserving \(\th \w \nu + \ta\).  Now consider the map:
\ew
\bcd[row sep = 0]
\mc{I}:\bb{R}\ds \bb{R}^{n-1} \ar[r] & \bb{R}\ds \bb{R}^{n-1}\\
v\ds w \ar[r, maps to] & -v \ds w.
\ecd
\eew
Since \(\mc{I}\) is also orientation-reversing, the composite \(F \circ \mc{I}\) is orientation-preserving and thus \((F \circ \mc{I})^*(\th \w \nu + \ta)\) is a \(\si_0\)-form.  However:
\ew
(F \circ \mc{I})^*(\th \w \nu + \ta) &= \mc{I}^*(F)^*(\th \w \nu + \ta)\\
&= \mc{I}^*(\th \w \nu + \ta)\\
&= -\th \w \nu + \ta
\eew
and thus \(-\nu \in \mc{N}_{\si_0}(\ta)\) as claimed. The proof is completed by noting that if \(n=2k+1\) and \(2 \le p \le 2k\) is even, then \(-\Id\) is an orientation-reversing automorphism preserving \(\si_0\).

\end{proof}

The significance of the notions of faithfulness, connectedness and abundance lies in the following result:

\begin{Thm}\label{FCA->A}
Let \(\si_0 \in \ww{p}\lt(\bb{R}^n\rt)^*\) be stable, faithful, connected and abundant.  Then, \(\si_0\) is ample; in particular, \(\si_0\)-forms satisfy the \(h\)-principle.
\end{Thm}

\begin{Rk}
The converse need not hold: see \sref{ossymp-hP}.
\end{Rk}

The proof proceeds by a series of lemmas.
\begin{Lem}\label{Am-Lem}
Let \(\bb{A}\) be a (real) finite-dimensional vector space and let \(A\cc\bb{A}\) be a path-connected, open subset such that:
\begin{itemize}
\item \(0\in\Conv(A)\);
\item \(A\) is scale-invariant, i.e.\ for all \(\la\in(0,\infty)\), \(\la\cdot A = A\).
\end{itemize}
Then, \(\Conv(A) = \bb{A}\), i.e.\ \(A\) is ample.
\end{Lem}

\begin{proof}
Since \(A\) is open and scale-invariant, so too is \(\Conv(A)\). However, by assumption \(\Conv(A)\) contains \(0\), and thus by openness it contains a small open ball about \(0\) in \(\bb{A}\). The scale-invariance of \(\Conv{A}\) then implies that \(\Conv(A) = \bb{A}\).

\end{proof}

\begin{Lem}\label{Scale-Inv}
Let \(\si_0 \in \ww{p}\lt(\bb{R}^n\rt)\), \(\ta \in \ww{p}\lt(\bb{R}^{n-1}\rt)\) and suppose \(\mc{N}_{\si_0}(\ta) \ne \es\).  Then, \(\mc{N}_{\si_0}(\ta)\) is scale-invariant.
\end{Lem}

\begin{proof}
Suppose \(\nu\in\mc{N}_{\si_0}(\ta)\), i.e.\ \(\th \w \nu + \ta\in\ww{p}\lt(\bb{R}\ds\bb{R}^{n-1}\rt)^*\) is a \(\si_0\)-form.  Consider the orientation-preserving isomorphism:
\ew
\bcd[row sep = 0]
F:\bb{R}\ds\bb{R}^{n-1} \ar[r] & \hs{1.75mm}\bb{R}\hs{1.75mm}\ds\bb{R}^{n-1}\\
\hs{7.5mm}v\ds w \ar[r, maps to] & \la v \ds w.
\ecd
\eew
Then, \(F^*\si = \th\w(\la\nu) + \ta\) is a \(\si_0\)-form, as required.

\end{proof}

\begin{Lem}\label{PConn}
Let \(\si_0 \in \ww{p}\lt(\bb{R}^n\rt)^*\) and suppose that \(\mc{O} \in \lqt{\oGr_{n-1}(\bb{R}^n)}{\mc{S}(\si_0)}\) satisfies \(\mc{T}_{\si_0}^{-1}(\{\mc{T}_{\si_0}(\mc{O})\}) = \{\mc{O}\}\).  Suppose, moreover, that the stabiliser in \(\GL_+(n-1;\bb{R})\) of some (equivalently every) \(\ta \in \mc{T}_{\si_0}(\mc{O})\) is path-connected. Then, for all \(\ta \in \mc{T}_{\si_0}(\mc{O})\), the space \(\mc{N}_{\si_0}(\ta) \pc \ww{p-1}\lt(\bb{R}^{n-1}\rt)^*\) is path-connected.  In particular, if \(\si_0\) is faithful and connected, then for every \(\ta \in \ww{p}\lt(\bb{R}^{n-1}\rt)^*\), either \(\mc{N}_{\si_0}(\ta) = \es\) or \(\mc{N}_{\si_0}(\ta)\) is path-connected.
\end{Lem}

\begin{proof}
Let \(\mc{O}\) be as in the statement of the lemma, let \(\ta \in \mc{T}_{\si_0}(\mc{O})\) and let \(\nu_1,\nu_2\in\mc{N}_{\si_0}(\ta)\). Then, by definition, the two \(p\)-forms:
\ew
\si_i = \th\w\nu_i + \ta \in \ww{p}\lt(\bb{R}\ds\bb{R}^{n-1}\rt)^*, ~ i=1,2
\eew
are both \(\si_0\)-forms on \(\bb{R}\ds\bb{R}^{n-1}\). Thus, there is \(F\in\GL_+\lt(\bb{R}\ds\bb{R}^{n-1}\rt)\) such that \(F^*\si_2 = \si_1\).

\begin{Cl}
The oriented hyperplanes \(\bb{R}^{n-1}\) and \(F\lt(\bb{R}^{n-1}\rt)\) in \(\bb{R}\ds\bb{R}^{n-1}\) lie in the same orbit of the stabiliser of \(\si_2\).
\end{Cl}

\begin{proof}[Proof of Claim]
Since \(\si_2\) is a \(\si_0\)-form, there is an oriented isomorphism \(\cal{I}:\bb{R}\ds\bb{R}^{n-1}\to\bb{R}^n\) such that \(\cal{I}^*\si_0 = \si_2\).  Hence, it is equivalent to prove that the oriented hyperplanes \(\cal{I}\lt(\bb{R}^{n-1}\rt)\) and \(\cal{I}\circ F\lt(\bb{R}^{n-1}\rt)\) in \(\bb{R}^n\) lie in the same \(\mc{S}(\si_0)\)-orbit.  Now, consider the commutative diagram:
\ew
\bcd
Emb\lt(\bb{R}^{n-1},\bb{R}^n\rt) \ar[r, "\cal{T}_{\si_0}"] \ar[d, "quot"] & \ww{p}\lt(\bb{R}^{n-1}\rt)^* \ar[d, "quot"]\\
\lqt{\oGr_{n-1}\lt(\bb{R}^n\rt)}{\mc{S}(\si_0)} \ar[r, "\mc{T}_{\si_0}"] & \rqt{\ww{p}\lt(\bb{R}^{n-1}\rt)^*}{\GL_+(n-1,\bb{R})}
\ecd
\eew
Since \(\mc{T}_{\si_0}^{-1}(\{\mc{T}_{\si_0}(\mc{O})\}) = \{\mc{O}\}\), it suffices to prove that both \(\cal{T}_{\si_0}\lt(\cal{I}|_{\bb{R}^{n-1}}\rt)\) and \(\cal{T}_{\si_0}\lt((\cal{I}\circ F)|_{\bb{R}^{n-1}}\rt)\) lie in the orbit \(\mc{T}_{\si_0}(\mc{O}) \in \rqt{\ww{p}\lt(\bb{R}^{n-1}\rt)^*}{\GL_+(n-1,\bb{R})}\), which is true since:
\ew
\lt(\cal{I}|_{\bb{R}^{n-1}}\rt)^*\si_0 = \si_2|_{\bb{R}^{n-1}} = \ta \in \mc{T}_{\si_0}(\mc{O})
\eew
and
\ew
\lt((\cal{I}\circ F)|_{\bb{R}^{n-1}}\rt)^*\si_0 = (F^*\si_2)|_{\bb{R}^{n-1}} = \si_1|_{\bb{R}^{n-1}} = \ta \in \mc{T}_{\si_0}(\mc{O}),
\eew
completing the proof of the claim.

\let\qed\relax
\end{proof}

Thus, choose \(G\in\GL_+\lt(\bb{R}\ds\bb{R}^{n-1}\rt)\) stabilising \(\si_2\) such that \(G\circ F\lt(\bb{R}^{n-1}\rt) = \bb{R}^{n-1}\) (and \(G \circ F\) identifies the orientations). By replacing \(F\) with \(G \circ F\), one may assume \wlg\ that \(F\) fixes the space \(\bb{R}^{n-1}\). Then:
\ew
\ta = \lt.\si_1\rt|_{\bb{R}^{n-1}} = \lt.\lt(F^*\si_2\rt)\rt|_{\bb{R}^{n-1}} = \lt(F|_{\bb{R}^{n-1}}\rt)^*\lt.\si_2\rt|_{\bb{R}^{n-1}} = \lt(F|_{\bb{R}^{n-1}}\rt)^*\ta.
\eew
Thus, \(F\) lies in the space:
\ew
\cal{K} = \lt\{F'\in\GL_+\lt(\bb{R} \ds \bb{R}^{n-1}\rt)~\m|~F'\lt(\bb{R}^{n-1}\rt) = \bb{R}^{n-1} \text{ and } F'|_{\bb{R}^{n-1}} \in \Stab(\ta) \cc \GL_+(n-1,\bb{R})\rt\}.
\eew
Since \(\Stab(\ta)\pc\GL_+(n-1,\bb{R})\) is path-connected, so too is \(\cal{K}\pc\GL_+\lt(\bb{R} \ds \bb{R}^{n-1}\rt)\), so one can choose a smooth 1-parameter family \((F_t)_{t\in[0,1]}\in\cal{K}\) such that \(F_0 = \Id\), \(F_1 = F\).  Then, for each \(t\):
\ew
F_t^*\si_2 = \th \w \nu(t) + \ta
\eew
for some \(\nu(t) \in \mc{N}_{\si_0}(\ta)\) (note that \(F_t^*\si_2\) is evidently a \(\si_0\)-form for each \(t\)) such that \(\nu(1) = \nu_1\) and \(\nu(0) = \nu_2\).  Thus, \(\mc{N}_{\si_0}(\ta)\) is path-connected.

\end{proof}

I now prove \tref{FCA->A}:

\begin{proof}[Proof of \tref{FCA->A}]
Let \(\si_0 \in \ww{p}\lt(\bb{R}^n\rt)^*\) be stable, faithful, connected and abundant, let \(\ta \in \ww{p}\lt(\bb{R}^{n-1}\rt)^*\) and suppose \(\mc{N}_{\si_0}(\ta) \ne \es\).  Since \(\si_0\) is stable, \(\mc{N}_{\si_0}(\ta) \cc \ww{p-1}\lt(\bb{R}^{n-1}\rt)^*\) is open.  Moreover \(\mc{N}_{\si_0}(\ta)\) is scale invariant by \lref{Scale-Inv}, path-connected by \lref{PConn} and \(0 \in \Conv\lt(\mc{N}_{\si_0}(\ta)\rt)\) since \(\ta\) is abundant.  Hence \(\mc{N}_{\si_0}(\ta)\) is ample by \lref{Am-Lem}.

\end{proof}

\section{Initial applications: \g\ 4-forms, \slc\ 3-forms and pseudoplectic forms}\label{1st-Appl}

Before proving \tref{4hP} in \srefs{ossymp-hP} and \ref{SG2hP}, this section illustrates the results of \srefs{ample-forms} and \ref{FCA} in some simple examples, by providing new, unified proofs of the three previously established relative \(h\)-principles, {\it viz.}\ the relative \(h\)-principles for \g\ 4-forms \cite{NIoG2S}, \slc\ 3-forms \cite{RoG2MwB} and pseudoplectic forms \cite{AoCItS&CG}.\\

\subsection{\(\mb{\Gg_2}\) 4-Forms}

\begin{Thm}[{\cite[Thm.\ 5.3]{NIoG2S}}\footnote{\cite[Thm.\ 5.3]{NIoG2S} actually only states the non-relative version of the \(h\)-principle for \g\ 4-forms (corresponding to \(A = \es\), in the notation of the introduction), however the proof presented {\it op.\ cit.}\ generalises readily to the case \(A \ne \es\).}]
\g\ 4-forms satisfy the relative \(h\)-principle.
\end{Thm}

\begin{proof}[Proof using \tref{FCA->A}]
Let \(\si_0 = \vps_0\in\ww{4}(\bb{R}^7)^*\) be the standard \g\ 4-form (see \eref{vps0}) and write \(g_0\) for the corresponding metric.  Since \(\mc{S}\lt(\vps_0\rt) = \Gg_2\) preserves a positive definite inner product and acts transitively on \(S^6\), by \pref{tr->conn} \(\vps_0\) is faithful and connected.  Moreover, since \(\vps_0\) has even degree on an odd-dimensional space, \(\vps_0\) is abundant by \pref{Abt-Lem}.  Thus, the result follows by \tref{FCA->A}.

\end{proof}

\subsection{\(\mb{\SL(3;\bb{C})}\) 3-forms}

\begin{Thm}[{\cite[\S4]{RoG2MwB}}]
\slc\ 3-forms satisfy the relative \(h\)-principle.
\end{Thm}

\begin{proof}[Proof using \tref{FCA->A}]
Let \(\si_0 = \rh_- \in \ww{3}\lt(\bb{R}^6\rt)^*\) be the standard \slc\ 3-form (see \eref{rh+}).  The stabiliser \(\SL(3;\bb{C})\) of \(\rh_-\) contains the subgroup \(\SU(3)\) which preserves a positive definite inner product and acts transitively on \(S^5\); thus \(\rh_-\) is faithful and connected by \pref{tr->conn}.  Moreover, by \pref{SL-Prop}, \(\rh_-\) is preserved by an orientation-reversing automorphism of \(\bb{R}^6\) and thus \(\rh_-\) is abundant by \lref{Abt-Lem}.  Thus, the result follows by \tref{FCA->A}.

\end{proof}

\subsection{Pseudoplectic forms}

\begin{Thm}[{\cite[Thm.\ 2.5]{AoCItS&CG}}\footnote{\cite[Thm.\ 2.5]{AoCItS&CG} actually only states the relative \(h\)-principle for pseudoplectic forms in the case \(q = 0\) (in the notation of the introduction), however the proof presented {\it op.\ cit.}\ generalises readily to the case \(q \ge 1\).}]\label{P-repr}
Pseudoplectic forms satisfy the relative \(h\)-principle.
\end{Thm}

\begin{proof}[Proof using \tref{FCA->A}]
Let \(\mu_0(k) = \th^{23} + \th^{45} + ... \th^{2k,2k+1}\) be the standard pseudoplectic form on \(\bb{R}^{2k+1}\) (see \eref{mu0}) and recall the 1-dimensional subspace \(\ell_{\mu_0(k)} = \<e_1\?\) defined in \rref{OP-Hyp}.  Given an oriented hyperplane \(\bb{B}\pc\bb{R}^{2k+1}\), on dimensional grounds, either \(\dim(\bb{B}\cap\ell_{\mu_0(k)}) = 1\), in which case \(\ell_{\mu_0(k)}\pc\bb{B}\) and \(\mu_0(k)|_\bb{B}\) is a degenerate bilinear form, or \(\dim(\bb{B} \cap \ell_{\mu_0(k)}) = 0\) and \(\bb{B}\) is transverse to \(\ell_{\mu_0(k)}\), in which case \(\mu_0(k)|_\bb{B}\) is either emproplectic or pisoplectic on \(\bb{B}\).  Hence, the image of the map:
\ew
\bcd
\mc{T}_{\mu_0(k)}: \lqt{\oGr_{2k}\lt(\bb{R}^{2k+1}\rt)}{\mc{S}(\mu_0(k))} \ar[r] & \rqt{\ww{2}\lt(\bb{R}^{2k}\rt)^*}{\GL_+(2k,\bb{R})}
\ecd
\eew
contains at least three distinct orbits and whence the action of \(\mc{S}(\mu_0(k))\) on \(\oGr_{2k}\lt(\bb{R}^{2k+1}\rt)\) has at least three orbits, also.

Therefore, to prove that \(\mu_0(k)\) is faithful, it suffices to prove that the action of \(\mc{S}(\mu_0(k))\) on \(\oGr_{2k}\lt(\bb{R}^{2k+1}\rt)\) has exactly three orbits.  Recall from \pref{P} that, \wrt\ the splitting \(\bb{R}^{2k+1} = \ell_{\mu_0(k)} \ds \<e_2,...,e_{2k+1}\?\), the stabiliser \(\mc{S}(\mu_0(k))\) consists precisely of those \((2k+1)\x(2k+1)\)-matrices of the form:
\ew
\bpm
\la & G_{2k \x 1}\\
0_{1 \x 2k} & F_{2k\x2k}
\epm
\eew
where \(F\in\Sp(2k;\bb{R})\) and \(\la>0\).  Next, note that oriented hyperplanes in \(\bb{R}^{2k+1}\) containing \(\ell_{\mu_0(k)}\) are in 1-1 correspondence with oriented hyperplanes in \(\<e_2,...,e_{2k+1}\?\).  Since \(\Sp(2k;\bb{R})\) contains the subgroup \(\SU(k)\), \(\Sp(2k;\bb{R})\) acts transitively on oriented hyperplanes in \(\<e_2,...,e_{2k+1}\?\) (cf.\ \pref{tr->conn}) and thus \(\mc{S}(\mu_0(k))\) acts transitively on the set of oriented hyperplanes in \(\bb{R}^{2k+1}\) containing \(\ell_{\mu_0(k)}\).  Similarly, the (unoriented) hyperplanes in \(\bb{R}^{2k+1}\) transverse to \(\ell_{\mu_0(k)}\) are in bijective correspondence with linear maps \(\<e_2,...,e_{2k+1}\?\to\ell_{\mu_0(k)}\) and hence form a single orbit for the action of \(\mc{S}(\mu_0(k))\). Thus, the action of \(\mc{S}(\mu_0(k))\) on oriented hyperplanes in \(\bb{R}^{2k+1}\) transverse to \(\ell_{\mu_0(k)}\) has at most 2 orbits.  It follows that the action of \(\mc{S}(\mu_0(k))\) on \(\oGr_{2k}\lt(\bb{R}^{2k+1}\rt)\) has exactly three orbits and \(\mu_0(k)\) is faithful.  To see that \(\mu_0(k)\) is also connected, firstly note that the stabiliser of both emproplectic and pisoplectic forms on \(\bb{R}^{2k}\) is connected, being isomorphic to \(\Sp(2k;\bb{R})\).  For the remaining case, suppose \(\ell_{\mu_0(k)} \pc \bb{B}\); then the restriction \(\mu_0(k)|_\bb{B}\) may be written in some basis \((f_1,...,f_{2k})\) as:
\ew
\mu_0(k)|_\bb{B} = f^{12} + ... + f^{2k-3,2k-2}.
\eew
This is an emproplectic form on \(\<f_1,...,f_{2k-2}\?\).  Splitting \(\bb{B} \cong \<f_1,...,f_{2k-2}\? \ds \<f_{2k-1},f_{2k}\?\) and applying \lref{stabiliser-comp-lem}, it follows that the stabiliser of \(\mu_0(k)|_\bb{B}\) in \(\GL_+(\bb{B})\) is connected. Thus, \(\mu_0(k)\) is connected.  Finally, since pseudoplectic forms constitute a single \(\GL(2k+1;\bb{R})\)-orbit, it follows from \lref{Abt-Lem} that pseudoplectic forms are abundant.  Thus, the result follows by \tref{FCA->A}.

\end{proof}

\begin{Rk}
Crowley--Nordstr\"{o}m and Donaldson used a technique known as `Hirsch's microextension trick' (after its use by Hirsch in \cite{IoM}) to prove their \(h\)-principles for \g\ 4-forms and \slc\ 3-forms, respectively.  E.g.\ For \g\ 4-forms, the argument may be sketched as follows:  given any 8-manifold \(\N\), define a subset \(\cal{S}(\N)\pc\ww{4}\T^*\N\) by declaring \(\al\in\ww{4}\T^*_p\N\) to lie in \(\cal{S}(\N)\) \iff\ the restriction of \(\al\) to every hyperplane in \(\T_p\N\) is a \g\ 4-form for some choice (equivalently, both choices) of orientation of the hyperplane. Then, \(\cal{S}(\N)\) is an open, \(\Diff_0(\N)\)-invariant subbundle of \(\ww{4}\T^*\N\).  Given an oriented 7-manifold \(\M\), it can be shown that every \g\ 4-form \(\ps\) on \(\M\) extends to a 4-form \(\Ps\) on the open (i.e.\ non-closed) manifold \((-\ep,\ep)\x\M\) (for some \(\ep>0\) sufficiently small) such that \(\Ps \in \cal{S}\big((-\ep,\ep)\x\M\big)\). The \(h\)-principle for \g\ 4-forms then follows from Gromov's \(h\)-principle for open, diffeomorphism-invariant relations on open manifolds; cf.\ \cite[Thm.\ 10.2.1]{ItthP}.  However, this method, whilst simple, is limited in scope, since for a general stable \(p\)-form \(\si_0\) on \(\bb{R}^n\), there are no \(p\)-forms \(\si\) on \(\bb{R}^{n+1}\) such that the restriction \(\si|_\bb{A}\) is a \(\si_0\)-form for every hyperplane \(\bb{A}\pc\bb{R}^{n+1}\) for some choice of orientation on \(\bb{A}\).  As a simple example of this phenomenon, suppose that \(\Stab_{\GL_+(n;\bb{R})}(\si_0) = \Stab_{\GL(n;\bb{R})}(\si_0)\), i.e.\ \(\si_0\) has no orientation-reversing automorphisms.  If there were some \(\si\in\ww{p}\lt(\bb{R}^{n+1}\rt)^*\) such that for all \(\bb{A}\in\Gr_n\lt(\bb{R}^{n+1}\rt)\) the restriction \(\si|_\bb{A}\) was a \(\si_0\)-form for some choice of orientation on \(\bb{A}\), then this choice of orientation would be unique (since the stabilisers of \(\si_0\) in \(\GL_+(n;\bb{R})\) and \(\GL(n;\bb{R})\) coincide) and would thus define a section of the `forgetful' degree 2 covering map \(\oGr_n\lt(\bb{R}^{n+1}\rt) \to \Gr_n\lt(\bb{R}^{n+1}\rt)\), yielding a contradiction, as claimed.

By contrast, the techniques introduced in this paper can be used to prove \(h\)-principles for stable forms \(\si_0\) satisfying \(\Stab_{\GL_+(n;\bb{R})}(\si_0) = \Stab_{\GL(n;\bb{R})}(\si_0)\).  Indeed, this property is satisfied by both \sg\ 3-forms and osemproplectic forms in dimension \(2k\), where \(k\) is odd, both of which are shown to satisfy the \(h\)-principle in this paper.\\
\end{Rk}

\section{Osemproplectic and ospseudoplectic forms}\label{ossymp-hP}

The aim of this section is to prove the relative \(h\)-principles for osemproplectic and ospseudoplectic forms.  The proofs proceed via a series of lemmas:

\begin{Lem}\label{Symp-N}
Let \(k \ge 2\) and let \(\om_+(k) \in \ww[+]{2}\lt(\bb{R}^{2k}\rt)^*\) be the standard emproplectic form on \(\bb{R}^{2k}\) defined in \eref{om/pm}.  Identify \(\bb{R}^{2k} \cong \bb{R} \ds \bb{R}^{2k-1}\) in the usual way and fix \(\ta \in \ww{2}\lt(\bb{R}^{2k-1}\rt)^*\).  Then, \(\mc{N}_{\om_+(k)}(\ta) \ne \es\) \iff\ \(\ta\) is pseudoplectic.  Moreover, in this case:
\ew
\mc{N}_{\om_+(k)}(\ta) = \lt\{ \nu \in \ww{1}\lt(\bb{R}^{2k-1}\rt)^* ~\m|~ \nu|_{\ell_{\ta}} > 0\rt\}
\eew
(see \rref{OP-Hyp} for the definition of \(\ell_\ta\)).
\end{Lem}

\begin{proof}
Write \(\th\) for the standard annihilator of \(\bb{R}^{2k-1}\) in \(\bb{R}^{2k}\)  and define \(\om = \th \w \nu + \ta\).  Then:
\ew
\om^k = \th \w \nu \w \ta^{k-1}.
\eew
Thus, \(\om\) is emproplectic \iff\ \(\ta^{k-1} \ne 0\) (i.e.\ \(\ta\) is pseudoplectic) and \(\nu|_{\ell_\ta} > 0\) (by definition of the choice of orientation on \(\ell_\ta\)).

\end{proof}

\begin{Cor}\label{Cosymp-N}
Let \(k \ge 3\), let \(\vpi_+(k) \in \ww[+]{2k-2}\lt(\bb{R}^{2k}\rt)^*\) be the standard osemproplectic form on \(\bb{R}^{2k}\) defined in \eref{vpi+}, identify \(\bb{R}^{2k} \cong \bb{R} \ds \bb{R}^{2k-1}\) and fix \(\ta \in \ww{2k-2}\lt(\bb{R}^{2k-1}\rt)^*\).  Then:
\ew
\mc{N}_{\vpi_+(k)}(\ta) = \lt\{ \nu \in \ww{2k-3}\lt(\bb{R}^{2k-1}\rt)^* ~\m|~ \nu \text{ is ospseudoplectic and } \ta|_{\Pi_\nu} > 0\rt\}.
\eew
(See \rref{OP-Hyp} for the definition of the hyperplane \(\Pi_\nu\).)  In particular, if \(\ta = 0\) then \(\mc{N}_{\vpi_+(k)}(\ta) = \es\).
\end{Cor}

\begin{proof}
The proof uses the duality described in \lref{no-orb-lem}.  Write \(e = 1 \ds 0 \in \bb{R} \ds \bb{R}^{2k-1}\), choose \(\si > 0 \in \ww{2k-1}\bb{R}^{2k-1}\) and set \(\up = e \w \si > 0 \in \ww{2k}\bb{R}^{2k}\).  Then, as in \lref{no-orb-lem}, \(\th \w \nu + \ta\) is osemproplectic \iff:
\ew
(\th \w \nu + \ta) \hk \up = \nu \hk \si + e \w (\ta \hk \si)
\eew
is emproplectic on \(\lt(\bb{R}^{2k}\rt)^*\), which by \lref{Symp-N} is equivalent to \(\nu \hk \si\) being pseudoplectic on \(\lt(\bb{R}^{2k-1}\rt)^*\) and \((\ta \hk \si)|_{\ell_{\nu \hk \si}} > 0\).  However, using the duality as in \lref{no-orb-lem} again, \(\nu \hk \si\) is pseudoplectic on \(\lt(\bb{R}^{2k-1}\rt)^*\) \iff\ \(\nu\) is ospseudoplectic on \(\bb{R}^{2k-1}\).  Moreover, one may verify that:
\ew
\ell_{\nu \hk \si} = \Ann(\Pi_\nu)
\eew
compatibly with orientations, and thus \((\ta \hk \si)|_{\ell_{\nu \hk \si}} > 0\) \iff\ \(\ta|_{\Pi_\nu} > 0\).

\end{proof}

\begin{Lem}\label{ossymp-set-ampl}
Let \(k \ge 2\).  Then:
\ew
\Conv \lt(\ww[+]{2k-2}\lt(\bb{R}^{2k}\rt)^*\rt) = \ww{2k-2}\lt(\bb{R}^{2k}\rt)^*.
\eew
(Note that when \(k = 2\), \(\ww[+]{2k-2}\lt(\bb{R}^{2k}\rt)^*\) is simply the orbit of emproplectic 2-forms.)
\end{Lem}

\begin{proof}
Since \(\ww[+]{2k-2}\lt(\bb{R}^{2k}\rt)^* \pc \ww{2k-2}\lt(\bb{R}^{2k}\rt)^*\) is open, path-connected and scale-invariant, by \lref{Am-Lem}, it suffices to prove that:
\ew
0 \in \Conv \lt(\ww[+]{2k-2}\lt(\bb{R}^{2k}\rt)^*\rt) = \ww{2k-2}\lt(\bb{R}^{2k}\rt)^*.
\eew
Write \(\lt(\th^1,...,\th^{2k}\rt)\) for the canonical basis of \(\lt(\bb{R}^{2k}\rt)^*\) and recall:
\ew
\vpi_+(k) = \sum_{i = 1}^k \th^{12...\h{2i-1,2i}...2k-1,2k}.
\eew
(Recall also that formally \(\vpi_+(2) = \om_+(2)\).)  Choose \(r \ge 1\) such that \(r + \frac{1}{r} = k \ge 2\).  For each ordered pair of distinct \(p,q \in \{1,...,k\}\), let \(F_r(p,q)\) denote the orientation-preserving automorphism of \(\bb{R}^{2k}\) given by:
\ew
F_r(p,q) =&\
\lt(
\begin{array}{ccccccc}
1 &             &     &            &                   &             &   \\
   & \ddots &    &             &                   &             &   \\
   &            & -r &             &                   &             &   \\
   &            &     & \ddots &                   &             &   \\
   &            &     &             & -\frac{1}{r} &            &   \\
   &            &     &             &                   & \ddots &   \\
   &            &     &             &                   & \ddots & 1\\
\end{array}
\rt)
\begin{array}{c}
\\
\\
\leftarrow 2p^\text{th} \text{ row}\\
\\
\leftarrow 2q^\text{th} \text{ row}\\
\\
\\
\end{array}\\
& \hs{-1mm}
\begin{array}{ccccccc}
 & & & \uparrow                    & \uparrow                    & &\\
 & & & 2p^\text{th} \text{ col.} & 2q^\text{th} \text{ col.} & &\\
\end{array}
\eew
Then, for all \(p\), \(q\) and \(r\):
\ew
F_r(p,q)^*\vpi_+(k) \in \ww[+]{2k-2}\lt(\bb{R}^{2k}\rt)^*
\eew
and thus:
\ew
2(k-1) \vpi_+(k) + \sum_{p \ne q \in \{1,...,k\}} F_r(p,q)^*\vpi_+(k) \in \Conv \lt(\ww[+]{2k-2}\lt(\bb{R}^{2k}\rt)^*\rt).
\eew
(N.B. the coefficients in the linear-combination on the left-hand side of the above expression are all positive and thus, even though they do not sum to 1, the above expression is valid since \(\ww[+]{2k-2}\lt(\bb{R}^{2k}\rt)^*\) is scale-invariant.)  A direct calculation then yields:
\ew
2(k-1) \vpi_+(k) + \sum_{p \ne q \in \{1,...,k\}} F_r(p,q)^*\vpi_+(k) = (k-1)\lt(k - r - \frac{1}{r}\rt)\vpi_+(k) = 0,
\eew
as required.

\end{proof}

\begin{Thm}\label{os}
Ossymplectic forms satisfy the relative \(h\)-principle.
\end{Thm}

\begin{proof}
Recall that the stabiliser of \(\vpi_+(k)\) in \(\GL_+(2k;\bb{R})\) is isomorphic to \(\Sp(2k;\bb{R})\).  Since \(\SU(k) \pc \Sp(2k;\bb{R})\), it follows by \pref{tr->conn} that \(\vpi_+(k)\) is faithful and connected.  Thus, it suffices to prove that \(\vpi_+(k)\) is abundant.  When \(k\) is even, this follows from \lref{Abt-Lem} since \(\vpi_+(k)\) admits an orientation-reversing automorphism (see \pref{dual-symp}).  However, in general, one must prove directly that \(\vpi_+(k)\) is abundant, i.e.\ (by \cref{Cosymp-N}) that 0 lies in the convex hull of the set:
\ew
\mc{N}_{\vpi_+(k)}(\ta) = \lt\{ \nu \in \ww{2k-3}\lt(\bb{R}^{2k-1}\rt)^* ~\m|~ \nu \text{ is ospseudoplectic and } \ta|_{\Pi_\nu} > 0\rt\}.
\eew

Choose a correctly-oriented basis \((e_2,...,e_{2k})\) of \(\bb{R}^{2k-1}\) with dual basis \(\lt(\th^2,...,\th^{2k}\rt)\) such that \(\ta = \th^{34...2k}\).  Then, for all \(\vpi \in \ww[+]{2k-4}\lt(\bb{R}^{2k-2}\rt)\), observe that \(\nu = \th^2 \w \vpi \in \mc{N}_{\vpi_+(k)}(\ta)\) and thus:
\ew
\th^2 \w \ww[+]{2k-4}\lt(\bb{R}^{2k-2}\rt)^* \cc \mc{N}_{\vpi_+(k)}(\ta).
\eew
However, by \lref{ossymp-set-ampl}, \(0 \in \Conv\lt(\ww[+]{2k-4}\lt(\bb{R}^{2k-2}\rt)^*\rt)\), completing the proof.

\end{proof}

Now, fix \(k \ge 2\) and consider the standard ospseudoplectic form \(\xi_0(k) = \th^1 \w \sum_{i = 1}^k \th^{23...\h{2i,2i+1}...2k,2k+1} \in \ww[OP]{2k-1}\lt(\bb{R}^{2k+1}\rt)^*\).  \tref{FCA->A} does not apply to ospseudoplectic forms, since \(\xi_0(k)\) is not faithful.  Indeed, recall from \rref{OP-Hyp} that \(\xi_0(k)\) canonically defines an oriented hyperplane \(\Pi_{\xi_0(k)} = \<e_2,...,e_{2k}\?\).  Then, both \(\lt\{\Pi_{\xi_0(k)}\rt\}\) and \(\lt\{\ol{\Pi}_{\xi_0(k)}\rt\}\) (where the overline denotes orientation-reversal) form singleton orbits for the action of \(\mc{S}(\xi_0(k))\) on \(\oGr_{2k}\lt(\bb{R}^{2k+1}\rt)\), however:
\ew
\xi_0(k)|_{\Pi_{\xi_0(k)}} = \xi_0(k)|_{\ol{\Pi}_{\xi_0(k)}} = 0.
\eew
Despite this observation, \tref{hPThm-1} does apply to ospseudoplectic forms:

\begin{Thm}\label{op}
\(\xi_0(k)\) is ample; hence ospseudoplectic forms satisfy the relative \(h\)-principle.
\end{Thm}

\begin{proof}
Write \(\oGr_{2k}\lt(\bb{R}^{2k+1}\rt)_{gen} = \oGr_{2k}\lt(\bb{R}^{2k+1}\rt) \osr \lt\{\Pi_{\xi_0(k)}, \ol{\Pi}_{\xi_0(k)}\rt\}\).  Then, \(\mc{S}(\xi_0(k))\) acts transitively on the space \(\oGr_{2k}\lt(\bb{R}^{2k+1}\rt)_{gen}\).  Indeed, let \(\Pi \in \oGr_{2k}\lt(\bb{R}^{2k+1}\rt)_{gen}\).  \(\Pi\) intersects \(\Pi_{\xi_0(k)}\) transversely and thus  \(\Pi \cap \Pi_{\xi_0(k)}\) has dimension \(2k-1\).  Moreover, \(\Pi \cap \Pi_{\xi_0(k)}\) can be canonically oriented as follows: choose any \(u \in \Pi\) such that \(\th^1(u) > 0\).  Then, the chosen orientation on \(\Pi\) together with the decomposition:
\ew
\Pi = \<u\? \ds \lt(\Pi \cap \Pi_{\xi_0(k)}\rt)
\eew
orients \(\Pi \cap \Pi_{\xi_0(k)}\) and thus \(\Pi \cap \Pi_{\xi_0(k)}\) defines an element of \(\oGr_{2k-1}\lt(\Pi_{\xi_0(k)}\rt)\).  Since \(\Sp(2k;\bb{R})\) acts transitively on \(\oGr_{2k-1}\lt(\bb{R}^{2k}\rt)\), by \eref{OP-stab} it follows that \(\mc{S}(\xi_0(k))\) acts transitively on \(\oGr_{2k-1}\lt(\Pi_{\xi_0(k)}\rt)\) and thus \wlg\ one may assume that:
\ew
\Pi \cap \Pi_{\xi_0(k)} = \<e_3,...,e_{2k+1}\?,
\eew
compatibly with its orientation.  Thus:
\ew
\Pi = \<e_1 + te_2, e_3,...,e_{2k+1}\?
\eew
for some \(t \in \bb{R}\).  Now, consider the automorphism of \(\bb{R}^{2k+1}\) given by:
\ew
F=
\lt(\begin{array}{c|ccc}
1 &    &             &  \\\hline
-t & 1 &             &  \\
   &    & \ddots &   \\
   &    &             & 1\\
\end{array}\rt).
\eew
By examining \eref{OP-stab}, \(F \in \mc{S}(\xi_0(k))\) and clearly \(F(\Pi) = \<e_1,e_3,...,e_{2k+1}\?\); thus \(\oGr_{2k}\lt(\bb{R}^{2k+1}\rt)_{gen}\) forms a single orbit, as claimed.  Moreover, \(\mc{T}_{\xi_0(k)}\lt(\oGr_{2k}\lt(\bb{R}^{2k+1}\rt)_{gen}\rt)\) is precisely the orbit of non-zero \((2k-1)\)-forms on \(\bb{R}^{2k}\).

Now, let \(\ta \in \mc{T}_{\xi_0(k)}\lt(\oGr_{2k}\lt(\bb{R}^{2k+1}\rt)_{gen}\rt)\).  Clearly, the stabiliser of \(\ta\) in \(\GL_+(2k;\bb{R})\) is connected.  Also, since:
\ew
\mc{T}_{\xi_0(k)}^{-1}\lt(\lt\{\mc{T}_{\xi_0(k)}\lt[\oGr_{2k}(\bb{R}^{2k+1})_{gen}\rt]\rt\}\rt) = \lt\{\oGr_{2k}(\bb{R}^{2k+1})_{gen}\rt\},
\eew
the set \(\mc{N}_{\xi_0(k)}(\ta)\) is path-connected by \lref{PConn}.  Moreover, since the \(\GL_+(2k+1;\bb{R})\)-orbit of ospseudoplectic forms is closed under the action of \(\GL(2k+1;\bb{R})\), by \lref{Abt-Lem} it follows that \(\xi_0(k)\) is abundant.  Thus, by \lref{Am-Lem}, \(\mc{N}_{\xi_0(k)}(\ta)\) is ample for all \(\ta \ne 0\).

Now, consider \(\ta = 0\).  Note that \(\th^1 \w \nu\) is ospseudoplectic \iff\ \(\nu\) is ossymplectic and thus:
\ew
\mc{N}_{\xi_0(k)}(0) = \ww[+]{2k-2}\lt(\bb{R}^{2k}\rt)^* \cup \ww[-]{2k-2}\lt(\bb{R}^{2k}\rt)^*.
\eew
This space has two path components and thus abundance alone is not sufficient to deduce that \(\mc{N}_{\xi_0(k)}(0)\) is ample.  However, by \lref{ossymp-set-ampl} the convex hull of each path component of \(\mc{N}_{\xi_0(k)}(0)\) equals \(\ww{2k-2}\lt(\bb{R}^{2k}\rt)^*\) and thus \(\xi_0(k)\) is ample, as claimed.

\end{proof}

\section{\sg\ 3- and 4-forms}\label{SG2hP}

The aim of this section is to prove the relative \(h\)-principles for \sg\ 3- and 4-forms.  Let \(\bb{R}^7\) have basis \((e_1,...,e_7)\) and dual basis \(\lt(\th^1,...,\th^7\rt)\) as usual.  Recall the standard \sg\ 3- and 4-forms defined in \erefs{svph0} and \eqref{vps0} respectively by:
\caw
\svph_0 = \th^{123} - \th^{145} - \th^{167} + \th^{246} - \th^{257} - \th^{347} - \th^{356};\\
\svps_0 = \th^{4567} - \th^{2367} - \th^{2345} + \th^{1357} - \th^{1346} - \th^{1256} - \th^{1247},
\caaw
inducing the metric \(\tld{g}_0 = \sum_{i=1}^3(\th^i)^{\ts2} - \sum_{i=4}^7(\th^i)^{\ts2}\) and volume form \(\th^{12...7}\).  For the purposes of calculations, it is advantageous to have additional `standard representatives' of \sg\ 3- and 4-forms:

\begin{Prop}
The 3-form:
\e\label{svph1}
\svph_1 = \frac{1}{2}\lt(\th^{147} + \th^{156} - \th^{237} + \th^{246} - \th^{345}\rt)
\ee
is of \sg-type.  It induces the metric and volume form:
\ew
\tld{g}_1 = -\th^1 \s \th^7 + \th^2 \s \th^6 - \th^3 \s \th^5 - \th^4 \s \th^4 \et vol_1 = \frac{1}{8}\th^{1234567}.
\eew
Moreover:
\ew
\svps_1 = \Hs_{\svph_1}\svph_1 = \frac{1}{4}\lt(\th^{2356} + 2\th^{2347} - 2\th^{1456} + \th^{1357} - \th^{1267}\rt).
\eew
\end{Prop}
\noindent(To prove this result, one simply calculates the bilinear form \(Q_{\svph_1} = \frac{1}{6}\lt(-\hk\svph_1\rt)^2\w\svph_1\) explicitly, from which the metric and volume form can simply be written down.)

Now consider the space \(\lqt{\oGr_6(\bb{R}^7)}{\tld{\Gg}_2}\).  Since \(\tld{g}_0\) is non-degenerate, taking orthocomplement \wrt\ \(\tld{g}_0\) establishes a \sg-equivariant isomorphism:
\e\label{OC}
\Gr_1(\bb{R}^7) &\cong \Gr_6(\bb{R}^7)\\
\bb{L} &\mt \bb{L}^\bot\\
\bb{B}^\bot &\mf \bb{B}.
\ee
Motivated by this, I term a hyperplane \(\bb{B} \pc \bb{R}^7\) spacelike, timelike or null according to whether its orthocomplement is spacelike, timelike or null and write \(\Gr_{6,+}(\bb{R}^7)\), \(\Gr_{6,-}(\bb{R}^7)\) and \(\Gr_{6,0}(\bb{R}^7)\) for the corresponding Grassmannians.  (Recall that a 1-dimensional subspace \(\ell\) is spacelike, timelike or null according to whether \(\tld{g}_0(u,u) >0\), \(<0\) or \(=0\), respectively, for some (equivalently all) \(u \in \ell \osr \{0\}\).)

\begin{Lem}
\ew
\lqt{\oGr_6(\bb{R}^7)}{\tld{\Gg}_2} = \lt\{\oGr_{6,i}(\bb{R}^7)~\m|~i = +,-,0\rt\}.
\eew
\end{Lem}

\begin{proof}
If \(\bb{B}\) is either spacelike or timelike, then \(\bb{R}^7 = \bb{B}^\bot\ds\bb{B}\) and so \eref{OC} can be lifted to an isomorphism \(\oGr_{1,\pm}(\bb{R}^7) \cong \oGr_{6,\pm}(\bb{R}^7)\).  Thus, since \sg\ acts transitively on each of \(\oGr_{1,\pm}(\bb{R}^7)\) \cite[Prop.\ 2.2]{SG2SoPRM}, each of \(\oGr_{6,\pm}(\bb{R}^7)\) are orbits of \sg.

For null planes, \(\bb{B}^\bot \pc \bb{B}\) and so a different approach is required.  Since \sg\ acts transitively on \(\Gr_{1,0}(\bb{R}^7)\) \cite[Prop.\ 5.4]{A3F&ESLGoTG2}, \sg\ also acts transitively on \(\Gr_{6,0}(\bb{R}^7)\) by \eref{OC} and thus the action of \sg\ on \(\oGr_{6,0}(\bb{R}^7)\) has at most two orbits.  Now recall the \sg\ 4-form \(\svps_0\) and consider the oriented null 6-plane \(\bb{B} = \<e_1,e_2,e_4,e_5,e_6,e_3 + e_7\?\).  Consider \(F \in \tld{\Gg}_2\) given by:
\ew
(e_1,e_2,e_3,e_4,e_5,e_6,e_7) \mt (e_1,-e_2,-e_3,e_4,e_5,-e_6,-e_7).
\eew
Then, \(F\) preserves \(\bb{B}\) and \(F|_\bb{B}\) is orientation-reversing.  Thus, \(\oGr_{6,0}(\bb{R}^7)\) forms a single orbit of \sg\ as claimed.

\end{proof}

\subsection{\(\mb{h}\)-principle for \(\mb{\tld{\Gg}_2}\) 4-forms}

\begin{Thm}\label{sg-4}
\sg\ 4-forms satisfy the relative \(h\)-principle.
\end{Thm}

\begin{proof}
By \lref{Abt-Lem}, \sg\ 4-forms are automatically abundant.  Thus, it suffices to prove that \sg\ 4-forms are faithful and connected.  Initially, consider the \sg\ 4-form \(\svps_0\) and the hyperplanes:
\caw
\bb{B}_+ = \<e_2,e_3,e_4,e_5,e_6,e_7\? \in \oGr_{6,+}(\bb{R}^7);\\
\bb{B}_- = \<e_1,e_2,e_3,e_4,e_5,e_6\? \in \oGr_{6,-}(\bb{R}^7).
\caaw

Then:
\ew
\svps_0|_{\bb{B}_+} = \th^{4567} - \th^{2367} - \th^{2345} = \frac{1}{2}\lt(\th^{23} - \th^{45} - \th^{67}\rt)^2
\eew
is an osemproplectic form, with connected stabiliser in \(\GL_+(\bb{B}_+)\) isomorphic to \(\Sp(6;\bb{R})\), while:
\ew
\svps_0|_{\bb{B}_-} = - \th^{2345} - \th^{1346} - \th^{1256} = -\frac{1}{2}\lt(\th^{16} + \th^{25} + \th^{34}\rt)^2
\eew
is ospisoplectic (also with connected stabiliser in \(\GL_+(\bb{B}_-)\)).

Now, turn to the null case.  Consider the \sg\ 4-form \(2\svps_1\) and the null-hyperplane \(\bb{B}_0 = \<e_1,...,e_6\?\).  Then:
\ew
2\svps_1|_{\bb{B}_0} = \frac{1}{2}\th^{2356} - \th^{1456} = \lt(\frac{1}{2}\th^{23} - \th^{14}\rt) \w \th^{56}
\eew
is a degenerate 4-form (i.e.\ neither osemproplectic nor ospisoplectic) and hence \sg\ 4-forms are faithful.  To verify that the stabiliser of \(2\svps_1|_{\bb{B}_0}\) in \(\GL_+(\bb{B}_0)\) is connected, split \(\bb{B}_0 = \<e_5,e_6\? \ds \<e_1,e_2,e_3,e_4\? \cong \bb{R}^2 \ds \bb{R}^4\) and apply \lref{stabiliser-comp-lem} to the pisoplectic (and hence multi-ossymplectic) 2-form:
\ew
\al = \frac{1}{2}\th^{23} - \th^{14} \in \ww{2}\lt(\bb{R}^4\rt)^*,
\eew
noting that \(-\al\) and \(\al\) lie in the same \(\GL_+(4;\bb{R})\)-orbit (since \((-\al)^2 = \al^2 = -\frac{1}{4}\th^{1234}\)) and that the stabilisers of \(\al\) in \(\GL_+(2k;\bb{R})\) and \(\GL(2k;\bb{R})\) coincide and are connected.  Thus, \sg\ 4-forms are also connected.

\end{proof}

\subsection{Faithfulness of \(\mb{\tld{\Gg}_2}\) 3-forms}

\begin{Prop}\label{SG23-F}
\(\svph_0\) is faithful.  More specifically, the orbits \(\mc{T}_{\svph_0}\lt(\oGr_{6,\mp}(\bb{R}^7)\rt) \in \rqt{\ww{3}\lt(\bb{R}^6\rt)}{\GL_+(6;\bb{R})}\) are precisely the orbits \(\ww[\pm]{3}\lt(\bb{R}^6\rt)\) of \slr\ 3-forms and \slc\ 3-forms respectively, whilst the orbit \(\mc{T}_{\svph_0}\lt(\oGr_{6,0}(\bb{R}^7)\rt) \in \rqt{\ww{3}\lt(\bb{R}^6\rt)}{\GL_+(6;\bb{R})}\) is not open, i.e.\ forms in this orbit are not stable.
\end{Prop}

\begin{proof}
I consider each orbit in turn.  For the timelike case, it suffices to prove that for some \sg\ 3-form \(\svph\) on \(\bb{R}^7\) and some oriented timelike subspace \(\bb{B} \pc \bb{R}^7\), the restriction \(\svph|_\bb{B}\) is an \slr\ 3-form.  Consider \(2\svph_1\) (see \eref{svph1}) and let \(\bb{B} \pc \bb{R}^7\) be the oriented timelike hyperplane \(\<e_1,e_5,e_6,-e_2,e_3,e_7\?\).  Then, \(2\svph_1|_\bb{B} = \th^{156} - \th^{237}\) is an \slr\ 3-form on \(\bb{B}\).  For the spacelike case, consider \(\svph_0\) and let \(\bb{B} \pc \bb{R}^7\) be the oriented spacelike hyperplane \(\<e_2,e_3,e_4,e_5,e_6,e_7\?\).  Then, \(\svph_0|_\bb{B} = \th^{246} - \th^{257} - \th^{347} - \th^{356}\) is an \slc\ 3-form on \(\bb{B}\).

Finally, for the null case, consider \(2\svph_1\) and let \(\bb{B} \pc \bb{R}^7\) be the oriented null hyperplane \(\<e_2,e_3,e_4,e_5,e_6,e_7\?\).  Define:
\e\label{rh0}
\rh_0 = 2\svph_1|_\bb{B} = -\th^{237} + \th^{246} - \th^{345}.
\ee
The `Hitchin map' \(K_{\rh_0}: \bb{B} \to \bb{B}\ts\ww{6}\lt(\bb{R}^6\rt)^*\) defined in \sref{6&7} is given by:
\ew
K_{\rh_0}(e_i) = \begin{dcases*}  e_5\ts \th^{234567} & if \(i = 2\);\\  e_6\ts \th^{234567} & if \(i = 3\);\\  e_7\ts \th^{234567} & if \(i = 4\);\\  0 & otherwise. \end{dcases*}
\eew
In particular, \(K_{\rh_0}^2 = 0\) and so by the results of \sref{6&7}, \(\rh_0\) is not stable

\end{proof}

The space \(\mc{T}_{\svph_0}\lt(\oGr_{6,0}(\bb{R}^7)\rt)\) shall be termed the orbit of parabolic 3-forms, and denoted \(\ww[0]{3}\lt(\bb{R}^6\rt)^*\).  (The motivation for this name derives from the fact that the stabiliser in \sg\ of a non-zero null vector is a maximal parabolic subgroup of \sg: see \cite[\S5]{A3F&ESLGoTG2}.)  In Djokovi\'{c}'s classification of 3-forms in dimensions \(n \le 8\), parabolic 3-forms correspond to the real form of the complex orbit `IV'; see \cite[\S9]{CoToa8DRVS}.

I remark that \pref{SG23-F} also shows, for all \(\rh \in \mc{T}_{\svph_0}\lt(\oGr_{6,\pm}\lt(\bb{R}^7\rt)\rt)\), that the stabiliser of \(\rh\) in \(\GL_+(6;\bb{R})\) is connected, being isomorphic to \(\SL(3;\bb{C})\) and \(\SL(3;\bb{R})^2\) respectively.  By \tref{FCA->A}, proving the relative \(h\)-principle for \sg\ 3-forms, is thus reduced to the following three lemmas:

\begin{Lem}\label{TL}
For each (equivalently any) \(\rh \in \ww[+]{3}\lt(\bb{R}^6\rt)^*\): \(0 \in \Conv\lt(\mc{N}_{\svph_0}(\rh)\rt)\).
\end{Lem}

\begin{Lem}\label{SL}
For each (equivalently any) \(\rh \in \ww[-]{3}\lt(\bb{R}^6\rt)^*\): \(0 \in \Conv\lt(\mc{N}_{\svph_0}(\rh)\rt)\).
\end{Lem}

\begin{Lem}\label{NL}
For each (equivalently any) \(\rh \in \ww[0]{3}\lt(\bb{R}^6\rt)^*\), \(\Stab_{\GL_+(6;\bb{R})}(\rh)\) is connected and \(0 \in \Conv\lt(\mc{N}_{\svph_0}(\rh)\rt)\).
\end{Lem}

The rest of this paper is devoted to proving each lemma in turn.\\

\subsection{Timelike case: \lref{TL}}

Let \(\rh \in \ww[+]{3}\lt(\bb{R}^6\rt)^*\) be an \slr\ 3-form.  The decomposition \(\bb{R}^6 = E_+ \ds E_-\) gives rise to a decomposition \(\lt(\bb{R}^6\rt)^* \cong E_+^* \ds E_-^*\) and hence:
\ew
\ww{p}\lt(\bb{R}^6\rt)^* \cong \Ds_{r+s = p} \ww{r}E_+^* \ts \ww{s}E_-^* = \Ds_{r+s = p} \ww{r,s}\lt(\bb{R}^6\rt)^*.
\eew
\(\rh\) defines an element of \(\ww{3,0}\lt(\bb{R}^6\rt)^* \ds \ww{0,3}\lt(\bb{R}^6\rt)^*\); thus the map:
\e\label{u-i-rh+}
\bcd[row sep = 0pt]
\bb{R}^6 \ar[r] & \ww{2,0}\lt(\bb{R}^6\rt)^* \ds \ww{0,2}\lt(\bb{R}^6\rt)^*\\
u \ar[r, maps to] & u \hk \rh
\ecd
\ee
is an \(\SL(3;\bb{R})^2\)-equivariant isomorphism.  Moreover, \(I_\rh\) defines a map:
\e\label{calIrh}
\bcd[row sep = 0pt]
\cal{I}_\rh:\ww{2}\lt(\bb{R}^6\rt)^* \ar[r] & \ss{2}\lt(\bb{R}^6\rt)^*\\
\om \ar[r, maps to] & \Big\{(a,b) \mt \frac{1}{2}\lt[\om(I_\rh a,b) + \om(I_\rh b,a)\rt]\Big\}
\ecd
\ee
(where \(\ss{2}\) denotes the symmetric square) which vanishes on the subspace \(\ww{2,0}\lt(\bb{R}^6\rt)^* \ds \ww{0,2}\lt(\bb{R}^6\rt)^*\) and satisfies \(\cal{I}_\rh \om(u_1, u_2) = \om(I_\rh u_1, u_2)\) for \(\om \in \ww{1,1}\lt(\bb{R}^6\rt)^*\).  In particular, \(\cal{I}_\rh\) defines an injection \(\ww{1,1}\lt(\bb{R}^6\rt)^* \emb \ss{2}\lt(\bb{R}^6\rt)^*\).

\begin{Prop}\label{TRP}
Let \(\rh\) be an \slr\ 3-form on \(\bb{R}^6\).  Then:
\ew
\mc{N}_{\svph_0}(\rh) = \lt\{ \om \in \ww{2}\lt(\bb{R}^6\rt)^* ~\m|~ \cal{I}_\rh\om \text{ has signature } (3,3) \text{ and } \om^3<0 \rt\}.
\eew
\end{Prop}

\begin{proof}
Firstly, I claim:
\e\label{1-1-split-+}
\mc{N}_{\svph_0}(\rh) = \lt[\mc{N}_{\svph_0}(\rh) \cap \ww{1,1}\lt(\bb{R}^6\rt)^*\rt] \x \ww{2,0}\lt(\bb{R}^6\rt)^* \x \ww{0,2}\lt(\bb{R}^6\rt)^*.
\ee
Indeed, let \(\om \in \ww{2}\lt(\bb{R}^6\rt)^*\) and define a 3-form on \(\bb{R}^7 \cong \bb{R} \ds \bb{R}^6\) via:
\e\label{sph-defn}
\sph = \th \w \om + \rh.
\ee
Let \(u \in \bb{R}^6\) and consider the orientation-preserving automorphism of \(\bb{R}^7\) given by:
\ew
F = \bpm
1_{1 \x 1} & 0_{1 \x 6}\\
u_{6 \x 1} & \Id_{6 \x 6}
\epm
\eew
Then, \(F^*\sph = \th \w \lt(\om + u \hk \rh\rt) + \rh\).  Thus, \(\om \in \mc{N}_{\svph_0}(\rh)\) \iff\ \(\om + u \hk \rh \in \mc{N}_{\svph_0}(\rh)\) for all \(u \in \bb{R}^6\) and \eref{1-1-split-+} follows by \eref{u-i-rh+}.  Moreover, given \(\om \in \ww{2}\lt(\bb{R}^6\rt)^*\), \(\cal{I}_\rh\om\) and \(\om^3\) only depend on the \((1,1)\)-part of \(\om\).  Thus, to prove \pref{TRP}, it suffices to prove:
\ew
\mc{N}_{\svph_0}(\rh) \cap \ww{1,1}\lt(\bb{R}^6\rt)^* = \lt\{ \om \in \ww{1,1}\lt(\bb{R}^6\rt)^* ~\m|~ \cal{I}_\rh\om \text{ has signature } (3,3) \text{ and } \om^3<0 \rt\}.
\eew

Recall the invariant quadratic form \(Q_{\sph}\) defined in \pref{Stable-in-7}.  Using \eref{sph-defn}, for \(a\in\bb{R}\) and \(u\in\bb{R}^6\) one may compute:
\ew
(ae_1 + u)\hk\sph = a\om - \th \w \lt(u \hk \om\rt) + u \hk \rh
\eew
and hence:
\e\label{2b-simp+}
6Q_{\sph}(ae_1 + u) &= \lt[(ae_1 + u)\hk\sph\rt]^2 \w \sph\\
&= a^2 \th \w \om^3 + \underbrace{\th \w \om \w (u \hk \rh)^2}_{(1)} - \underbrace{2a \th \w (u \hk \om) \w \om \w \rh}_{(2)} + \underbrace{2 a \th \w \om^2 \w (u \hk \rh)}_{(3)}\\
&\hs{3cm} - \underbrace{2 \th \w (u\hk\om) \w (u\hk\rh) \w \rh}_{(4)}.
\ee

Term (2) vanishes since \(\om \in \ww{1,1}\lt(\bb{R}^6\rt)^*\) and \(\rh \in \ww{3,0}\lt(\bb{R}^6\rt)^* \ds \ww{0,3}\lt(\bb{R}^6\rt)^*\) and hence \(\om \w \rh = 0\).  To simplify the remaining terms, I utilise the following lemma:
\begin{Lem}\label{SwapForm}
For \(\al\in\ww{p}(\bb{R}^n)^*\), \(\be\in\ww{q}(\bb{R}^n)^*\) with \(p+q = n+1\):
\ew
\forall u\in\bb{R}^n: (u\hk\al)\w\be = (-1)^{p-1}\al\w(u\hk\be).
\eew
\end{Lem}
(To prove this lemma, it suffices by linearity to consider \(u = e_1\), \(\al = \th^{i_1...i_p}\), \(\be = \th^{j_1...j_q}\) with \(1\le i_1< ... < i_p \le n\) and \(1\le j_1 < ... < j_q \le n\). The result then follows by direct calculation.)

Returning to \eref{2b-simp+}, firstly consider term \((3)\). Applying \lref{SwapForm} on \(\bb{R}^6\) yields \(\om^2\w(u\hk\rh) = -(u\hk\om^2)\w\rh = -2(u\hk\om)\w\om\w\rh\), which vanishes as above.  Similarly, for term \((1)\), since \(\om \w (u \hk \rh)^2 = \om \w \lt[u \hk \lt((u\hk\rh)\w\rh\rt)\rt]\), \lref{SwapForm} yields:
\ew
\om \w (u \hk \rh)^2 = -(u\hk\om)\w(u\hk\rh)\w\rh
\eew
and hence terms \((1)\) and \((4)\) may be combined to give \(- 3 \th \w (u\hk\om) \w (u\hk\rh) \w \rh\).  However, \((u \hk \rh) \w \rh = I_{\rh}(u) \hk vol_{\rh}\) by definition of \(I_{\rh}\) and thus by \lref{SwapForm} again:
\ew
-(u\hk\om) \w (u\hk\rh) \w \rh = \cal{I}_\rh\om(u,u) \cdot vol_{\rh}.
\eew
Hence, terms \((1)\) and \((4)\) may collectively be written as \(3 \cal{I}_\rh\om(u,u) \cdot \th \w vol_{\rh}\) and whence:
\e\label{Q-calc-orth}
6Q_{\sph}(ae_1 + u) = a^2\th \w \om^3 + 3\cal{I}_\rh\om(u,u) \th \w vol_{\rh}.
\ee
In particular, \(\bb{L} = \bb{R} \ds 0\) and \(\bb{B} = 0 \ds \bb{R}^6\) are orthogonal \wrt\ \(Q_{\sph}\).

Recall that \(\sph\) is a \sg\ 3-form \iff\ \(Q_{\sph}\) has signature \((3,4)\).  Moreover, since \(\rh = \sph|_\bb{B}\) is an \slr\ 3-form, whenever \(\sph\) is a \sg\ 3-form the hyperplane \(\bb{B} \pc \bb{R}^7\) is timelike by \pref{SG23-F}.  Since \(\bb{L}\) and \(\bb{B}\) are orthogonal, it follows that \(\sph\) is a \sg\ 3-form \iff\ \(Q_{\sph}\) has signature \((3,3)\) on \(\bb{B}\) and signature \((0,1)\) on \(\bb{L}\).  However, by \eref{Q-calc-orth} this is precisely the statement that \(\cal{I}_\rh\om\) has signature \((3,3)\) and \(\om^3 < 0\), as required.

\end{proof}

I now prove \lref{TL}:

\begin{proof}[Proof of \lref{TL}]
\Wlg\ take \(\rh = \rh_+\) (see \eref{rh+}) and consider the 2-forms:
\ew
\om_1 = 2 \th^{14} - \th^{25} - \th^{36}, \hs{5mm} \om_2 = - \th^{14} + 2 \th^{25} - \th^{36} \et \om_3 = - \th^{14} - \th^{25} + 2 \th^{36}.
\eew
Then:
\caw
\cal{I}_{\rh_+}\om_1 = 4 \th^1 \s \th^4 - 2 \th^2 \s \th^5  - 2 \th^3 \s \th^6, \hs{5mm} \cal{I}_{\rh_+}\om_2 = -2 \th^1 \s \th^4 + 4 \th^2 \s \th^5  - 2 \th^3 \s \th^6\\
\et \cal{I}_{\rh_+}\om_3 = -2 \th^1 \s \th^4 - 2 \th^2 \s \th^5  + 4 \th^3 \s \th^6
\caaw
which all have signature \((3,3)\).  Moreover, \(\om_i^3 = -12\th^{12...6}\) for \(i=1,2,3\).  Thus, by \pref{TRP} \(\om_i \in \mc{N}_{\svph_0}(\rh_+)\) for all \(i = 1,2,3\).  Therefore:
\ew
\Conv \lt(\mc{N}_{\svph_0}(\rh_+)\rt) \owns \frac{1}{3}\lt(\om_1 + \om_2 + \om_3\rt) = 0,
\eew
as required.

\end{proof}

\subsection{Spacelike case: \lref{SL}}

The spacelike case is closely analogous to the timelike case; accordingly, the exposition in this subsection will be brief.  Let \(\rh \in \ww[-]{3}\lt(\bb{R}^6\rt)^*\) be an \slc\ 3-form.  The complex structure \(J_\rh\) induces a type-decomposition \(\ww{p}\lt(\bb{R}^6\rt)^* \ts_\bb{R} \bb{C} = \Ds_{r+s = p} \ww{r,s}\lt(\bb{R}^6\rt)^*\).  As in \cite[p.\ 32]{RG&HG}, for \(r \ne s\), write \(\ls\ww{r,s}\lt(\bb{R}^6\rt)^*\rs = \lt(\ww{r,s}\lt(\bb{R}^6\rt)^* \ds \ww{s,r}\lt(\bb{R}^6\rt)^*\rt) \cap \ww{r+s}\lt(\bb{R}^6\rt)^*\) for the set of real forms of type \((r,s)+(s,r)\); likewise, write \(\lt[\ww{r,r}\lt(\bb{R}^6\rt)^*\rt] = \ww{r,r}\lt(\bb{R}^6\rt)^* \cap \ww{2r}\lt(\bb{R}^6\rt)^*\) for the set of real forms of type \((r,r)\).  Then, for all \(p\):
\ew
\ww{2p}\lt(\bb{R}^6\rt)^* = \lt(\Ds_{\substack{r+s = 2p\\ r < s}} \ls\ww{r,s}\lt(\bb{R}^6\rt)^*\rs\rt) \ds \lt[\ww{p,p}\lt(\bb{R}^6\rt)^*\rt] \et \ww{2p+1}\lt(\bb{R}^6\rt)^* = \Ds_{\substack{r+s = 2p\\ r < s}} \ls\ww{r,s}\lt(\bb{R}^6\rt)^*\rs.
\eew
As in the timelike case, \(\rh\) defines an element of \(\ls\ww{3,0}\lt(\bb{R}^6\rt)^*\rs\) and \(u \in \bb{R}^6 \mt u \hk \rh \in \ls\ww{2,0}\lt(\bb{R}^6\rt)^*\rs\) defines an \(\SL(3;\bb{C})\)-equivariant isomorphism.  Moreover, \(J_\rh\) defines a map:
\e\label{calJrh}
\bcd[row sep = 0pt]
\cal{J}_\rh:\ww{2}\lt(\bb{R}^6\rt)^* \ar[r] & \ss{2}\lt(\bb{R}^6\rt)^*\\
\om \ar[r, maps to] & \Big\{(a,b) \mt -\frac{1}{2}\lt[\om(J_\rh a,b) + \om(J_\rh b,a)\rt]\Big\}
\ecd
\ee
(note the difference in sign convention from the timelike case) which vanishes on the subspace \(\ls\ww{2,0}\lt(\bb{R}^6\rt)^*\rs\) and satisfies \(\cal{J}_\rh \om(u_1, u_2) = -\om(J_\rh u_1, u_2)\) for \(\om \in \lt[\ww{1,1}\lt(\bb{R}^6\rt)^*\rt]\).  In particular, \(\cal{J}_\rh\) defines an injection \(\lt[\ww{1,1}\lt(\bb{R}^6\rt)^*\rt] \emb \ss{2}\lt(\bb{R}^6\rt)^*\).

\begin{Prop}
Let \(\rh\) be an \slc\ 3-form on \(\bb{R}^6\).  Then:
\ew
\mc{N}_{\svph_0}(\rh) = \lt\{ \om \in \ww{2}\lt(\bb{R}^6\rt)^* ~\m|~ \cal{J}_\rh\om \text{ has signature } (2,4) \rt\}.
\eew
\end{Prop}

\begin{proof}
As in the proof of \pref{TRP}, it suffices to prove that:
\ew
\mc{N}_{\svph_0}(\rh) \cap \lt[\ww{1,1}\lt(\bb{R}^6\rt)^*\rt] = \lt\{ \om \in \lt[\ww{1,1}\lt(\bb{R}^6\rt)^*\rt] ~\m|~ \cal{J}_\rh\om \text{ has signature } (2,4) \rt\}.
\eew

Given \(\om \in \lt[\ww{1,1}\lt(\bb{R}^6\rt)^*\rt]\), writing \(\sph = \th \w \om + \rh \in \ww{3}\lt(\bb{R} \ds \bb{R}^6\rt)^*\), the calculations from the proof of \pref{TRP} yield:
\e\label{Q-S}
6Q_{\sph}(ae_1 + u) = a^2\th \w \om^3 + 6\cal{J}_\rh\om(u,u) \th \w vol_{\rh},
\ee
where the final term has a factor of \(6\) now (rather than a factor of 3) since:
\ew
(u \hk \rh) \w \rh = -2J_\rh(u) \hk vol_\rh  \et  \cal{J}_\rh \om(u_1,u_2) = -\frac{1}{2}\lt[ \om(J_\rh u_1, u_2) + \om(J_\rh u_2, u_1) \rt]
\eew
when \(\rh\) is an \slc\ 3-form, as opposed to:
\ew
(u \hk \rh) \w \rh = I_\rh(u) \hk vol_\rh  \et  \cal{I}_\rh \om(u_1,u_2) = +\frac{1}{2}\lt[ \om(I_\rh u_1, u_2) + \om(I_\rh u_2, u_1) \rt]
\eew
when \(\rh\) is an \slr\ 3-form.  In particular, \(\bb{L} = \bb{R} \ds 0\) and \(\bb{B} = 0 \ds \bb{R}^6\) are again orthogonal \wrt\ \(Q_{\sph}\).

Since \(\sph|_\bb{B} = \rh\) is an \slc\ 3-form, by \pref{SG23-F} whenever \(\sph\) is a \sg\ 3-form, the hyperplane \(\bb{B} \pc \bb{R}^7\) is spacelike and thus \(Q_{\sph}\) must have signature \((2,4)\) upon restriction to \(\bb{B}\) and \((1,0)\) upon restriction to \(\bb{L}\).  Thus, by \eref{Q-S}, one sees that \(\sph\) is a \sg\ 3-form \iff\ \(\cal{J}_\rh\om\) has signature \((2,4)\) and \(\om^3 > 0\).  However, now (unlike the timelike case) the condition that \(\cal{J}_\rh\om\) has signature \((2,4)\) automatically forces \(\om^3>0\).  Thus:
\ew
\mc{N}_{\svph_0}(\rh) = \lt\{ \om \in \ww{2}\lt(\bb{R}^6\rt)^* ~\m|~ \cal{J}_\rh\om \text{ has signature } (2,4) \rt\}
\eew
as required.

\end{proof}

I now prove \lref{SL}:

\begin{proof}[Proof of \lref{SL}]
\Wlg\ take \(\rh = \rh_-\) (see \eref{rh+}) and consider the 2-forms:
\ew
\om_1 = 2 \th^{12} - \th^{34} - \th^{56}, \hs{5mm} \om_2 = - \th^{12} + 2 \th^{34} - \th^{56} \et \om_3 = - \th^{12} - \th^{34} + 2 \th^{56}.\eew
Then:
\caw
\cal{J}_{\rh_-}\om_1 = 2\lt(\th^1\rt)^{\ts2} + 2\lt(\th^2\rt)^{\ts2} - \lt(\th^3\rt)^{\ts2} - \lt(\th^4\rt)^{\ts2} - \lt(\th^5\rt)^{\ts2} - \lt(\th^6\rt)^{\ts2};\\
\cal{J}_{\rh_-}\om_2 = - \lt(\th^1\rt)^{\ts2} - \lt(\th^2\rt)^{\ts2} + 2\lt(\th^3\rt)^{\ts2} + 2\lt(\th^4\rt)^{\ts2} - \lt(\th^5\rt)^{\ts2} - \lt(\th^6\rt)^{\ts2};\\
\cal{J}_{\rh_-}\om_3 = - \lt(\th^1\rt)^{\ts2} - \lt(\th^2\rt)^{\ts2} - \lt(\th^3\rt)^{\ts2} - \lt(\th^4\rt)^{\ts2} + 2\lt(\th^5\rt)^{\ts2} + 2\lt(\th^6\rt)^{\ts2},\\
\caaw
which all have signature \((2,4)\).  Thus by \pref{TRP}, \(\om_i \in \mc{N}_{\svph_0}(\rh_-)\) for all \(i = 1,2,3\).  Therefore:
\ew
\Conv \lt(\mc{N}_{\svph_0}(\rh_+)\rt) \owns \frac{1}{3}\lt(\om_1 + \om_2 + \om_3\rt) = 0,
\eew
as required.

\end{proof}

\subsection{Null case: \lref{NL} -- Connectedness of \(\mb{\Stab_{\GL_+(6;\bb{R})}(\rh)}\)}

As usual, \wlg\ assume that \(\rh = \rh_0\) (see \eref{rh0}).  To prove \lref{NL} -- which is manifestly invariant under the natural \(\GL_+(6;\bb{R})\) action on \(\rh_0\) -- it is beneficial to reduce this `gauge freedom' to \(\SL(6;\bb{R})\), the gauge being (partially) `fixed' by defining \(vol_{\rh_0} = \th^{234567}\).  One can then define a linear map \(H_{\rh_0}:\bb{R}^6 \to \bb{R}^6\) via:
\ew
K_{\rh_0} = H_{\rh_0} \ts vol_{\rh_0}.
\eew
(The need to arbitrarily fix a volume form arises since \(K_{\rh_0}\), being nilpotent, has no non-trivial \(\lt(\ww{6}\lt(\bb{R}^6\rt)^*\rt)^n\)-valued invariants for any \(n\) and thus parabolic 3-forms do not canonically define volume forms as \slc\ and \slr\ 3-forms do.)

To compute \(\Stab_{\GL_+(6;\bb{R})}(\rh_0)\), I begin by identifying a convenient subgroup:
\begin{Prop}\label{SL3-Subgp}
Identify \(\bb{R}^6 = \<e_2,e_3,e_4\? \ds \<e_5,e_6,e_7\? \cong \bb{R}^3 \ds \bb{R}^3\) and let \(\SL(3;\bb{R})\) act diagonally on \(\bb{R}^6\) according to this splitting. Write \(\xi:\SL(3;\bb{R}) \to \SL(6;\bb{R})\) for the corresponding group homomorphism.  Then, \(\xi(\SL(3;\bb{R}))\) preserves \(\rh_0\), \(vol_{\rh_0}\) and \(H_{\rh_0}\).
\end{Prop}

\begin{proof}
Clearly \(\xi(\SL(3;\bb{R}))\) preserves \(vol_{\rh_0}\).  Moreover, \wrt\ the basis \(\<e_2,...,e_7\?\):
\ew
H_{\rh_0} =
\bpm
0 & 0\\
\Id & 0
\epm
\eew
and thus, for \(A\in\SL(3;\bb{R})\):
\ew
\xi(A)\circ H_{\rh_0} = \bpm 0 & 0\\ A & 0 \epm = H_{\rh_0} \circ \xi(A),
\eew
as required.  Now consider the map:
\ew
\bcd[row sep = 0pt]
j:\ww{3}\lt(\bb{R}^6\rt)^* \ar[r] & \ww{3}\lt(\bb{R}^6\rt)^*\\
\th^{qrs} \ar[r, maps to] & H_{\rh_0}^*(\th^{qr})\w\th^s + H_{\rh_0}^*(\th^q)\w\th^r\w H_{\rh_0}^*(\th^s) + \th^q\w H_{\rh_0}^*(\th^{rs}).
\ecd
\eew
Since \(\xi(\SL(3;\bb{R}))\) preserves \(H_{\rh_0}\), it also preserves \(j\).  However, \(-j(\th^{567}) = - \th^{237} + \th^{246} - \th^{345} = \rh_0\) and thus \(\xi(\SL(3;\bb{R}))\) also preserves \(\rh_0\).

\end{proof}

Note that the subspace \(\<e_5,e_6,e_7\?\) can be invariantly defined as the kernel of the map \(K_{\rh_0}\). By applying \pref{SL3-Subgp}, one obtains:

\begin{Cor}\label{para-stab-cor}
Every \(F \in \Stab_{\GL_+(6;\bb{R})}(\rh_0)\) preserves the subspace \(\<e_5,e_6,e_7\?\).  Moreover, \(\Stab_{\GL_+(6;\bb{R})}(\rh_0)\) acts transitively on non-zero vectors and on ordered pairs of linearly independent vectors in \(\<e_5,e_6,e_7\?\).
\end{Cor}

I now prove the first half of \lref{NL}.  Specifically:

\begin{Lem}
\(\Stab_{\GL_+(6;\bb{R})}(\rh_0)\) is connected. Explicitly:
\ew
\Stab_{\GL_+(6;\bb{R})}(\rh_0) = \xi\lt(\SL(3;\bb{R})\rt)\cdot \cal{G}
\eew
where \(\xi\) was defined in \pref{SL3-Subgp}, \(\cal{G}\) is the contractible subgroup of \(\Stab_{\GL_+(6;\bb{R})}(\rh_0)\) defined by:
\e\label{calG}
\cal{G} = \lt\{
\lt(
\begin{array}{ccc|ccc}
d & & & & &\\
e & d^{-1} & & & &\\
f & & d^{-1} & & &\\\hline
k & l & m & d^2 & &\\
n & o & p & de & 1 &\\
q & r & s & df & & 1
\end{array}
\rt)
~\m|~
d > 0 \text{ and } d^{-1}el + d^{-1}fm - d^{-2}k - s - o = 0 \rt\}
\ee
and \(\xi\lt(\SL(3;\bb{R})\rt)\cap\cal{G} = \xi(\mc{G})\), where \(\mc{G}\pc\SL(3;\bb{R})\) consists of the set of \(3\x3\)-matrices of the form:
\ew
\bpm
1 & &\\
\la & 1 &\\
\mu & & 1
\epm
\text{ for } \la,\mu\in\bb{R}.
\eew
\end{Lem}

\begin{proof}
Define:
\ew
\cal{G} = \lt\{ F \in \Stab_{\GL_+(6;\bb{R})}(\rh_0) ~\m|~ F(e_6) = e_6 \text{ and } F(e_7) = e_7 \rt\}.
\eew
Then, since (by \cref{para-stab-cor}) \(\Stab_{\GL_+(6;\bb{R})}(\rh_0)\) preserves \(\<e_5,e_6,e_7\?\) and \(\xi(\SL(3;\bb{R}))\) acts transitively on ordered pairs of linearly independent vectors in \(\<e_5,e_6,e_7\?\), it follows that:
\ew
\Stab_{\GL_+(6;\bb{R})}(\rh_0) = \xi(\SL(3;\bb{R})) \cdot \cal{G}.
\eew

The next task is to verify \eref{calG}.  Let \(F\in\cal{G}\).  Since \(\th^{23} = -e_7 \hk \rh_0\) and \(F\) preserves \(\rh_0\) and \(e_7\), it follows that \(F^*\th^{23} = \th^{23}\) and similarly \(F^*\th^{24} = \th^{24}\), since \(\th^{24} = e_6 \hk \rh_0\). Since \(F\) also preserves \(\<e_5,e_6,e_7\?\), \wrt\ the decomposition \(\bb{R}^6 = \<e_2,e_3,e_4\? \ds \<e_5,e_6,e_7\?\) one can write:
\ew
F =
\lt(
\begin{array}{c|ccc}
F_1 & & &\\\hline
& a & &\\
F_2 & b & 1 &\\
& c & & 1
\end{array}
\rt),
\eew
where \(a,b,c\in\bb{R}\) with \(a\ne0\), \(F_2 \in \End(\bb{R}^3,\bb{R}^3)\), \(F_1 \in \GL(3;\bb{R})\) is such that \(F_1^*\th^{23} = \th^{23}\) and \(F_1^*\th^{24} = \th^{24}\), and \(a \cdot \det(F_1) > 0\).

To better understand the map \(F_1\), let \(\bb{B} = \<e_2,e_3,e_4\?\) and temporarily restrict attention to \(\bb{B}\).  Since \(\<e_4\?\pc\bb{B}\) is the kernel of the linear map \(u \in \bb{B} \mt u \hk \th^{23} \in \bb{B}^*\), the space \(\<e_4\?\) must be preserved by \(F_1\).  Likewise, \(\<e_3\?\) must also be preserved by \(F_1\), since \(F_1\) preserves \(\th^{24}\). Thus:
\ew
F_1 =
\bpm
d & &\\
e & \la &\\
f & & \mu
\epm
\eew
for some \(d, \mu, \la \in \bb{R}\osr\{0\}\) and \(e,f\in\bb{R}\). The conditions \(F_1^*\th^{23} = \th^{23}\) and \(F_1^*\th^{24} = \th^{24}\) then force \(\la = d^{-1}\) and \(\mu = d^{-1}\).

Returning now to \(\bb{R}^6\), it has been shown that:
\ew
F =
\lt(
\begin{array}{ccc|ccc}
d & & & & &\\
e & d^{-1} & & & &\\
f & & d^{-1} & & &\\\hline
k & l & m & a & &\\
n & o & p & b & 1 &\\
q & r & s & c & & 1
\end{array}
\rt)
\eew
for \(k,l,m,n,o,p,q,r,s\in\bb{R}\).  One may then compute that \(F^*\rh_0 = \rh_0\) is equivalent to \(a = d^2\), \(c = df\), \(b = de\), together with the condition:
\ew
d^{-1}el + d^{-1}fm - d^{-2}k - s - o = 0.
\eew
Moreover, given \(a = d^2\) one has \(\det(F) = d > 0\).  Thus, it has been established that:
\ew
\cal{G} =
\lt\{\lt(
\begin{array}{ccc|ccc}
d & & & & &\\
e & d^{-1} & & & &\\
f & & d^{-1} & & &\\\hline
k & l & m & d^2 & &\\
n & o & p & de & 1 &\\
q & r & s & df & & 1
\end{array}
\rt)~\middle|~d>0 \text{ and } d^{-1}el + d^{-1}fm - d^{-2}k - s - o = 0 \rt\},
\eew
as claimed.

The expression for \(\xi(\SL(3;\bb{R}))\cap\cal{G}\) is now manifest.  To see that \(\cal{G}\) is contractible, consider the projection:
\ew
\bcd[row sep = 0pt]
\cal{G} \ar[r, "\mLa{\pi}"] & (0,\infty) \x \bb{R}^6\\
(d,e,f,k,l,m,n,o,p,q,r,s) \ar[r, maps to] & (d,e,f,n,p,q,r).
\ecd
\eew
Then, \(\pi\) is surjective, with fibre over \((d,e,f,n,p,q,r)\) given by:
\ew
\lt\{(k,m,l,o,s) \in \bb{R}^5 ~\m|~ d^{-2}k - d^{-1}fm - d^{-1}el + o + s = 0\rt\}.
\eew
Thus, \(\cal{G}\) is topologically a rank-4 vector bundle over the contractible space \((0,\infty)\x\bb{R}^6\), hence contractible.  This completes the proof.

\end{proof}

\subsection{Null case: \lref{NL} -- \(\mb{0 \in \Conv\lt(\mc{N}_{\svph_0}(\rh)\rt)}\)}

By analogy with \erefs{calIrh} and \eqref{calJrh}, define:
\ew
\bcd[row sep = 0pt]
\cal{H}_{\rh_0}:\ww{2}\lt(\bb{R}^6\rt)^* \ar[r] & \ss{2}\lt(\bb{R}^6\rt)^*\\
\om \ar[r, maps to] & \Big\{(a,b) \mt \frac{1}{2}\lt[\om(H_{\rh_0} a,b) + \om(H_{\rh_0} b,a)\rt]\Big\}.
\ecd
\eew

\begin{Prop}\label{NRP}
\ew
\mc{N}_{\svph_0}(\rh_0) = \lt\{ \om \in \ww{2}\lt(\bb{R}^6\rt)^* ~\m|~ \cal{H}_{\rh_0}\om \text{ has signature } (2,1,3) \rt\}.
\eew
(Here signature \((2,1,3)\) means that \(\cal{H}_{\rh_0}\) has a maximal positive definite subspace of dimension \(2\), a maximal negative definite subspace of dimension \(3\) and a 1-dimensional kernel.)
\end{Prop}

The proof proceeds via a series of lemmas:

\begin{Lem}\label{Signature}
For \(n\ge1\), \(x_1,...,x_n\in\bb{R}\osr\{0\}\) and \(y \in \bb{R}\), any symmetric, `forward-triangular' \((2n+1)\x(2n+1)\) bilinear form:
\ew
M = \bpm
* & * & \hdots & * & \hdots & * & x_1\\
* & * & \hdots & * & \hdots & x_2 & \\
\vdots & \vdots & & & \iddots & &\\
* & * & & y & & &\\
\vdots & \vdots & \iddots & & & &\\
* & x_2 & & & & &\\
x_1 & & & & & &
\epm
\eew
is non-degenerate \iff\ \(y \ne 0\), having signature \((n+1,n)\) if \(y>0\) and signature \((n,n+1)\) if \(y<0\).  Moreover, the \(2n \x 2n\) bilinear form:
\ew
M' = \bpm
* & * & \hdots & * & x_1\\
* & * & \hdots & x_2 & \\
\vdots & \vdots & \iddots & &\\
* & x_2 & & &\\
x_1 & & & &
\epm
\eew
has signature \((n,n)\).
\end{Lem}

\begin{proof}
Start with the matrix \(M\).  Call the diagonal running from the \((2n+1,1)\)-entry to the \((1,2n+1)\)-entry the counter diagonal and call the elements in front of the counter diagonal the strictly forward entries.  Clearly, the bilinear form is degenerate when \(y = 0\).  Moreover, when \(y \ne 0\), \(M\) is non-singular for any values of the strictly forward entries and thus it suffices to compute the signature of \(\M\) when all of the strictly forward entries vanish.  However, in this case, \(M\) has eigenvalues \(y\), \(\pm x_1\), \(\pm x_2\),...,\(\pm x_n\), with corresponding eigenvectors:
\ew
\bpm 0 \\ 0 \\ \vdots \\ 1 \\ \vdots \\ 0 \\ 0 \epm, \bpm 1 \\ 0 \\ \vdots \\ 0 \\ \vdots\\ 0 \\ \pm1 \epm, \bpm 0 \\ 1 \\ \vdots \\ 0 \\ \vdots\\ \pm1 \\ 0 \epm, ..., \bpm 0 \\ \vdots \\ 1 \\ 0 \\ \pm 1 \\ \vdots\\ 0 \epm.
\eew
The case of \(M'\) is similar.

\end{proof}

\begin{Lem}
\ew
\Ker(\cal{H}_{\rh_0}) = \bb{R}^6 \hk \rh_0.
\eew
\end{Lem}

\begin{proof}
Take a basis of \(\bb{R}^6 \hk \rh_0\) as follows:
\ew
(e_2 \hk \rh_0, e_3 \hk \rh_0, ..., e_7 \hk \rh_0) = \lt(-\th^{37} + \th^{46}, \th^{27} - \th^{45}, -\th^{26} + \th^{35}, -\th^{34}, \th^{24}, -\th^{23}\rt).
\eew
Extend this to a basis of \(\ww{2}\lt(\bb{R}^6\rt)^*\) via:
\ew
\lt(\th^{25}, \th^{36}, \th^{47}, \th^{56}, \th^{57}, \th^{67}, \th^{26}+\th^{35}, \th^{27}+\th^{45}, \th^{37}+\th^{46}\rt).
\eew
(By analogy with the spacelike and timelike cases, denote the span of this latter set of 2-forms by \(\ww{1,1}\lt(\bb{R}^6\rt)^*\).)  Then:
\ew
\ww{2}\lt(\bb{R}^6\rt)^* = \ww{1,1}\lt(\bb{R}^6\rt)^* \ds \bb{R}^6\hk\rh_0,
\eew
although the reader should note that this splitting is only defined here for \(\rh_0\); no attempt is made to define \(\ww{1,1}\lt(\bb{R}^6\rt)^*\) for an arbitrary parabolic 3-form.  Then, \(\cal{H}_{\rh_0}\) vanishes identically on \(\bb{R}^6 \hk \rh_0\): indeed:
\ew
\bcd[row sep = 0pt]
-\th^{37} + \th^{46} \ar[r, maps to, "\cal{H}_{\rh_0}"] & \th^3\s \th^4 - \th^3 \s \th^4 = 0\\
\th^{27} - \th^{45} \ar[r, maps to, "\cal{H}_{\rh_0}"] & -\th^2\s \th^4 + \th^2 \s \th^4 = 0
\ecd
\eew
and similarly for the other basis vectors. Moreover, one may verify that:
\e\label{(1,1)-Im}
\bcd[row sep = 0pt]
\bpm
\th^{25}\\
\th^{36}\\
\th^{47}\\
\th^{56}\\
\th^{57}\\
\th^{67}\\
\th^{26}+\th^{35}\\
\th^{27}+\th^{45}\\
\th^{37}+\th^{46}
\epm
\ar[r, maps to, "\cal{H}_{\rh_0}"] &
\bpm
-\th^2\s \th^2\\
-\th^3 \s \th^3\\
-\th^4 \s \th^4\\
\th^2 \s \th^6 - \th^3 \s \th^5\\
\th^2 \s \th^7 - \th^4 \s \th^5\\
\th^3 \s \th^7 - \th^4\s \th^6\\
-2\th^2\s \th^3\\
-2 \th^2\s \th^4\\
-2 \th^3 \s \th^4
\epm.
\ecd
\ee
Since these images are linearly independent, the map \(\cal{H}_{\rh_0}\) is injective when restricted to \(\ww{1,1}\lt(\bb{R}^6\rt)^*\). This completes the proof.

\end{proof}

\begin{Lem}\label{Kernel}
For all \(\om\in\ww{2}\lt(\bb{R}^6\rt)^*\):
\ew
\Ker(\cal{H}_{\rh_0}\om)\cap \<e_5,e_6,e_7\? \ne 0.
\eew
Here \(\Ker(\cal{H}_{\rh_0}\om)\) denotes the kernel of the map:
\ew
\flat:\bb{R}^6 &\to \lt(\bb{R}^6\rt)^*\\
u &\mt \cal{H}_{\rh_0}\om(u,-).
\eew
\end{Lem}

\begin{proof}
It is equivalent to show that \(\flat|_{\<e_5,e_6,e_7\?}\) is not injective.  Thus, fix \(\om \in \ww{2}\lt(\bb{R}^6\rt)^*\).  Recalling that \(\cal{H}_{\rh_0}\) vanishes on \(\bb{R}^6 \hk \rh_0\) and inspecting \eref{(1,1)-Im}, one sees that \(\flat|_{\<e_5,e_6,e_7\?}\) only depends on the component of \(\om\) in the subspace \(\<\th^{56},\th^{57},\th^{67}\? \pc \ww{1,1}\lt(\bb{R}^6\rt)^*\), so \wlg\ assume that:
\ew
\om = \la_1\th^{56} + \la_2 \th^{57} + \la_3 \th^{67}.
\eew
Thus:
\ew
\cal{H}_{\rh_0}\om = \la_1\lt(\th^2 \s \th^6 - \th^3 \s \th^5\rt) + \la_2\lt(\th^2 \s \th^7 - \th^4 \s \th^5\rt) + \la_3\lt(\th^3 \s \th^7 - \th^4\s \th^6\rt).
\eew
Hence, \(\flat|_{\<e_5,e_6,e_7\?}\) maps \(\<e_5,e_6,e_7\?\) into \(\<e_2,e_3,e_4\?\) and is represented by the matrix:
\e\label{Null-Space-Sys}
\frac{1}{2}\bpm
0 & \la_1 & \la_2\\
-\la_1 & 0 & \la_3\\
-\la_2 & -\la_3 & 0
\epm
\ee
which has determinant 0, as required.

\end{proof}

I now prove \pref{NRP}:

\begin{proof}[Proof of \pref{NRP}]
As in the proof of \pref{TRP}:
\ew
\mc{N}_{\svph_0}(\rh_0) = \lt[\mc{N}_{\svph_0}(\rh_0) \cap \ww{1,1}\lt(\bb{R}^6\rt)^*\rt] \x \lt(\bb{R}^6 \hk \rh_0\rt)
\eew
and, therefore, since \(\cal{H}_{\rh_0}\) vanishes on \(\bb{R}^6 \hk \rh_0\), it suffices to prove that:
\ew
\mc{N}_{\svph_0}(\rh_0) \cap \ww{1,1}\lt(\bb{R}^6\rt)^* = \lt\{ \om \in \ww{1,1}\lt(\bb{R}^6\rt)^* ~\m|~ \cal{H}_{\rh_0}\om \text{ has signature } (2,1,3) \rt\}.
\eew

Let \(\om \in \ww{1,1}\lt(\bb{R}^6\rt)^*\) and define \(\sph = \th \w \om + \rh_0\).  Proceeding as in the proof of \pref{TRP}, one obtains:
\e\label{null-Q}
6Q_{\sph}(ae_1 + u) = a^2\th^1 \w \om^3 - 6a \th^1 \w (u\hk\om) \w \om \w \rh_0 + 3\cal{H}_{\rh_0}\om(u,u) \th^1 \w vol_{\rh_0}
\ee
where now, unlike for \(\SL(3;\bb{R})^2\) and \slc\ 3-forms, \(\om \w \rh_0\) need not vanish.  Initially, suppose that \(\sph\) is of \sg-type and write \(\bb{B} = 0 \ds \bb{R}^6 \pc \bb{R}^7\).  Since \(\sph|_\bb{B} = \rh_0\) is parabolic, it follows that \(\bb{B}\) is null, hence \(Q_{\sph}\) has signature \((2,1,3)\) upon restriction to \(\bb{B}\) and whence by \eref{null-Q} \(\cal{H}_{\rh_0}\om\) has signature \((2,1,3)\), as required.

Conversely, suppose that \(\cal{H}_{\rh_0}\om\) has signature \((2,1,3)\).  Then, \(\cal{H}_{\rh_0}\om\) has a 1-dimensional kernel which by \lref{Kernel} must be contained in \(\<e_5,e_6,e_7\?\). By applying a suitable \(\SL(3;\bb{R})\)-symmetry (see \pref{SL3-Subgp}), \wlg\ one can assume that:
\ew
\Ker(\cal{H}_{\rh_0}\om) = \<e_7\?.
\eew
Since \(\om\in\ww{1,1}\lt(\bb{R}^6\rt)^*\), by examining the matrix for \(\flat|_{\<e_5,e_6,e_7\?}\) in \eref{Null-Space-Sys}, it follows that \(\om\) has the form:
\ew
\om = A \th^{25} + B \th^{36} + C \th^{47} + D \th^{56} + E\lt(\th^{26} + \th^{35}\rt) + F\lt(\th^{27} + \th^{45}\rt) + G\lt(\th^{37} + \th^{46}\rt),
\eew
for some constants \(A,B,C,D,E,F,G\in\bb{R}\) with \(D\ne0\).  Upon restriction to \(\<e_2,...,e_6\?\) the bilinear form \(\cal{H}_{\rh_0}\om\) is represented by the matrix:
\ew
& {\bma \hs{11.5pt}e_2 & \hs{7pt} e_3 & \hs{5pt}e_4 & \hs{4pt}e_5 & \hs{3.5pt}e_6 \ema}\\
{\bma e_2\\ e_3\\ e_4\\ e_5\\ e_6 \ema} ~&~ {\bpm -A & -E & -F & 0 & \frac{D}{2}\\ -E & -B & -G & -\frac{D}{2} & \\ -F & -G & -C & & \\ 0 & -\frac{D}{2} & & & \\ \frac{D}{2} & & & & \epm}.
\eew
By assumption, this bilinear form is non-degenerate with signature \((2,3)\), and thus it follows from \lref{Signature} that \(C > 0\).

Next, observing \(\om\w\rh_0 = -D \th^{23567}\) yields:
\ew
\big((-)\hk\om\big)\w\om\w\rh_0 = D\lt(C \th^7 + F \th^5 + G \th^6\rt)\ts vol_{\rh_0}.
\eew
Substituting this result into \eref{null-Q} and polarising shows that \(2Q_{\sph}\) is represented by the symmetric \(7\x7\)-matrix:
\ew
& {\bma \hs{21.5pt}e_1 & \hs{13pt} e_7 & \hs{9pt}e_2 & \hs{2mm}e_3 & \hs{2mm} e_4 & \hs{8pt}e_5 & \hs{5.5mm}e_6 \ema}\\
{\bma e_1\\ e_7\\ e_2\\ e_3\\ e_4\\ e_5\\ e_6 \ema} ~&~ {\lt( \begin{array}{cc|ccccc} \frac{H}{3} & -DC & & & & -DF & -DG\\ -DC & & & & & &\\\hline  & & -A & -E & -F & 0 & \frac{D}{2}\\ & & -E & -B & -G & -\frac{D}{2} & \\ & & -F & -G & -C & &  \\ -DF & & 0 & -\frac{D}{2} & & & \\ -DG & & \frac{D}{2} & & & & \end{array}\rt)}.
\eew
where \(H\in\bb{R}\) is such that \(\th^1\w\om^3 = H\th^1\w vol_{\rh_0}\).  Thus, to complete the proof, it \stp\ that this matrix has signature \((3,4)\).  In fact, I show that for any \(h,r,s,t\in\bb{R}\), \(r\ne 0\), the matrix:
\ew
& {\bma \hs{13pt}e_1 & \hs{-5pt}e_7 & \hs{1pt}e_2 & \hs{7pt}e_3 & \hs{5pt}e_4 & ~ e_5 & ~ e_6 \ema}\\
M_{h,r,s,t} = ~\ ~\  {\bma e_1\\ e_7\\ e_2\\ e_3\\ e_4\\ e_5\\ e_6 \ema} ~&~ {\lt( \begin{array}{cc|ccccc} h & r & & & & s & t\\ r & & & & & &\\\hline & & -A & -E & -F & 0 & \frac{D}{2}\\ & & -E & -B & -G & -\frac{D}{2} & \\ & & -F & -G & -C & & \\ s & & 0 & -\frac{D}{2} & & & \\ t & & \frac{D}{2} & & & & \end{array}\rt)}
\eew
has signature \((3,4)\), which is sufficient to complete the proof, as \(-DC \ne 0\).  Since \(\det M_{h,r,s,t} = -Cr^2\lt(\frac{D}{2}\rt)^4 \ne 0\), \(M_{h,r,s,t}\) is non-singular for all values of \(h,s,t\) and thus it suffices to consider \(M_{0,r,0,0}\).  However, \(M_{0,r,0,0}\) is block diagonal with blocks:
\ew
\bpm 0 & r \\ r & 0 \epm \et \bpm -A & -E & -F & 0 & \frac{D}{2}\\ -E & -B & -G & -\frac{D}{2} & \\ -F & -G & -C & & \\ 0 & -\frac{D}{2} & & & \\ \frac{D}{2} & & & & \epm.
\eew
By \lref{Signature} the former block has signature \((1,1)\) and the latter block has signature \((2,3)\), and thus \(M_{0,r,0,0}\) has signature \((3,4)\), as claimed.

\end{proof}

I now prove the second part of \lref{NL}.  Specifically:
\begin{Lem}
For each (equivalently any) \(\rh \in \ww[0]{3}\lt(\bb{R}^6\rt)^*\), \(0 \in \Conv\lt(\mc{N}_{\svph_0}(\rh)\rt)\).
\end{Lem}

\begin{proof}
As usual, \wlg\ let \(\rh = \rh_0\).  Consider the 2-form \(\om_0 = - 2 \th^{47} + 2\ep(\th^{67} + \th^{25} - \th^{36})\) for some \(\ep\in\bb{R}\osr\{0\}\) to be specified later. One may compute using \eref{(1,1)-Im} that:
\ew
\cal{H}_{\rh_0}\om_0 = 2 \th^4 \s \th^4 + 2\ep(\th^3 \s \th^7 - \th^4 \s \th^6 - \th^2 \s \th^2 + \th^3 \s \th^3).
\eew
This may be represented by the \(6\x6\)-matrix:
\ew
&~~ {\bma \hs{7pt}e_2 & \hs{1pt}e_3 & \hs{-2pt}e_7 & \hs{-2pt}e_4 & e_6 & \hs{-2pt}e_5 \ema}\\
{\bma e_2\\ e_3\\ e_7\\ e_4\\ e_6\\ e_5 \ema}~
&{\bpm -2\ep & & & & & \\ & 2\ep & \ep & & & \\ & \ep & & & & \\ & & & 2 & -\ep & \\ & & & -\ep & &\\ & & & & & 0 \epm}
\eew
which has signature \((2,1,3)\) by applying \lref{Signature} to each matrix along the (block) diagonal. Thus, \(\om_0\in\mc{N}_{\svph_0}(\rh_0)\).

Now, consider \(\om_\pm = \th^{47} \pm \th^{56}\). Then, one may verify again using \eref{(1,1)-Im} that:
\ew
\cal{H}_{\rh_0}\om_\pm = -\th^4\s \th^4 \pm (\th^2 \s \th^6 - \th^3 \s \th^5)
\eew
both of which have signature \((2,1,3)\). Thus, \(\om_\pm\in\mc{N}_{\svph_0}(\rh_0)\). Moreover, \(\mc{N}_{\svph_0}(\rh_0)\pc\ww{2}\lt(\bb{R}^6\rt)^*\) is open, so it follows that for all \(\ep\in\bb{R}\osr\{0\}\) with \(|\ep|\) sufficiently small, the 3-forms:
\ew
\om'_\pm = \om_\pm - \ep(\th^{67} + \th^{25} - \th^{36})
\eew
also lie in \(\mc{N}_{\svph_0}(\rh_0)\). Fix some suitable choice of \(\ep\); then, the three 2-forms \(\om_0,\om'_\pm\) all lie in \(\mc{N}_{\svph_0}(\rh_0)\).  One may then compute that:
\ew
\frac{1}{3}\lt(\om_0 + \om'_+ + \om'_-\rt) = & \frac{1}{3}\lt(- 2 \th^{47} + 2\ep(\th^{67} + \th^{25} - \th^{36})\rt)\\
& ~+ \th^{47} + \th^{56} - \ep(\th^{67} + \th^{25} - \th^{36})\\
& \lt.~+ \th^{47} - \th^{56} - \ep(\th^{67} + \th^{25} - \th^{36})\rt)\\
=& 0,
\eew
completing the proof.

\end{proof}

Thus, by \tref{FCA->A}, it has been proven:
\begin{Thm}\label{sg-3}
\sg\ 3-forms satisfy the relative \(h\)-principle.
\end{Thm}
~\qed

\appendix

\section{Hodge-type decomposition on non-compact manifolds}\label{NCH-Sec}

The aim of this appendix is to prove \tref{NCH}.  Recall the statement of the theorem:\vs{3mm}

\noindent{\bf Theorem \ref{NCH}.}
\em Let \(\M\) be an \(n\)-manifold (not necessarily oriented and possibly non-compact or with boundary).  Then, there exists an injective, continuous linear operator \(\io:\Ds_p\dR{p}(\M) \to \Ds_p\Om^p(\M)\) of degree 0 and a continuous linear operator \(\de:\Ds_p\Om^p(\M) \to \Ds_p\Om^p(\M)\) of degree \(-1\) satisfying \(\de^2 = 0\) such that for each \(0 \le p \le n\):
\begin{equation*}
\Om^p(\M) = \io\dR{p}(\M) \ds \dd\Om^{p-1}(\M) \ds \de\Om^{p+1}(\M) \tag{\ref*{NCHD-eq}}
\end{equation*}
in the category of \F\ spaces, where both \(\dd\) and \(\de\) act as \(0\) on \(\io\dR{p}(\M)\) and the maps \(\dd: \de\Om^{p+1}(\M) \to \dd\Om^p(\M)\) and \(\de: \dd\Om^p(\M) \to \de\Om^{p+1}(\M)\) are mutually inverse.  In particular, the projections \(\Om^p(\M) \to \dd\Om^{p-1}(\M)\) and \(\Om^p(\M) \to \de\Om^{p+1}(\M)\) are given by \(\dd \circ \de\) and \(\de \circ \dd\), respectively.\vs{2mm}\em

The most significant features of the above theorem is that there is no compactness assumption on \(\M\) and that there is no assumption that \(\M\) be oriented.  Indeed, in the special case where \(\M\) is closed and oriented, after picking a Riemannian metric \(g\) on \(\M\), \tref{NCH} is already known to hold by Hodge Decomposition, with \(\io\) induced by the isomorphism from \(\dR{p}(\M)\) to the space \(\cal{H}^p(\M)\) of harmonic \(p\)-forms on \(\M\) and with \(\de = \dd^*G\), where \(\dd^*\) is the co-exterior derivative and \(G\) is the Green's operator for the Hodge Laplacian \(\De\) defined by \(g\).  Likewise, on compact orientable manifolds with boundary, \tref{NCH} is known to hold by \cite[Cor.\ 2.4.9 and Thm.\ 6.1]{HD-AMfSBVP}, with \(\io\) induced by the isomorphism from \(\dR{p}(\M)\) to the space \(\cal{H}^p_N(\M)\) of \(p\)-forms which are both closed and coclosed and which have vanishing normal component along the boundary of \(\M\), and with \(\de = \dd^* G_N\), where \(G_N\) is the Green's operator for the operator \(\De\) with Neumann boundary conditions.  Explicitly, \(G_N\) acts as 0 on the space \(\cal{H}^p_N(\M)\), whilst for \(p\)-forms \(\al\) which are \(L^2\) orthogonal to \(\cal{H}^p_N(\M)\), \(G_N(\al)\) is the unique solution to the elliptic boundary value problem:
\caw
\De G_N(\al) = \al \text{ on } \M;\\
{\bf n}\,G_N(\al) = 0 \text{ and } {\bf n}\,\dd G_N(\al) = 0 \text{ on } \del\M,
\caaw
where {\bf n} denotes the normal component along \(\del \M\).  However, in the case of non-compact manifolds, to the author's knowledge Hodge decomposition-type results have only been previously obtained by restricting attention to oriented manifolds and spaces of forms with restricted growth `at infinity' (such as Lebesgue or Sobolev spaces), as well as by imposing other geometric assumptions on the chosen Riemannian metric \(g\), such as completeness or lower bounds on curvature; see, e.g.\ \cite{HD-AMfSBVP, THKWHLoNM}.  By contrast, \tref{NCH} applies to any \(n\)-manifold, regardless of whether it is compact or oriented, and applies to the full space \(\Om^p(\M)\) of \(p\)-forms, regardless of their rate of growth at infinity.

A second notable feature of \tref{NCH} is that, in contrast to other Hodge decomposition-type theorems, elliptic regularity does not play a role in the proof of \tref{NCH}, and it is not necessary to choose a Riemannian metric on \(\M\).  Instead, the proof combines \F\ space theory with direct calculations involving triangulations, performed in \cite[Ch.\ IV.C]{GIT}.\\

\subsection{Preliminaries}\label{App-Prelim}

I begin by recounting some preliminary definitions.  The main references for this subsection are \cite[Ch.\ I]{FA(KY)} and \cite[\S18]{TVSI} for \F\ spaces, and \cite[Ch.\ II]{EDT} and \cite[App.\ II]{GIT} for triangulations, although the reader should note that the discussion of the \F\ topologies on the spaces of simplicial cochains, cocycles and coboundaries are the work of the author and cannot, to the author's knowledge, be found elsewhere in the literature.  In this appendix, all vector spaces shall be taken to be real.

A topological vector space is termed a \F\ space if it is complete, metrisable and locally convex.  Given a \F\ space \(F\), one may always choose an increasing sequence \(\|-\|_0 \le \|-\|_1 \le . . .\) of continuous semi-norms such that the sets \(U_k = \lt\{ f \in F ~\m|~ \|f\|_k < \frac{1}{k+1}\rt\}\) form a basis of open neighbourhoods of \(0 \in F\) (in particular, the semi-norms \(\|-\|_k\) are separating, i.e.\ if \(\|f\|_k = 0\) for all \(k \in \bb{N}\), then \(f = 0\)).  A choice of such semi-norms is termed a grading on \(F\).  Whilst gradings are not unique, given any two gradings \(\lt(\|-\|_k\rt)_k\) and \(\lt(\|-\|'_k\rt)_k\), for every \(k \in \bb{N}\), there exist \(k' \in \bb{N}\) and \(C_k > 0\) such that:
\e\label{equiv-grads}
\|-\|_k \le C_k\|-\|'_{k'} \,.
\ee
By definition, a sequence \((f_n)_n\) in \(F\) converges to \(f\) \iff\ \(\|f_n - f\|_k \to 0\) as \(n \to \0\) for all \(k\) (say that \((f_n)_n\) converges to \(f\) \wrt\ each \(\|-\|_k\)).  In particular, the completeness of \(F\) can be phrased in terms of a choice of grading by the following condition:\\

\hangindent = 5mm\noindent(I) For any sequence \((f_n)_n\) in \(F\) which is Cauchy \wrt\ each \(\|-\|_k\), there exists \(f\) in \(F\) such that \(f_n \to f\) as \(n \to \0\) \wrt\ each \(\|-\|_k\).\\

\noindent Conversely, given a vector space \(F\) and a collection of semi-norms \(\|-\|_0 \le \|-\|_1 \le ...\) which is separating and which satisfies (I), the sets \(\lt\{U_k\rt\}_k\) define a basis of open neighbourhoods about \(0 \in F\) for a unique translation-invariant topology on \(F\), making \(F\) into a \F\ space.

Given a \F\ space \(F\) and a closed subspace \(E\), both \(E\) and the quotient \(\rqt{F}{E}\) are also \F\ spaces.  Explicitly, given a grading \((\|-\|_k)_k\) on \(F\), the topology on \(E\) can be induced by restricting the seminorms \(\|-\|_k\) to \(E\) and the topology on \(\rqt{F}{E}\) can be induced by the seminorms:
\ew
\|f+E\|'_k = \inf \lt\{ \|f'\|_k ~\m|~ f' \in f + E \rt\}.
\eew

The following examples of \F\ spaces will be of particular interest in this appendix:
\begin{Exs}\label{FSpace-Exs}~

\noindent(1) Any finite dimensional vector space is \F.  More generally, any Banach space \((B, \|-\|)\) is a graded \F\ space, with grading given by \(\|-\|_k = \|-\|\) for all \(k\).\\

\noindent(2) Consider the space \(\upomega = \bb{R}^\bb{N}\), i.e.\ the space of all real sequences.  Define seminorms \(\|-\|_k\) on \(\upomega\) by:
\ew
\|x\|_k = \sup_{0 \le i \le k} |x_i|.
\eew
Then, \(\lt(\upomega, \lt(\|-\|_k\rt)_k\rt)\) is a graded \F\ space, and the induced topology on \(\upomega\) is simply the product topology.\\

\noindent(3) Let \(\M\) be a manifold (possibly with boundary) and let \(K_0 \cc K_1 \cc ...\) be a compact exhaustion of \(\M\) (in the case where \(\M\) is compact, set \(K_k = \M\) for all \(k\)).  Pick a Riemannian metric \(g\) on \(\M\) and let \(\nabla\) denote the corresponding Levi-Civita connection.  For each \(0 \le p \le n\), define seminorms \(\|-\|_k\) on \(\Om^p(\M)\) via:
\ew
\|\al\|_k = \sup_{x \in K_k} \lt(\sum_{i = 0}^k \lt. \lt| \nabla^{\s i} \al \rt|_g^2 \rt|_x\rt)^\frac{1}{2}.
\eew
Then, \(\lt(\Om^p(\M), \lt(\|-\|_k\rt)_k\rt)\) is a graded \F\ space and the induced topology on \(\Om^p(\M)\) is the topology of uniform convergence of all derivatives on compact subsets.  One may verify that, \wrt\ this topology, the exterior derivative \(\dd: \Om^p(\M) \to \Om^{p+1}(\M)\) is continuous and thus the subspace \(\Om^p_\cl(\M) \pc \Om^p(\M)\), being closed, is a \F\ space.  Likewise, recall that, given any \(\al \in \Om^p(\M)\) and any differentiable singular \(p\)-chain \(A\) in \(\M\), one may compute the integral \(\bigintsss_A \al\) (see \cite[\S4.6]{FoDM&LG}).  One may verify that \(\bigintsss_A: \Om^p(\M) \to \bb{R}\) is continuous.  Moreover, by de Rham's and Stokes' Theorem, the space \(\dd\Om^{p-1}(\M)\) may be characterised as \(\mLa{\bigcap}_A \ker \bigintsss_A\), where the intersection runs over all differentiable singular \(p\)-cycles in \(\M\) (see \cite[Thm.\ 4.17]{FoDM&LG}).  Hence, the space \(\dd\Om^{p-1}(\M) \pc \Om^p(\M)\) is also closed and whence a \F\ space.  Thus, \(\dR{p}(\M) = \rqt{\Om^p_\cl(\M)}{\dd\Om^{p-1}(\M)}\) is also a \F\ space.\\
\end{Exs}

\begin{Rk}
Note that since \(\dd\Om^{p-1}(\M)\) is the image of the \F\ space \(\Om^{p-1}(\M)\) under the continuous linear map \(\dd\), one could also topologise \(\dd\Om^{p-1}(\M)\) via identification with the quotient \(\rqt{\Om^{p-1}(\M)}{\Om^p_\cl(\M)}\).  The resulting topology on \(\dd\Om^{p-1}(\M)\) however, is precisely the same as the subspace topology described above.  Indeed, give \(\dd\Om^{p-1}(\M)\) its subspace topology and consider the continuous linear map \(\dd:\Om^{p-1}(\M) \to \dd\Om^{p-1}(\M)\).  By definition, this map factors through the quotient \(\rqt{\Om^{p-1}(\M)}{\Om^p_\cl(\M)}\) to give a continuous linear bijection \(\rqt{\Om^{p-1}(\M)}{\Om^p_\cl(\M)} \to \dd\Om^{p-1}(\M)\), which is an isomorphism (i.e.\ has a continuous inverse) by the Open Mapping Theorem (see \cite[\S II.5]{FA(KY)}).\\
\end{Rk}

Given a graded \F\ space \((F,(\|-\|_k)_k)\), one may associate a family of Banach spaces \(F_k\), termed the local Banach spaces associated with the grading, by setting \(F_k\) to be the completion of the normed vector space \(\lt(\rqt{F}{\ker\|-\|_k}, \|-\|_k\rt)\), where \(\ker\|-\|_k = \lt\{ f \in F ~\m|~ \|f\|_k = 0\rt\}\), as usual.  Since \(\|-\|_k \le \|-\|_{k+1}\) for each \(k\), the identity map \(\Id:F \to F\) naturally induces maps \(\io^k_l: F_k \to F_l\) whenever \(k \ge l\) which satisfy \(\io^l_j \circ \io^k_l = \io^k_j\) for \(k \ge l \ge j\).  A \F\ space \(F\) is termed nuclear (see \cite[Ch.\ X(App.)]{FA(KY)}) if given some (equivalently any) choice of grading \(\|-\|_k\) on \(F\), for any \(l \in \bb{N}\) there exists \(k \ge l\) such that the mapping \(\io^k_l: F_k \to F_l\) is nuclear, i.e.\ such that there exist sequences \(\Ph_n \in F_k^*\), \(f_n \in F_l\) and \(c_n \in \bb{R}\) satisfying:
\ew
\sup_n \|\Ph_n \|_k < \0, \hs{3mm} \sup_n \|f_n\|_l < \0 \et \sum_n |c_n| < \0
\eew
such that for all \(f \in F_k\):
\ew
\io^k_l(f) = \sum_n c_n\Ph_n(f) f_n.
\eew
(Here \(F^*_k\) denotes the topological dual space of \(F_k\).)  The key result concerning nuclear spaces which will be required in this appendix is the following, proved in \cite[Thm.\ 1.6]{SRoCLMBFS}:
\begin{Thm}\label{nuclear-split}
Let \(E\) and \(G\) be \F\ spaces and suppose that \(E\) is nuclear.  Fix a grading \(\|-\|_k\) on \(E\) and suppose that the linear map:
\ew
\bcd[row sep = 0pt]
\mc{I}: \prod_{k \in \bb{N}} \mc{L}(G,E_k) \ar[r] & \prod_{k \in \bb{N}} \mc{L}(G,E_k)\\
(A_k)_k \ar[r, maps to] & \lt(\io^{k+1}_k A_{k+1} - A_k\rt)_k
\ecd
\eew
is surjective, where \(\mc{L}(-,-)\) denotes the space of continuous linear maps.  Then, every short exact sequence of \F\ spaces:
\ew
\bcd
0 \ar[r] & E \ar[r, "\al"] & F \ar[r, "\be"] & G \ar[r] & 0
\ecd
\eew
splits, i.e.\ there exists \(\ga \in \mc{L}(G,F)\) such that \(\be \circ \ga = \Id_G\).  In particular, \(F = \al(E) \ds \ga(G)\), in the sense that every \(f \in F\) can be written uniquely as the sum of elements of \(\al(E)\) and \(\ga(G)\), and moreover that the projections \(F \to \al(E)\) and \(F \to \ga(G)\) are continuous.
\end{Thm}

This appendix also considers a further source of \F\ spaces, {\it viz.}\ the spaces of cochains, cocycles and coboundaries in a simplicial complex.  As stated above, to the author's knowledge this \F\ space perspective on these spaces is new and cannot be found elsewhere in the literature.

Let \(K \cc \bb{R}^m\) be a (possibly infinite) \(n\)-dimensional simplicial complex and for each \(0 \le p \le n\) recall the vector spaces \(C_p(K)\), \(Z_p(K)\) and \(B_p(K)\) of (real) chains, cycles and boundaries in \(K\), and likewise recall the vector spaces \(C^p(K)\), \(Z^p(K)\) and \(B^p(K)\) of cochains, cocycles and coboundaries in \(K\).  \(C^p(K)\) naturally has the structure of a \F\ space as follows: let \(\si_0,\si_1,\si_2,...\) be a (possibly finite) enumeration of the \(p\)-dimensional simplices in \(K\).  Then, one obtains a bijection \(C^p(K) \to \bb{R}^\pt\), where \(\pt = 1,2,...\), or \(\bb{N}\), given by:
\ew
a \in C^p(K) \mt \lt(a(\si_0), a(\si_1), a(\si_2), . . . \rt) \in \bb{R}^\pt,
\eew
and thus \(C^p(K)\) inherits the structure of a \F\ space by \exrefs{FSpace-Exs}(1) \& (2) (note that the topology is independent of the choice of ordering of the simplices).  Moreover, by examining the explicit formula for the coboundary map \(\dd\) on cochains (see, e.g.\ \cite[App.\ II, \S7, eqn.\ 5]{GIT}), one may verify that \(\dd\) is continuous \wrt\ the \F\ space topologies on \(C^p(K)\) and \(C^{p+1}(K)\), hence \(Z^p(K) \pc C^p(K)\) is closed and whence it too is a \F\ space.  Likewise, for each chain \(A \in Z_p(K)\), one obtains a continuous linear map \(\ep_A: C^p(K) \to \bb{R}\) given by evaluation along \(A\).  Then, clearly \(B^p(K) = \mLa{\bigcap}_{A \in Z_p(K)} \Ker(\ep_A) \pc C^p(K)\) is a closed subspace and thus it also inherits a \F\ space topology (note that, as for the space \(\dd\Om^p(\M)\), this topology makes \(B^p(K)\) isomorphic to the quotient \(\rqt{C^p(K)}{Z^p(K)}\)).  In particular, the space \(\sch{p}{K} = \rqt{Z^p(K)}{B^p(K)}\) is a \F\ space.

Now, let \(\M\) be an \(n\)-dimensional manifold and choose a triangulation \((K,f)\) of \(\M\), i.e.\ choose an \(n\)-dimensional simplicial complex \(K\) together with a homeomorphism \(f: |K| \to \M\) such that, for each simplex \(\si \in K\), the restriction \(f|_\si:\si \to \M\) is a smooth embedding.  In the case where \(\del\M \ne \es\), assume moreover that \(f^{-1}(\del\M) = |L|\) for some subcomplex \(L \le K\) (see \cite[Defn.\ 7.1]{EDT} for the definition of a simplicial complex and a subcomplex).  Recall that there is a natural linear map \(\int: \Om^p(\M) \to C^p(K)\) defined by:
\ew
\mbox{\(\int(\al)(\si)\)} = \bigintsss_\si f^*\al
\eew
for each \(p\)-simplex \(\si\) of \(K\) and \(\al \in \Om^p(\M)\).  One may verify that \(\int\) is continuous \wrt\ the \F\ topologies on \(\Om^p(\M)\) and \(C^p(K)\).  Moreover, by Stokes' theorem, \(\dd \circ \int = \int \circ \,\, \dd\) and thus \(\int\) descends to define continuous linear maps \(\int: Z^p(K) \to \Om^p_\cl(\M)\), \(\int: B^p(K) \to \dd\Om^{p-1}(\M)\) and \(\int: \sch{p}{K} \to \dR{p}(\M)\), the last of which is bijective by de Rham's Theorem and hence an isomorphism by the Open Mapping Theorem.

Suppose, initially, that \(\M\) is closed.  It is proven by Whitney in \cite[Ch.\ IV, \S\S25--28]{GIT} that there is a linear map \(\mc{S}: C^p(K) \to \Om^p(\M)\) such that:
\begin{itemize}
\item[(i)] \(\int \circ\ \mc{S} = \Id\) on \(C^p(K)\);

\item[(ii)] \(\mc{S} \circ \dd = \dd \circ \mc{S}\);

\item[(iii)] For each simplex \(\si \in K\), writing \(\de_\si\) for the cochain dual to \(\si\), the \(p\)-form \(\mc{S}(\de_\si)\) vanishes identically near \(\M \osr \mathrm{Star}(\si)\) (see \cite[Defn.\ 7.1]{EDT} for the definition of the star of a simplex);

\item[(iv)] Writing {\bf 1} for the 0-cochain which takes the value 1 on each vertex of \(K\), \(\mc{S}({\bf 1})\) is the constant function \(1 \in \Om^0(\M)\).
\end{itemize}
Moreover, in part by using this result, Whitney proves that for all \(\al \in \Om^p_\cl(\M)\) and \(a \in C^{p-1}(K)\) such that \(\int(\al) = \dd a\), there exists \(\be \in \Om^{p-1}(\M)\) such that \(\dd\be = \al\) and \(\int(\be) = a\).  Notably, it is not proven that the assignment \((\al,a) \mt \be\) can be taken to be linear, nor that either this assignment or \(\mc{S}\) are continuous (indeed, the topologies on the spaces \(C^p(K)\) and \(\Om^p(\M)\) etc.\ are not considered {\it loc.\ cit.}).  However, by modifying the proof in \cite{GIT} in an obvious (and straightforward) way, one may ensure that each step in the construction of the assignments \(\mc{S}\) and \((\al,a) \mt \be\) is linear and continuous, and hence that the assignments themselves are, also, linear and continuous.  Moreover, the proof is valid in the case where \(\M\) is non-compact, even though \cite{GIT} only considers the case where \(\M\) is closed.  Thus, one obtains the following result:

\begin{Lem}\label{Cochain-Lem}
For each \(0 \le p \le n\):

(1) There exists a continuous linear map \(\mc{S}: C^p(K) \to \Om^p(\M)\) satisfying points (i)--(iv) above.

(2) Consider the closed subspace \(F^p(\M,K) \cc \dd\Om^p(\M) \ds C^p(K)\) defined by:
\ew
F^p(\M,K) = \lt\{ (\al, a) \in \dd\Om^p(\M) \ds C^p(K) ~\m|~ \mbox{\(\int\)}(\al) = \dd a \rt\}.
\eew
Then, there is a continuous linear map \(\mc{P}: F^p(\M,K) \to \Om^p(\M)\) such that for all \((\al,a) \in F^p(\M,K)\):
\ew
\mbox{\(\int\)}(\mc{P}(\al,a)) = a \et \dd\mc{P}(\al,a) = \al.
\eew\\
\end{Lem}

\subsection{Proof of \tref{NCH}}

I now prove \tref{NCH}.  I begin with a definition.
\begin{Defn}\label{LF-FSpace-Defn}
Let \(F\) be a \F\ space and choose a grading \(\lt(\|-\|_k\rt)_k\) on \(F\).  I shall term \(F\) locally finite-dimensional if for every \(k \in \bb{N}\), the space \(F_k\) (equivalently, \(\rqt{F}{\ker\|-\|_k}\)) is finite-dimensional.
\end{Defn}
By \eref{equiv-grads}, this definition is independent of choice of grading.  Moreover, the following proposition is readily verified:
\begin{Prop}\label{Properties-LF-Prop}
Let \(F\) be a locally finite-dimensional \F\ space.  Then:
\begin{enumerate}
\item Every closed subspace of \(F\) is also locally finite-dimensional;

\item Every quotient of a locally finite-dimensional \F\ space by a closed subspace is locally finite-dimensional;

\item \(F\) is nuclear.
\end{enumerate}
\end{Prop}
\qed

\begin{Ex}\label{LF-FSpace-Ex}
Recall the space \(\upomega\) and let \(\lt(\|-\|_k\rt)_k\) be the grading on \(\upomega\) described in \exref{FSpace-Exs}(2).  Then, \(\rqt{\upomega}{\ker\|-\|_k} \cong \bb{R}^k\) for each \(k \ge 0\).  Thus, \(\upomega\) is locally finite-dimensional.  In particular, given a simplicial complex \(K\), the vector spaces \(B^p(K)\), \(Z^p(K)\) and \(C^p(K)\) are all locally finite-dimensional.
\end{Ex}

The motivation for \dref{LF-FSpace-Defn} derives from the fact that locally finite-dimensional \F\ spaces enjoy the following property:

\begin{Thm}\label{LF->Ext=0-Thm}
Let \(E\), be a locally finite-dimensional \F\ space.  Then, every short exact sequence of \F\ spaces:
\ew
\bcd
0 \ar[r] & E \ar[r, "\al"] & F \ar[r, "\be"] & G \ar[r] & 0
\ecd
\eew
splits.
\end{Thm}

\begin{proof}
Since \(E\) is nuclear (by \pref{Properties-LF-Prop}(3)), it suffices to prove that the map \(\mc{I}\) defined in \tref{nuclear-split} is surjective.  For each \(\io^{k+1}_k: E_{k+1} \to E_k\), since \(E_{k+1}\) and \(E_k\) are finite-dimensional and \(\io^{k+1}_k\) is surjective, one may choose a linear map \(\la^k_{k+1}:E_k \to E_{k+1}\) such that \(\io^{k+1}_k \circ \la^k_{k+1} = \Id_{E_k}\).  In general, for any \(l < k\) define \(\la^l_k = \la^{k-1}_k \circ ... \circ \la^l_{l+1}\).  Using the \(\la^l_k\), define a map:
\ew
\bcd[row sep = 0pt]
\mc{J}: \prod_{k \in \bb{N}} \mc{L}(G,E_k) \ar[r] & \prod_{k \in \bb{N}} \mc{L}(G,E_k)\\
(A_k)_k \ar[r, maps to] & \lt(\sum_{i = 0}^{k-1}\la^i_k \circ A_i\rt)_k
\ecd
\eew
where, as usual, \(\sum_{i=0}^{-1}\) is interpreted as the empty sum.  Then, one computes that:
\ew
\mc{I}\circ\mc{J}\lt[(A_k)_k\rt] &= \lt(\io^{k+1}_k \circ \sum_{i = 0}^k\la^i_{k+1} \circ A_i - \sum_{i = 0}^{k-1}\la^i_k \circ A_i\rt)_k\\
&= \lt(\sum_{i = 0}^k\la^i_k \circ A_i - \sum_{i = 0}^{k-1}\la^i_k \circ A_i\rt)_k\\
&= \lt(A_k\rt)_k
\eew
and thus, \(\mc{I}\) is surjective (having right inverse \(\mc{J}\)).

\end{proof}

Using \tref{LF->Ext=0-Thm}, I now prove \tref{NCH}.  I begin with the following lemma:

\begin{Lem}\label{Splitting-Lem}
Let \(\M\) be an \(n\)-manifold (not necessarily oriented and possibly non-compact or with boundary).  Then, for any \(0 \le p \le n\) the following sequences split in the category of \F\ spaces:
\begin{enumerate}
\item \(\bcd 0 \ar[r] & \dd\Om^{p-1}(\M) \ar[r, hook] & \Om^p_\cl(\M) \ar[r, "proj"] & \dR{p}(\M) \ar[r] & 0 \ecd\)\\

\item \( \bcd 0 \ar[r] & \Om^p_\cl(\M) \ar[r, hook] & \Om^p(\M) \ar[r, "\dd"] & \dd\Om^p(\M) \ar[r] & 0. \ecd\)
\end{enumerate}
(Here, as usual, \(\Om^{-1}(\M)\) is formally taken to be \(0\).)
\end{Lem}

\begin{proof}
(1) Choose a triangulation \((K,f)\) of \(\M\) and consider the corresponding diagram:
\ew
\bcd[row sep = 15mm, column sep = 15mm]
0 \ar[r] & \dd\Om^{p-1}(\M) \ar[r, hook] \ar[d, "\int"] & \Om^p_\cl(\M) \ar[r, "proj", shift left = 3pt] \ar[d, "\int", shift left = 3pt] & \dR{p}(\M) \ar[r] \ar[d, "\int"] \ar[l, "\io_p", dashed, shift left = 3pt] & 0\\
0 \ar[r] & B^p(K) \ar[r, hook] & Z^p(K) \ar[u, "\mc{S}", shift left = 3pt] \ar[r, "proj", shift left = 3pt] & \sch{p}{K} \ar[l, "\fr{p}", dashed, shift left = 3pt] \ar[r] & 0
\ecd
\eew

Since \(B^p(K)\) is locally finite-dimensional (see \exref{LF-FSpace-Ex}), the bottom sequence in this diagram splits and thus there is a map \(\fr{p}: \sch{p}{K} \to Z^p(K)\) such that \(proj \circ \fr{p} = \Id_{\sch{p}{K}}\).  Now, define \(\io_p: \dR{p}(\M) \to \Om^p_\cl(\M)\) by \(\io_p = \mc{S} \circ \fr{p} \circ \int\).  I claim that \(proj \circ \io_p = \Id_{\dR{p}(\M)}\).  Since \(\int: \dR{p}(\M) \to \sch{p}{K}\) is an isomorphism, it suffices to prove that \(\int \circ\, proj \circ \io_p = \int\).  A direct calculation then yields:
\ew
\mbox{\(\int\)} \circ\, proj \circ \io_p &= \mbox{\(\int\)} \circ\, proj \circ \mc{S} \circ \fr{p} \circ \mbox{\(\int\)}\\
&= proj \circ \mbox{\(\int\)} \circ \mc{S} \circ \fr{p} \circ \mbox{\(\int\)}\\
&= proj \circ \fr{p} \circ \mbox{\(\int\)} \hs{5mm} \mbox{(since \(\int \circ\, \mc{S} = \Id_{Z^p(K)}\), by \lref{Cochain-Lem})}\\
&= \mbox{\(\int\)} \hs{5mm} \text{(since \(proj \circ \fr{p} = \Id_{\sch{p}{K}}\))},
\eew
as required.\\

(2) Now, consider the diagram:
\ew
\bcd[row sep = 15mm, column sep = 15mm]
0 \ar[r] & \Om^p_\cl(\M) \ar[r, hook] \ar[d, "\int"] & \Om^p(\M) \ar[r, "\dd", shift left = 3pt] \ar[d, "\int"] & \dd\Om^p(\M) \ar[r] \ar[d, "\int"] \ar[l, "\de_{p+1}", dashed, shift left = 3pt] & 0\\
0 \ar[r] & Z^p(K) \ar[r, hook] & C^p(K) \ar[r, "\dd", shift left = 3pt] & B^{p+1}(K) \ar[l, "\fr{d}", dashed, shift left = 3pt] \ar[r] & 0
\ecd
\eew
and recall the space \(F^p(\M,K)\) defined in \lref{Cochain-Lem}.  Since \(Z^p(K)\) is locally finite-dimensional (see \exref{LF-FSpace-Ex}), the lower sequence in this diagram splits and thus there is a map \(\fr{d}: B^{p+1}(K) \to C^p(K)\) such that \(\dd \circ \fr{d} = \Id_{B^{p+1}(K)}\).  In particular, the image of the map \(\Id \ds \lt(\fr{d} \circ \int\rt): \dd\Om^p(\M) \to \dd\Om^p(\M) \ds C^p(K)\) lies in the space \(F^p(\M,K)\) and thus the composite \(\de_{p+1} = \mc{P} \circ \lt[ \Id \ds \lt(\fr{d} \circ \int\rt)\rt]: \dd\Om^p(\M) \to \Om^p(\M)\) is well-defined.  I claim that \(\dd \circ \de_{p+1} = \Id_{\dd\Om^p(\M)}\).  Indeed, for \(\al \in \dd\Om^p(\M)\):
\ew
\dd \circ \de_{p+1}(\al) &= \dd \circ \mc{P} \lt[ \al \ds \lt(\fr{d} \circ \mbox{\(\int\)}(\al)\rt)\rt]\\
&= \al,
\eew
by \lref{Cochain-Lem}.  This completes the proof.

\end{proof}

Using \lref{Splitting-Lem}, I now complete the proof of \tref{NCH}

\begin{proof}[Proof of \tref{NCH}]
Let \(\io_p\) and \(\de_p\) be as in \lref{Splitting-Lem} and define \(\io\) to be the composite:
\ew
\bcd \Ds_{p = 0}^n \dR{p}(\M) \ar[r, "\Ds_{p = 0}^n \io_p"] & \Ds_{p=0}^n \Om^p_\cl(\M) \ar[r, hook] & \Ds_{p=0}^n \Om^p(\M)\ecd.
\eew
Then, by \lref{Splitting-Lem}:
\ew
\Om^p(\M) &= \Om^p_\cl(\M) \ds \de_{p+1}\dd\Om^p(\M)\\
&= \io\dR{p}(\M) \ds \dd\Om^{p-1}(\M) \ds \de_{p+1}\dd\Om^p(\M).
\eew
Now, define \(\de\) on \(\Om^p(\M)\) by setting:
\ew
\de =
\begin{dcases*}
0 &on \(\io\dR{p}(\M) \ds \de_{p+1}\dd\Om^p(\M)\)\\
\de_p &on \(\dd\Om^{p-1}\).
\end{dcases*}
\eew
Then, \(\de\Om^p(\M) = \de_p\dd\Om^{p-1}(\M)\) and thus one obtains \eref{NCHD-eq}, as claimed.  By construction, \(\de\) clearly satisfies \(\de^2 = 0\).  The rest of the claim now follows at once.

\end{proof}

~\vs{5mm}

\noindent Laurence H.\ Mayther\\
University of Cambridge\\
United Kingdom\\
{\it lhm32@cam.ac.uk}

\end{document}

%% file: Festino_1pt0.tex
\usepackage{adjustbox,array,bigints,cancel,color,comment,extpfeil,mathdots,mathrsfs,mathtools,MnSymbol,multicol,scalerel,setspace,tcolorbox,tikz,tikz-cd,upgreek}
\usepackage[fontsize = 10pt]{fontsize}
\usepackage[left = 2.5cm, right = 2.5cm, top = 2.5cm, bottom = 2.5cm]{geometry}
\usepackage{hyperref}
\usetikzlibrary{patterns}
\numberwithin{equation}{section}

\theoremstyle{plain}
\newtheorem{Cl}[equation]{Claim}

\newtheorem{Cor}[equation]{Corollary}

\newtheorem{Lem}[equation]{Lemma}

\newtheorem{Prop}[equation]{Proposition}

\newtheorem{Thm}[equation]{Theorem}

\theoremstyle{definition}

\newtheorem{Defn}[equation]{Definition}
\newtheorem{Ex}[equation]{Example}
\newtheorem{Exs}[equation]{Examples}

\theoremstyle{remark}
\newtheorem{Rk}[equation]{Remark}

\theoremstyle{plain}
\newtheorem*{Cl*}{Claim}
\newtheorem*{Conj*}{Conjecture}
\newtheorem*{Lem*}{Lemma}
\newtheorem*{Prop*}{Proposition}
\newtheorem*{Q*}{Question}
\newtheorem*{Schol*}{Scholium}
\newtheorem*{SubCl*}{Subclaim}
\newtheorem*{Thm*}{Theorem}

\theoremstyle{definition}
\newtheorem*{Cond*}{Condition}
\newtheorem*{Cstr*}{Construction}
\newtheorem*{Defn*}{Definition}
\newtheorem*{Ex*}{Example}
\newtheorem*{Exs*}{Examples}
\newtheorem*{Md*}{Method}
\newtheorem*{Nt*}{Notation}
\newtheorem*{Pty*}{Property}

\theoremstyle{remark}

\newtheorem*{Rk*}{Remark}
\newtheorem*{Rks*}{Remarks}
\newtheorem*{A-d}{Aside}

\newcommand{\cag}{\begin{equation}\begin{gathered}}
\newcommand{\caag}{\end{gathered}\end{equation}}
\newcommand{\caw}{\begin{equation*}\begin{gathered}}
\newcommand{\caaw}{\end{gathered}\end{equation*}}
\newcommand{\e}{\begin{equation}\begin{aligned}}
\newcommand{\ee}{\end{aligned}\end{equation}}
\newcommand{\ew}{\begin{equation*}\begin{aligned}}
\newcommand{\eew}{\end{aligned}\end{equation*}}

\newcommand{\bcd}{\begin{tikzcd}}
\newcommand{\ecd}{\end{tikzcd}}
\newcommand{\bma}{\begin{matrix}}
\newcommand{\ema}{\end{matrix}}
\newcommand{\bpm}{\begin{pmatrix}}
\newcommand{\epm}{\end{pmatrix}}
\newcommand{\bvm}{\begin{vmatrix}}
\newcommand{\evm}{\end{vmatrix}}

\newcommand{\nts}{\begin{tcolorbox}}
\newcommand{\ntss}{\end{tcolorbox}}

\newcommand{\aref}[1]{Appendix \ref{#1}}

\newcommand{\cref}[1]{Corollary \ref{#1}}

\newcommand{\dref}[1]{Definition \ref{#1}}

\newcommand{\eref}[1]{eqn.\hspace{0.6mm}(\ref{#1})}
\newcommand{\erefs}[1]{eqns.\hspace{0.6mm}(\ref{#1})}
\newcommand{\exref}[1]{Example \ref{#1}}
\newcommand{\exrefs}[1]{Examples \ref{#1}}

\newcommand{\lref}[1]{Lemma \ref{#1}}

\newcommand{\pref}[1]{Proposition \ref{#1}}

\newcommand{\rref}[1]{Remark \ref{#1}}

\newcommand{\sref}[1]{\S\ref{#1}}
\newcommand{\srefs}[1]{\S\S\ref{#1}}
\newcommand{\tref}[1]{Theorem \ref{#1}}
\newcommand{\trefs}[1]{Theorems \ref{#1}}

\DeclareMathSizes{10}{10}{8}{7}

\newcommand{\mss}[1]{\mbox{\scriptsize \(#1\)}}

\newcommand{\mns}[1]{\mbox{\normalsize \(#1\)}}

\newcommand{\mLa}[1]{\mbox{\Large \(#1\)}}

\newcommand{\bb}[1]{\mathbb{#1}}
\newcommand{\cal}[1]{\mathscr{#1}}
\newcommand{\fr}[1]{\mathfrak{#1}}
\newcommand{\mb}[1]{\mbox{\boldmath \(#1\)}}
\newcommand{\mc}[1]{\mathcal{#1}}

\newcommand{\as}{\hspace{5mm}\text{as}\hspace{5mm}}
\newcommand{\et}{\hspace{5mm}\text{and}\hspace{5mm}}
\newcommand{\hs}[1]{\hspace{#1}}

\newcommand{\vs}[1]{\vspace{#1}}

\DeclareMathSymbol{\Alpha}{\mathalpha}{operators}{"41}
\DeclareMathSymbol{\Beta}{\mathalpha}{operators}{"42}
\DeclareMathSymbol{\Epsilon}{\mathalpha}{operators}{"45}
\DeclareMathSymbol{\Zeta}{\mathalpha}{operators}{"5A}
\DeclareMathSymbol{\Eta}{\mathalpha}{operators}{"48}
\DeclareMathSymbol{\Iota}{\mathalpha}{operators}{"49}
\DeclareMathSymbol{\Kappa}{\mathalpha}{operators}{"4B}
\DeclareMathSymbol{\Mu}{\mathalpha}{operators}{"4D}
\DeclareMathSymbol{\Nu}{\mathalpha}{operators}{"4E}
\DeclareMathSymbol{\Omicron}{\mathalpha}{operators}{"4F}
\DeclareMathSymbol{\Rho}{\mathalpha}{operators}{"50}
\DeclareMathSymbol{\Tau}{\mathalpha}{operators}{"54}
\DeclareMathSymbol{\Chi}{\mathalpha}{operators}{"58}
\DeclareMathSymbol{\omicron}{\mathord}{letters}{"6F}

\newcommand{\al}{\alpha}
\newcommand{\be}{\beta}
\newcommand{\ga}{\gamma}
\newcommand{\de}{\delta}
\newcommand{\ep}{\varepsilon}

\newcommand{\ze}{\zeta}
\renewcommand{\th}{\theta}

\newcommand{\io}{\iota}
\newcommand{\ka}{\kappa}
\newcommand{\la}{\lambda}
\newcommand{\omi}{\omicron} %\omi is used, to avoid confusion \om (shorthand for \omega).
\newcommand{\vpi}{\varpi}
\newcommand{\rh}{\rho}

\newcommand{\si}{\sigma}

\newcommand{\ta}{\tau}
\newcommand{\up}{\upsilon}
\newcommand{\ph}{\phi}
\newcommand{\vph}{\upvarphi}
\newcommand{\vps}{\uppsi}
\newcommand{\ch}{\chi}
\newcommand{\ps}{\psi}
\newcommand{\om}{\omega}

\newcommand{\Ga}{\Gamma}
\newcommand{\De}{\Delta}

\newcommand{\La}{\Lambda}
\newcommand{\Ph}{\Phi}

\newcommand{\Ps}{\Psi}
\newcommand{\Om}{\Omega}

\newcommand{\<}{\langle}
\newcommand{\?}{\rangle}

\newcommand{\Ann}{\operatorname{Ann}}

\newcommand{\Ds}{\bigoplus}
\newcommand{\ds}{\oplus}
\newcommand{\End}{\operatorname{End}}

\newcommand{\Hom}{\operatorname{Hom}}
\newcommand{\Id}{\operatorname{Id}}
\newcommand{\Ker}{\operatorname{Ker}}

\newcommand{\lqt}[2]{\left.\raisebox{-1mm}{\(#2\)}\middle\backslash\raisebox{1mm}{\(#1\)}\right.}

\newcommand{\rqt}[2]{\left.\raisebox{1mm}{\(#1\)}\middle/\raisebox{-1mm}{\(#2\)}\right.}
\newcommand{\ts}{\otimes}

\newcommand{\Tr}{\operatorname{Tr}}

\newcommand{\x}{\times}

\newcommand{\del}{\partial}

\newcommand{\Hocl}[1]{\overset{\circ}{H^{#1}}_{\kern-1.9mm\cl}}

\newcommand{\lop}{\left\|\kern-1.30mm\left\|}
\newcommand{\op}{\|\kern-1.30mm\|}
\newcommand{\rop}{\right\|\kern-1.30mm\right\|}
\newcommand{\SI}{\operatorname{\cal{I}\kern-1.5pt nd}}

\newcommand{\cc}{\subseteq}

\newcommand{\es}{\emptyset}

\newcommand{\mf}{\leftmapsto}

\newcommand{\mt}{\mapsto}

\newcommand{\osr}{\backslash}
\newcommand{\oto}[1]{\xrightarrow{#1}}
\newcommand{\pc}{\subset}

\newcommand{\B}{\mathrm{B}}

\newcommand{\CL}{\mathcal{C}l}
\newcommand{\cl}{\mathrm{closed}}

\newcommand{\Conv}{\mathrm{Conv}}

\newcommand{\dd}{\mathrm{d}}

\newcommand{\Diff}{\operatorname{Diff}}
\newcommand{\dR}[1]{H^{#1}_{\operatorname{dR}}}

\newcommand{\emb}{\hookrightarrow}

\newcommand{\Gr}{\mathrm{Gr}}

\newcommand{\hk}{\righthalfcup}

\newcommand{\Hs}{\raisebox{1pt}{\mss{\bigstar}}}

\newcommand{\ls}{\left\lsem}

\newcommand{\M}{\mathrm{M}}

\newcommand{\N}{\mathrm{N}}
\newcommand{\oGr}{\widetilde{\mathrm{\Gr}}}

\newcommand{\Op}{\mathcal{O}p}

\newcommand{\rs}{\right\rsem}
\newcommand{\s}{\odot}

\newcommand{\sch}[2]{\operatorname{H}^{#1}\left(#2\right)}

\newcommand{\sph}{\widetilde{\phi}}

\renewcommand{\ss}[2][{}]{\bigodot{\hspace{-1mm}}^{#2}_{#1}\hspace{0.6mm}}

\newcommand{\svph}{\widetilde{\upvarphi}}
\newcommand{\svps}{\widetilde{\uppsi}}
\newcommand{\T}{\mathrm{T}}
\newcommand{\tl}{{\mns{\sim}}}

\newcommand{\w}{\wedge}
\newcommand{\ww}[2][{}]{\bigwedge{\hspace{-1mm}}^{#2}_{\hspace{1mm}#1}\hspace{0.1mm}}

\newcommand{\0}{\infty}
\newcommand{\1}{\cdot}

\renewcommand{\ge}{\geqslant}
\newcommand{\gl}{\hspace{0.4mm}\raisebox{0.8mm}{\(>\)}\kern-1.8mm\raisebox{-0.8mm}{\(<\)}\hspace{0.4mm}}
\newcommand{\gle}{\hspace{0.4mm}\raisebox{1.2mm}{\(\ge\)}\kern-1.8mm\raisebox{-1.2mm}{\(\le\)}\hspace{0.4mm}}
\renewcommand{\le}{\leqslant}
\newcommand{\pt}{\bullet}
\newcommand{\sgn}{\operatorname{sign}}

\newcommand{\g}{\(\mathrm{G}_2\)}
\newcommand{\Gg}{\mathrm{G}}
\newcommand{\GL}{\operatorname{GL}}

\newcommand{\sg}{\(\widetilde{\mathrm{G}}_2\)}
\newcommand{\SL}{\operatorname{SL}}
\newcommand{\slc}{\(\operatorname{SL}(3;\mathbb{C})\)}
\newcommand{\slr}{\(\operatorname{SL}(3;\mathbb{R})^2\)}

\newcommand{\Sp}{\operatorname{Sp}}

\newcommand{\Stab}{\operatorname{Stab}}
\newcommand{\SU}{\operatorname{SU}}

\renewcommand{\iff}{if and only if}

\newcommand{\stp}{suffices to prove}
\newcommand{\Wlg}{Without loss of generality}
\newcommand{\wlg}{without loss of generality}

\newcommand{\wrt}{with respect to}

\newcommand{\F}{Fr\'{e}chet}

\newcommand{\ol}[1]{\overline{#1}}
\newcommand{\h}{\widehat}
\newcommand{\lt}{\left}
\newcommand{\m}{\middle}

\newcommand{\rt}{\right}

\newcommand{\tld}{\widetilde}

%% file: Relative_h-principles_for_closed_stable_forms.bbl
\begin{thebibliography}{10}
\bibitem{MwEH} Bryant, R.L., `Metrics with exceptional holonomy', {\it Ann. Math.} {\bf 126} (1987), no. 3, 525--576.

\bibitem{THKWHLoNM} Bueler, E.L., `The heat kernel weighted Hodge Laplacian on noncompact manifolds', {\it Trans. Am. Math. Soc.} {\bf 351} (1999), no. 2, 683--713.

\bibitem{HFS&SH} Cort\'{e}s, V., Leistner, T., Sch\"{a}fer, L. and Schulte-Hengesbach, F., `Half-flat structures and special holonomy', {\it Proc. London Math. Soc.} {\bf 102} (2011), no. 3, 113--158.

\bibitem{NIoG2S} Crowley, D. and Nordstr\"{o}m, J., `New invariants of \g-structures', {\it Geom. Topol.} {\bf 19} (2015), no. 5, 2949--2992.

\bibitem{CoToa8DRVS} Djokovi\'{c}, D.Ž., `Classification of trivectors of an eight-dimensional real vector space', {\it Linear Multilinear Algebra} {\bf 13} (1983), no. 1, 3--39.

\bibitem{RoG2MwB} Donaldson, S.K., `Remarks on \g-manifolds with boundary', in {\it Celebrating the \(50^\text{th}\) Anniversary of the Journal of Differential Geometry: Lectures given at the Geometry and Topology Conference at Harvard University in 2017}, ed. by H-D. Cao, J. Li, R.M. Schoen and S-T. Yau, Surveys in Differential Geometry, {\bf XXII} (International Press of Boston, Somerville (MA), 2018), 103--124.

\bibitem{ItthP} Eliashberg, Y. and Mishachev, N., {\it Introduction to the \(h\)-Principle}, Graduate Studies in Mathematics, {\bf 48} (American Mathematical Society, Providence (RI), 2002).

\bibitem{CIoPDR} Gromov, M.L., `Convex integration of differential relations. I', trans. by N. Stein, {\it Math. USSR Inv.} {\bf 7} (1973), no. 2, 329--343.

\bibitem{PDR} Gromov, M.L., {\it Partial Differential Relations}, Ergebnisse der Mathematik und ihrer Grenzgebiete, 3. Folge, {\bf 9} (Springer-Verlag, Berlin, 1987).

\bibitem{A3F&ESLGoTG2} Herz, C., `Alternating 3-forms and exceptional simple Lie groups of type \g', {\it Can. J. Math.} {\bf XXXV} (1983), no. 5, 776--806.

\bibitem{IoM} Hirsch, M.W., `Immersions of manifolds', {\it Trans. Am. Math. Soc.} {\bf 93} (1959), no. 2, 242--276.

\bibitem{TGo3Fi6&7D} Hitchin, N.J., `The geometry of three-forms in six and seven dimensions', arXiv:math/0010054 [math.DG] (2000); see also `The geometry of three-forms in six and seven dimensions', {\it J. Differ. Geom.} {\bf 56} (2000), no. 3, 547--576.

\bibitem{SF&SM} Hitchin, N.J., `Stable forms and special metrics', in {\it Global Differential Geometry: The Mathematical Legacy of Alfred Gray}, ed. M. Fern\'{a}ndez and J.A. Wolf, Contemporary Mathematics, {\bf 288} (American Mathematical Society, Providence (RI), 2001), 70--89.

\bibitem{SG2SoPRM} Kath, I., `\(\Gg^*_{2(2)}\)-structures on pseudo-Riemannian manifolds', {\it J. Geom. Phys.} {\bf 27} (1998), 155--177.

\bibitem{TVSI} K\"{o}the, G., {\it Topological Vector Spaces I}, trans. by D.J.H. Garling, revised second printing, Grundlehren der mathematischen Wissenschaften, {\bf 159} (Springer-Verlag, Berlin, 1983).

\bibitem{MASF} L\^{e}, H-V., Pan\'{a}k, M. and Van\v{z}ura, J., `Manifolds admitting stable forms', {\it Comment. Math. Univ. Carolin.} {\bf 49} (2008), no. 1, 101--117.

\bibitem{UA&BotDHFoG2&SG2F} Mayther, L.H., `Unboundedness above and below of the Donaldson--Hitchin functionals on \g- and \sg-forms', \href{https://arxiv.org/abs/2308.15438}{arXiv:2308.15438 [math.DG]} (2023).

\bibitem{AoCItS&CG} McDuff, D., `Application of convex integration to symplectic and contact geometry', {\it Ann. Inst. Fourier (Grenoble)} {\bf 37} (1987), no. 1, 107--133.

\bibitem{EDT} Munkres, J.R., {\it Elementary Differential Topology}, Annals of Mathematics Studies, {\bf 54} (Princeton University Press, Princeton (NJ), 1963).

\bibitem{DOoVB} Palais, R.S., `Differential operators on vector bundles', in {\it Seminar on the Atiyah--Singer Index Theorem}, ed. by R.S. Palais, Annals of Mathematics Studies, {\bf 57} (Princeton University Press, Princeton (NJ), 1965), 51--93.

\bibitem{HToIDM} Palais, R.S., `Homotopy theory of infinite dimensional manifolds', {\it Topology} {\bf 5} (1966), no. 1, 1--16.

\bibitem{RG&HG} Salamon, S.M., {\it Riemannian Geometry and Holonomy Groups}, Pitman Research Notes in Mathematics Series, {\bf 201} (Longman Scientific and Technical, Essex, 1989).

\bibitem{HD-AMfSBVP} Schwarz, G., {\it Hodge Decomposition - A Method for Solving Boundary Value Problems}, Lecture Notes in Mathematics, {\bf 1607} (Springer-Verlag, Berlin, 1995).

\bibitem{CIT} Spring, D., {\it Convex Integration Theory: Solutions to the \(h\)-Principle in Geometry and Topology},  Monographs in Mathematics, {\bf 92} (Birkh\"{a}user Verlag, Switzerland, 1998).

\bibitem{SRoCLMBFS} Vogt, D., `Some results on continuous linear maps between \F\ spaces', in {\it Functional Analysis: Surveys and Recent Results III}, ed. by K-D. Bierstedt and B. Fuchssteiner, North-Holland Mathematics Studies, {\bf 90} (Elsevier Science Publishers B.V., Amsterdam, 1984), 349--381.

\bibitem{FoDM&LG} Warner, F.W., {\it Foundations of Differentiable Manifolds and Lie Groups}, (Scott, Foresman \& Company, 1971; repr. as Graduate Texts in Mathematics, {\bf 94}, Springer-Verlag, Berlin, 1983).

\bibitem{GIT} Whitney, H., {\it Geometric Integration Theory}, Princeton Mathematical Series, {\bf 21} (Princeton University Press, Princeton (NJ), 1957).

\bibitem{FA(KY)} Yosida, K., {\it Functional Analysis}, \(5^\text{th}\) edn., Grundlehren der mathematischen Wissenschaften, {\bf 123} (Springer-Verlag, Berlin, 1978).
\end{thebibliography}
